\documentclass[12pt]{amsart}
\makeatletter
\newcommand*{\rom}[1]{\expandafter\@slowromancap\romannumeral #1@}
\makeatother
\usepackage{etex}
\usepackage{amsmath,amsthm,amsfonts,amssymb,tikz-cd,mathdesign,float}
\usepackage{lmodern}
\usepackage{relsize}
\usepackage{euler}
\usepackage{times}
\usetikzlibrary{positioning,automata}
\usepackage{mathtools}
\usepackage{youngtab}
\usepackage{pgfplots}
\usepackage{pst-plot}
\usepackage{pgfplots}
\usepackage[alphabetic]{amsrefs}
\pgfplotsset{compat=1.10}
\usepgfplotslibrary{fillbetween}
\usetikzlibrary{patterns}
\usetikzlibrary{matrix}
\numberwithin{equation}{section}

\newtheorem{thm}{Theorem}[section]
\newtheorem{prop}[thm]{Proposition}
\newtheorem{lem}[thm]{Lemma}
\newtheorem{cor}[thm]{Corollary}
\usepackage{xcolor}
\theoremstyle{definition}
\newtheorem{defn}[thm]{Definition}
\newtheorem{rmk}[thm]{Remark}
\newtheorem{ex}[thm]{Example}
\DeclareMathOperator{\Hom}{Hom}
\usepackage[pdftex,colorlinks=true,allcolors=blue]{hyperref}
\address{Department of Mathematics, Northeastern University,
    Boston, MA, 02115, USA}
\email{tsvelikhovskiy.b@husky.neu.edu}
\DeclareMathAlphabet\mathbfcal{OMS}{cmsy}{b}{n} 
\begin{document}
\title{On Categories $\mathcal{O}$ of quiver varieties overlying the bouquet graphs}
\author{Boris Tsvelikhovsky}
\maketitle
\begin{abstract}
We study representation theory of quantizations of  Nakajima quiver varieties associated to bouquet quivers. We show that there are no finite dimensional representations of the quantizations $\overline{\mathcal{A}}_{\lambda}(n, \ell)$ if dim $V=n$ is greater than $1$ and so is the number of loops $\ell$. We find that there is a Hamiltonian torus action with finitely many fixed points in case  $n\leq 3$, provide the dimensions of Hom-spaces between standard objects in category $\mathcal{O}$ and compute the multiplicities of simples in standards for $n=2$ in case of one-dimensional framing and  generic one-parameter subgroups. 
We establish the abelian localisation theorem and find the values of parameters, for which
the quantizations have infinite homological dimension. 
\end{abstract}
\date{ }
{\hypersetup{linkcolor=black}
\tableofcontents
}
\pgfkeys{/pgfplots/scale/.style={
  x post scale=#1,
  y post scale=#1,
  z post scale=#1}
}
\pgfkeys{/pgfplots/axis labels at tip/.style={
    xlabel style={at={(current axis.right of origin)}, xshift=1.5ex, anchor=center},
    ylabel style={at={(current axis.above origin)}, yshift=1.5ex, anchor=center}}
}
\section{Introduction}
\par 
Our primary goal is to study category $\mathcal{O}$ of quantizations of the Nakajima quiver variety with underlying quiver $Q=B_{\ell}$, which has one vertex, $\ell$ loops, where $\ell \in \mathbb{Z}_{\geq 0}$  and a one-dimensional framing. The  notion of category $\mathcal{O}$ in the context of conical symplectic resolutions  was introduced  in  \cite{BLPW}.  In particular in  \cite{Los17_1} the author studies the properties of category $\mathcal{O}$ for the Gieseker varieties. These are the framed moduli spaces of torsion free sheaves on $\mathbb{P}^2$ with rank $r$ and second chern class $n$. They admit a description as quiver varieties for the quiver with one vertex, one loop, $n$-dimensional space assigned to the vertex and an $r$-dimensional framing (see chapter $2$ of \cite{Nak1} for details). The results and methods of  \cite{Los17_1} provide invaluable tools for our research. We briefly recall the setup. 
\subsection{Generalities on category $\mathcal{O}$ for conical symplectic resolutions}
We fix the base field to be $\mathbb{C}$. Let $X_0$ be a normal Poisson affine variety equipped with  an action of the multiplicative group $\mathbb{S}:=\mathbb{C}^{*}$, s.t. the Poisson bracket has a negative degree with respect to this action. We assume that $\mathbb{C}[X_0]=\underset{i \geq 0}{\bigoplus }\mathbb{C}[X_0]_{i}$ with $\mathbb{C}[X]_{0}=\mathbb{C}$ w.r.t. the grading coming from the $\mathbb{S}$-action (this action will be called the \textit{contracting action}).  Geometrically this means that there is a unique fixed point $o\in X_0$ and the entire variety is contracted to this point by the $\mathbb{S}$-action. Let $(X,\omega)$ be a symplectic variety and $\rho:X\rightarrow X_0$ a projective resolution of singularities, which is also a morphism of Poisson varieties. In addition, assume that the action of $\mathbb{S}$ admits a $\rho$-equivariant lift to $X$.  A pair $(X,\rho)$ as above is called a \textit{conical symplectic resolution}.
\begin{defn}
Let $(X,\rho)$ be a conical symplectic resolution.  A \textit{quantization} of the affine variety $X_0$ is an algebra $\mathcal{A}$ together with an isomorphism $gr\mathcal{A}\xrightarrow{\sim}\mathbb{C}[X]$ of graded Poisson algebras. By a quantization of $X$ we understand a sheaf (in the conical topology, i.e. open spaces are Zariski open and $\mathbb{S}$-stable) of filtered algebras $\tilde{\mathcal{A}}$ (the  filtration is complete and separated) together with an isomorphism $gr\mathcal{A}\xrightarrow{\sim}\mathcal{O}_{X}$ of sheaves of graded Poisson algebras.
\end{defn}
\par
   Suppose,  that $X$ is equipped with a Hamiltonian action of a torus $T$ with finitely many  fixed points, i.e. $|X^T| < \infty$. Assume, in addition, that the action of $T$ commutes with  the contracting action of $\mathbb{S}$. A one-parametric subgroup $\nu: \mathbb{C}^{*}\rightarrow T$ is called \textit{generic} if $X^T=X^{\nu(\mathbb{C}^{*})}$. To a  generic one-parametric subgroup $\nu: \mathbb{C}^{*}\rightarrow T$ one can associate a category of modules over the algebra $\mathcal{A}$ defined above, called category $\mathcal{O}_{\nu}(\mathcal{A})$. Namely, the action of $\nu$ lifts to  $\mathcal{A}$ and induces a grading on it, i.e. $\mathcal{A}=\underset{i \in \mathbb{Z}}{\bigoplus }\mathcal{A}_{\lambda}(n,\ell)_{i,\nu}$. We denote 
   \begin{align}
   &\mathcal{A}^{\geq 0,\nu}=\underset{i \geq 0}{\bigoplus }\mathcal{A}_{i,\nu}, \mathcal{A}^{\leq 0,\nu}=\underset{i \leq 0}{\bigoplus }\mathcal{A}_{i,\nu} ~(\mbox{similarly define } \mathcal{A}^{< 0,\nu},\mathcal{A}^{> 0,\nu}) \mbox{ and } \\ &\mathcal{C}_{\nu}(\mathcal{A}):=\mathcal{A}^{\geq 0,\nu}/\left(\mathcal{A}^{\geq 0,\nu}\cap\mathcal{A}\mathcal{A}^{>0,\nu}\right)=\mathcal{A}_{0}/\underset{i > 0}{\bigoplus }\mathcal{A}_{-i}\mathcal{A}_{i}. 
   \end{align}
   \par
 Let $\mathcal{A}\operatorname{-mod}$ be the category of finitely generated $\mathcal{A}$-modules.
 \begin{defn} 
 	\label{CatO} 
 The category $\mathcal{O}_{\nu}(\mathcal{A})$ is the full subcategory of $\mathcal{A}\operatorname{-mod}$, on which $\mathcal{A}^{\geq0,\nu}$ acts locally finitely. 
 \end{defn}


\begin{ex}
Let $\mathfrak{g}$ be a simple Lie algebra with Borel subalgebra $\mathfrak{b}$ and Cartan subalgebra $\mathfrak{h}$. In order to fit the classical BGG category $\mathcal{O}$ in this framework, one needs to consider the Springer resolution  $X=T^*(G/B)\rightarrow \mathcal{N}=X_0$ of the nilpotent cone $\mathcal{N}\subset \mathfrak{g}$. The tori are the maximal torus $T\subset GL(V)$ and $\mathbb{S}:=\mathbb{C}^{*}$ acting by inverse scaling on the cotangent fibers. Let $\mu:Z(\mathfrak{g})\rightarrow \mathbb{C}$ be a central character, then the block $\mathcal{O_{\mu}}\subset\mathcal{O}$ consists of finitely generated $U(\mathfrak{g})$-modules for which $U(\mathfrak{b})$ acts locally finitely, $U(\mathfrak{h})$ semisimply and the center with generalized character $\mu$. Pick a generic one-parameter subgroup $\nu(\mathbb{C}^{*})\subset T$, s.t. $\mathfrak{b}$ is spanned by elements with positive $\nu(\mathbb{C}^{*})$-weights. Let $U(\mathfrak{g})_{\mu}=U(\mathfrak{g})/\mathcal{I}_{\mu}$ with $\mathcal{I}_{\mu}$ the ideal generated by $z-\mu(z)$ for $z\in Z(\mathfrak{g})$ be the central reduction of $U(\mathfrak{g})$ w.r.t the central character $\mu$. Then $U(\mathfrak{g})_{\mu}$ is known to be a quantization of the nilcone $\mathcal{N}$. We want to compare the category $\mathcal{O}_{\nu}(U(\mathfrak{g})_{\mu})$ with $\mathcal{O_{\mu}}$.  The difference in the requirements for an object $M\in U(\mathfrak{g})_{\mu}\operatorname{-mod}$ to be in $\mathcal{O}_{\nu}(U(\mathfrak{g})_{\mu})$ or $\mathcal{O_{\mu}}$ is that for the former containment $Z(\mathfrak{g})$ must act on $M$  with an honest character $\mu$, while for the latter the action of $U(\mathfrak{h})$ on $M$ has to be semisimple. In case $\mu$ is regular these conditions are interchangable, i.e. one gets an equivalent category by dropping one condition and adding the other (see Theorem $1$ in \cite{Soerg}), and, hence, the categories  $\mathcal{O}_{\nu}(U(\mathfrak{g})_{\mu})$ and $\mathcal{O_{\mu}}$ are equivalent. 
\end{ex}

\subsection{Category $\mathcal{O}$ for the quantizations of quiver varieties with $Q=B_{\ell}$}
We study the Nakajima quiver variety with underlying quiver $Q$, which has one vertex, $\ell$ loops, where $\ell \in \mathbb{Z}_{\geq 0}$  and a one-dimensional framing. This variety admits the following description. One starts with a vector space $V$ of dimension $n$ and considers the space $R:=\mathfrak{gl}(V)^{\oplus \ell}\oplus V^*$, which has a natural $G:=GL(V)$ action. The identification of $\mathfrak{g}:=\mathfrak{gl}(V)$ with $\mathfrak{g}^*$ via the trace form enables to identify the cotangent bundle $T^*R$ with  $\mathfrak{gl}(V)^{\oplus 2\ell}\oplus V^* \oplus V$. Next notice that $T^*R$ is a symplectic vector space with a Hamiltonian action of  $G$. The corresponding moment map is given by
\begin{equation}
 \label{eq:(moment)}
\mu(X_{1}, \hdots, X_{\ell},Y_{1}, \hdots, Y_{\ell},i,j) = \sum\limits_{k=0}^{\ell}[X_{k},Y_{k}]-ji.
 \end{equation} 
\par
 \begin{defn}The \textit{affine quiver variety} $\mathcal{M}(n,\ell)$  is the categorial quotient $\mu^{-1}(0)//G:=\mbox{Spec }\mathbb{C}[\mu^{-1}(0)]^{G}$. 

 \end{defn}
\par
To define the Nakajima quiver variety $\mathcal{M}^{\theta}(n,\ell)$, we need to choose some character $\theta$ of $G$. It is known that $\theta$ an integral power of the determinant, i.e. $\theta=\mbox{det}^k$ for some $k \in \mathbb{Z}$.

 \begin{defn}
 The GIT quotient $\mathcal{M}^{\theta}(n,\ell):=\mu^{-1}(0)^{\theta-ss}//^{\theta}{G}$ is called the \textit{Nakajima quiver variety} with parameter $\theta$. 
  \end{defn}
 \par
 The torus $T=(\mathbb{C}^{*})^{\ell}$ acts on $R$ by rescaling  $X_{1}, \hdots, X_{\ell}$. This naturally gives rise to an action on $T^*R$. This action is Hamiltonian and commutes with the action of $G$ and, therefore, descends to $\mathcal{M}(n,\ell)$ and $\mathcal{M}^{\theta}(n,\ell)$. The action of $s\in \mathbb{S}$ is given
by multiplication of all the components of $x \in T^*R$ by $s^{-1}$. Similarly, it commutes with the action of $G$ and descends to $\mathcal{M}(n,\ell)$ and $\mathcal{M}^{\theta}(n,\ell)$.
   \par
  For any $\theta\neq 0$ the action of $G$ on $\mu^{-1}(0)^{\theta-ss}$ is free. This implies that the variety $\mathcal{M}^{\theta}(n,\ell)$ is smooth and symplectic and  is known to be a symplectic resolution of the normal Poisson variety $\mathcal{M}(n,\ell)$. We denote by $\rho$ the corresponding map $\rho:\mathcal{M}^{\theta}(n,\ell)\rightarrow \mathcal{M}(n,\ell)$. It is a conical symplectic resolution.

   \par
   Set $\overline{R}=\mathfrak{sl}(V)^{\oplus \ell}\oplus V^{*}$ and let $\overline{\mathcal{M}}(n,\ell)$ be the affine variety $\mu^{-1}(0)//G$, where slightly abusing notation, we denote by $\mu$ the moment map for the Hamiltonian action of $G$ on $T^*\overline{R}$. Similarly, we set  $\overline{\mathcal{M}}^{\theta}(n,\ell):=\mu^{-1}(0)^{\theta-ss}//^{\theta}{G}$. Next we describe quantizations of $\overline{\mathcal{M}}(n,\ell)$. Denote the ring of differential operators on $\overline{R}$ by $D(\overline{R})$. 
 \begin{defn}
 \label{qComomentMap} A $G$-equivariant linear map $\Phi: \mathfrak{g}\rightarrow D(\overline{R})$, satisfying $[\Phi(x), a]=x_{\overline{R}}(a)$ for any $x \in \mathfrak{g}$ and $a \in D(\overline{R})$ is called a \textit{quantum comoment map}. 
  \begin{rmk}
 The quantum comoment map $\Phi$ is defined up to adding a character $\lambda: \mathfrak{g}\rightarrow \mathbb{C}$. 
 \end{rmk} 
 \par
 Notice that we can identify $D(\overline{R})$ with $D(\overline{R}^*)$ via the Fourier transform sending $\partial_r \in D(\overline{R})$ to the function $r \in D(\overline{R}^*)$ and $r^* \in D(\overline{R})$ to $-\partial_{r^*} \in D(\overline{R}^*)$. Thus defined isomorphism $D(\overline{R})\rightarrow D(\overline{R}^*)$  allows to consider two quantum comoment maps $\Phi,\widetilde{\Phi}:\mathfrak{gl}(V)\rightarrow D(\overline{R})$ sending $x \in \mathfrak{g}$ to the corresponding vector field $x_{\overline{R}}$ or $x_{\overline{R}^*}$. Now  define the \textit{symmetrized quantum comoment map} to be $\Phi^{sym}:=\frac{\Phi+\widetilde{\Phi}}{2}$.
 A direct computation shows that $\Phi^{sym}(x)=\Phi(x)-\zeta(x)$, where $\zeta$ is half the character of the action of $G$ on $\Lambda^{top}R$. For our quiver $Q$ with one-dimensional framing $\zeta(x)=\frac{1}{2}tr(x)$.
 \end{defn}   
 
Next we take a character $\lambda$ of $\mathfrak{g}$ and consider the quantizations $$\mathcal{\overline{A}}_{\lambda}(n,\ell):= (D(\overline{R})/[D(\overline{R})\{\Phi(x)-\lambda(x), x \in \mathfrak{g}\}])^{G},$$ 
$$\mathcal{\overline{A}}^{sym}_{\lambda}(n,\ell):= (D(\overline{R})/[D(\overline{R})\{\Phi^{sym}(x)-\lambda(x), x \in \mathfrak{g}\}])^{G}.$$ 
   \par
   The filtration on $\mathcal{\overline{A}}_{\lambda}(n,\ell)$ is induced from the Bernstein filtration on $D(\overline{R})$ (here deg $\overline{R}=$ deg $\overline{R}^*=1$). Recall that $\mathbb{C}[\overline{\mathcal{M}}(n,\ell)]=(\mathbb{C}[T^*\overline{R}]/I)^G$, where $I:=\{\mu^*(\xi), \xi \in \mathfrak{g}\}$ is the ideal generated by the image of $\mathfrak{g}$ under the comoment map, and denote $\mathcal{I}_{\lambda}:=\{\Phi(x)-\lambda(x), x \in \mathfrak{gl}(V)\}$. The surjectivity of the natural map $\mathbb{C}[\overline{\mathcal{M}}(n,\ell)]\rightarrow \mbox{gr } \mathcal{\overline{A}}_{\lambda}(n,\ell)$ follows from the containment $I \subset  \mbox{gr } \mathcal{I}_{\lambda}$. The reverse containment of ideals follows from the regularity of the sequence $\mu^*(\xi_1), \hdots, \mu^*(\xi_{n^2})$, where $\xi_1, \hdots, \xi_{n^2}$ is some basis for $\mathfrak{g}$. The regularity of the sequence is equivalent to flatness of the moment map   
$\mu$.   
\par
 We notice that the difference between $\mathcal{\overline{A}}_{\lambda}(n,\ell)$ and the algebra $\mathcal{A}_{\lambda}(n,\ell)$ (constructed analogously for $R=\mathfrak{gl}(V)^{\oplus \ell}\oplus V^{*}$) is that $\mathcal{A}_{\lambda}(n,\ell)=D(\mathbb{C}^{\ell})\otimes \mathcal{\overline{A}}_{\lambda}(n,\ell)$. Thus, some questions about representation theory of $\mathcal{\overline{A}}_{\lambda}(n,\ell)$ reduce to analogous ones for $\mathcal{A}_{\lambda}(n,\ell)$.
\par
The quantizations $\overline{\mathcal{A}}^{\theta}$ of  $\overline{\mathcal{M}}^{\theta}(n,\ell)$ are parameterized (up to isomorphism) by the points of $H^2(\overline{\mathcal{M}}^{\theta}(n,\ell))\simeq \mathbb{C}$ (see \cite{Bez-Ka}). The quantization corresponding to $\lambda$ will be denoted by $\overline{\mathcal{A}}_{\lambda}^{\theta}$.

\par
\textbf{Notation.} We will denote by $\overline{\mathcal{A}}_{\lambda}\operatorname{-mod}$ the category of finitely generated  $\overline{\mathcal{A}}_{\lambda}$-modules and  by $\overline{\mathcal{A}}_{\lambda}^{\theta}\operatorname{-mod}$ - the category of coherent $\overline{\mathcal{A}}_{\lambda}^{\theta}$ - modules.
\par  
There are two basic functors between the categories of  $\overline{\mathcal{A}_{\lambda}}\operatorname{-mod}$ and $\overline{\mathcal{A}_{\lambda}}^{\theta}\operatorname{-mod}$ and the corresponding derived categories: $$\mathcal{A}_{\lambda}\operatorname{-mod}  \mathrel{\mathop{\rightleftarrows}^{\mathrm{Loc^{\theta}_{\lambda}}}_{\mathrm{\Gamma_{\lambda}}}}   \mathcal{A}_{\lambda}^{\theta}\operatorname{-mod}.$$  
\par
$$D^{b}(\overline{\mathcal{A}}_{\lambda}\operatorname{-mod})  \mathrel{\mathop{\rightleftarrows}^{\mathrm{LLoc_{\lambda}^{\theta}}}_{\mathrm{R\Gamma_{\lambda}}}}  D^{b}(\overline{\mathcal{A}}_{\lambda}^{\theta}\operatorname{-mod}).$$

\begin{defn}
If the functors $\mbox{Loc}^{\theta}_{\lambda}, \Gamma_{\lambda}$ ($\mbox{LLoc}_{\lambda}^{\theta}, \mbox{R}\Gamma_{\lambda}$) are mutually inverse equivalences, we say that \textit{abelian (derived) localization holds} for the pair $(\lambda,\theta)$.
\end{defn}
\begin{rmk}
	The main Theorem of \cite{MN2} asserts that the derived equivalence holds if and only if the homological dimension of the algebra $\overline{\mathcal{A}}_{\lambda}$ is finite.
\end{rmk}

 \begin{rmk}
 The categories $\mathcal{O}_{\nu}(\mathcal{\overline{A}}_{\lambda}(n,\ell))$ and $\mathcal{O}_{\nu}(\mathcal{A}_{\lambda}(n,\ell))$ are, in fact, equivalent. Indeed, recall that $\mathcal{A}_{\lambda}(n,\ell)=D(\mathbb{C}^{\ell})\otimes \mathcal{\overline{A}}_{\lambda}(n,\ell)$ and let $t_1,\hdots,t_{\ell}$ be the coordinates on $\mathbb{C}^{\ell}$. Then the functor $\mathcal{O}_{\nu}(\mathcal{\overline{A}}_{\lambda}(n,\ell))\rightarrow\mathcal{O}_{\nu}(\mathcal{A}_{\lambda}(n,\ell))$ given by $M\mapsto \mathbb{C}[t_1,\hdots,t_{\ell}]\otimes M$ produces an equivalence of categories. It has a quasi-inverse functor which sends $N \in \mathcal{O}_{\nu}(\mathcal{A}_{\lambda}(n,\ell))$ to the annihilator of $\langle \partial t_1,\hdots, \partial t_{\ell}\rangle$.
 \end{rmk}
\begin{defn} 
	\label{StandFunctor}
	We have the \textit{standardization} and \textit{costandardization}
  functors $\triangle_{\nu}$ and $\nabla_{\nu}:\mathcal{C}_{\nu}(\mathcal{\overline{A}}_{\lambda}(n,\ell))\operatorname{-mod} \rightarrow \mathcal{O}_{\nu}(\mathcal{\overline{A}}_{\lambda}(n,\ell))$ given by 
   $$\triangle_{\nu}(N):=\mathcal{\overline{A}}_{\lambda}(n,\ell)/\mathcal{\overline{A}}_{\lambda}(n,\ell)\mathcal{\overline{A}}^{>0}_{\lambda}(n,\ell)\otimes_{\mathcal{C}_{\nu}(\mathcal{\overline{A}}_{\lambda}(n,\ell))} N$$
    $$\nabla_{\nu}(N):=\mbox{Hom}_{\mathcal{C}_{\nu}(\mathcal{\overline{A}}_{\lambda}(n,\ell))}(\mathcal{\overline{A}}_{\lambda}(n,\ell)/\mathcal{\overline{A}}^{<0}_{\lambda}(n,\ell)\mathcal{\overline{A}}_{\lambda}(n,\ell),N).$$
   \end{defn}
   \par
   We consider the restricted $\mbox{Hom}$ (w.r.t. the natural grading on $\mathcal{\overline{A}}_{\lambda}(n,\ell)/\mathcal{\overline{A}}^{<0}_{\lambda}(n,\ell)\mathcal{\overline{A}}_{\lambda}(n,\ell)$) in the definition of the operator $\nabla_{\nu}$ above.


   The next result can be found in \cite{Los16} (see Proposition $2.2$).
\begin{prop}
\label{AbLoc}
Suppose that abelian localization holds and $\lambda$ is generic (outside some finite set). Choose a generic one-parametric subgroup $\nu$. Then the following is true:
\begin{enumerate}
	\item [(1)] the category $\mathcal{O}_{\nu}(\mathcal{\overline{A}}_{\lambda}(n,\ell))$ depends only on the chamber of $\nu$;
   \item[(2)] the natural functor $D^{b}(\mathcal{O}_{\nu}(\mathcal{\overline{A}}_{\lambda}((n,\ell)))\rightarrow D^{b}(\mathcal{\overline{A}}_{\lambda}(n,\ell)\operatorname{-mod})$ is a full embedding;
    \item[(3)] $\mathcal{C}_{\nu}(\mathcal{\overline{A}}_{\lambda}(n,\ell))=\mathbb{C}[\overline{\mathcal{M}}^{\theta}(n,\ell)^{T}]$;
   \item[(4)] Assume, in addition, that there are finitely many fixed points for the action of $\nu$. The category $\mathcal{O}_{\nu}(\mathcal{\overline{A}}_{\lambda}(n,\ell))$ is highest weight with standard objects $\triangle_{\nu}(p_{i})$ and costandard objects  $\nabla_{\nu}(p_{i})$ for $p_{i} \in \mathbb{C}[\overline{\mathcal{M}}^{\theta}(n,\ell)^{T}]$. The order required for highest weight structure comes from the contraction order on the fixed points.
   \end{enumerate}
   \end{prop}

\subsection{Main results and structure of the paper}
We present the most important results of the paper in order of appearance. In Section $5$ it is established that abelian localisation holds for $(\lambda,\theta)$ with $\theta<0$ and $\lambda<1-\ell$ or $\theta>0$ and $\lambda>\ell-2$ (Theorem  \ref{AbLocHolds}). It is shown that 
if  $\lambda\in (-\infty;1-\ell)\cup (\ell-2;+\infty)$, then the algebra $ \overline{\mathcal{A}}_{\lambda}(2, \ell)$ has finite homological dimension (Corollary \ref{FinHomDim}). In Section $6$ we determine the Hom-spaces between standard objects in $\mathcal{O}_{\nu}(\mathcal{\overline{A}}_{\lambda}(2,\ell))^{\triangle}$ (see Theorem \ref{HomsStand}) and compute the multiplicities of simples in standards (Corollary \ref{MultiplicityCor}). We show that the algebra $\overline{\mathcal{A}}_{\lambda}(2, \ell)$ is not of finite homological dimension for $\lambda \in (-\ell;\ell-1)\cap \mathbb{Z} \mbox{ or } \lambda=-\frac{1}{2}$ (see Theorem \ref{SingParam}). Finally, the complete form of abelian localisation is established in Theorem \ref{AbLocThm}.
\par
The structure of the paper is as follows. Section $2$ gives preliminary results on the varieties $\overline{\mathcal{M}}^{\theta}(n,\ell)$ and the category $\mathcal{O}_{\nu}(\mathcal{\overline{S}}_{\lambda}(2,\ell))$. It is shown that $\overline{\mathcal{M}}^{\theta}(n,\ell)$ has finitely many fixed points w.r.t. the Hamiltonian torus action for $n\leq 3$, the central fiber of the resolution $\bar{\rho}:\overline{\mathcal{M}}^{\theta}(n,\ell)\rightarrow \overline{\mathcal{M}}(n,\ell)$  is of dimension less than $\frac{1}{2}\overline{\mathcal{M}}^{\theta}(n,\ell)$ for $n,\ell>1$. From this (using Gabber's theorem) one deduces that there are no finite dimensional $\overline{\mathcal{A}}_{\lambda}(n,\ell)$-modules  with generic $\nu$. Furthermore, the resolutions  $\bar{\rho}:\overline{\mathcal{M}}^{\theta}(n,\ell)\rightarrow \overline{\mathcal{M}}(n,\ell)$  serve as counterexamples to Conjecture $1.3.1$ in  \cite{ES}. The explanation of this phenomenon concludes the section (see Remark \ref{CounterEx} for details).
\par
In Section $3$, following the recipe of \cite{Nak94}, \cite{Nak} (see also Section $2$ of \cite{BezrLos}), the description of symplectic leaves of  $\overline{\mathcal{M}}(n,\ell)$ and slices to points on them  for $n=2,3$ is obtained. One of the two nontrivial slices to $\overline{\mathcal{M}}(2,\ell)$ turns out to be a hypertoric variety. The description of $T$-fixed points on that slice is provided.
\par
 Following the lines of \cite{BLPW1}, we give an overview on generalities on hypertoric varieties and categories $\mathcal{O}$ associated to them and provide a description of category $\mathcal{O}$ for the slice  (Proposition \ref{SignVectors}, Section $4$). 
 
The next section is devoted to the proof of Theorem  \ref{AbLocHolds} and the description of the locus of $\lambda$, for which the algebra $ \overline{\mathcal{A}}_{\lambda}(2, \ell)$ has finite homological dimension (Corollary \ref{FinHomDim}).
 \par
 Then, using the construction of restriction functor introduced in \cite{B-Et} for rational Cherednik algebras (quantizations of the Hilbert scheme of points on $\mathbb{C}^2$) and its generalization for the Gieseker scheme in \cite{Los17_1}, we define a functor $Res:\mathcal{O}_{\nu}(\mathcal{\overline{A}}_{\lambda}(2,\ell))\rightarrow \mathcal{O}_{\nu}(\mathcal{\overline{S}}_{\lambda}(2,\ell))$, where $\mathcal{O}_{\nu}(\mathcal{\overline{S}}_{\lambda}(2,\ell))$ stands for the category $\mathcal{O}$ for the slice. This functor is exact and faithful on standard objects. It serves as the main ingredient in the proof of Theorem \ref{HomsStand}, which appears in Section $6$.
 \par
 Section $7$ is dedicated to the proof of Theorem \ref{SingParam}. The main ingredients required here are the results of McGerty and Nevins from \cite{MN1}.
 \par
 \textbf{Acknowledgements.}
 I would like to thank Ivan Losev for introducing me to the subject, constant guidance and numerous helpful suggestions. I am grateful to Pavel Etingof for explaining the connections of the results in Section $2$ to those in \cite{ES}.
\section{First results on $\mathcal{O}_{\nu}(\mathcal{\overline{A}}_{\lambda}(n,\ell))$}
In this section we collect some basic information on the category $\mathcal{O}_{\nu}(\mathcal{\overline{A}}_{\lambda}(n,\ell))$. Recall that $\mathcal{\overline{A}}_{\lambda}(n,\ell):= (D(\overline{R})/[D(\overline{R})\{\Phi(x)-\lambda(x), x \in \mathfrak{g}\}])^{G}$ stands for the quantization of $\overline{\mathcal{M}}^{\theta}(n,\ell)$. We fix our  choice of character $\theta=det^{-1}$.
\begin{lem}
	\label{Symplectomorphism}
	There is an isomorphism $\mathcal{\overline{A}}_{\lambda}(n,\ell)\cong\mathcal{\overline{A}}_{-\lambda-1}(n,\ell) $.
\end{lem}
\begin{proof}
	There is a symplectomorphism $\gamma: \overline{\mathcal{M}}^{\theta}(n,\ell)\simeq \overline{\mathcal{M}}^{-\theta}(n,\ell)$ produced by $$(X_{1}, \hdots, X_{\ell}, Y_{1}, \hdots, Y_{\ell},i,j) \mapsto (Y_{1}^{*}, \hdots, Y_{\ell}^{*}, -X_{1}^{*}, \hdots, -X_{\ell}^{*},j^{*},-i^{*}),$$ thus, inducing multiplication by $-1$ on $H^{2}(\overline{\mathcal{M}}^{\theta}(n,\ell), \mathbb{Z})$. As the image of $\lambda$ under the period map is $\lambda+\frac{1}{2}\in H^{2}(\overline{\mathcal{M}}^{\theta}(n,\ell), \mathbb{Z})$, the result follows. 
\end{proof}
\par
 To study the category  $\mathcal{O}_{\nu}(\mathcal{\overline{A}}_{\lambda}(n,\ell))$, we first need to obtain some information on the torus fixed points. This is summarized in the theorem below.
\begin{rmk}
	Since the case $\ell=1$ was studied in \cite{Los17_1}, henceforth we assume $\ell\geq 2$.
\end{rmk}
\begin{thm}
\label{FixedPts}
The variety $\overline{\mathcal{M}}^{\theta}(n,\ell)$ has finitely many $T$-fixed points if dim$V\leq 3$.
\end{thm}
\begin{proof}
Let $\tilde{p}=(X_{1}, \hdots, X_{\ell},Y_{1}, \hdots, Y_{\ell},i,j)\in \mu^{-1}(0)$ be a point in the preimage of a fixed point $p\in \overline{\mathcal{M}}^{\theta}(n,\ell)$, then there exists a homomorphism $\eta_p:T \rightarrow G$, s.t. the following system of equalities is satisfied ($t=(t_{1},\hdots t_{\ell})\in T$): \begin{equation}
 \label{eq:(FixedPts)}
 \begin{cases} t_{1}X_{1}=\eta_p(t)X_{1}\eta_p(t)^{-1} \\ \hdots \\ t_{\ell}X_{\ell}=\eta_p(t)X_{\ell}\eta_p(t)^{-1} \\ t_{1}^{-1}Y_{1}=\eta_p(t)Y_{1}\eta_p(t)^{-1} \\ \hdots \\ t_{\ell}^{-1}Y_{\ell}=\eta_p(t)Y_{\ell}\eta_p(t)^{-1}\\ i=\eta_p(t)^{-1}i \\ j=\eta_p(t)j. \end{cases}
 \end{equation}
\par
Let $\{\varepsilon_1,\ldots,\varepsilon_\ell\}$ be the set of coordinate characters of the torus $T$, i.e. $\varepsilon_i(t_1,\hdots,t_\ell)=t_i$. The  weight decomposition of $V$ with respect to $\eta_p$  is 
$$V=\underset{\chi \in char(T)}{\bigoplus} V_{\chi},$$
with $V_{\chi}=\{v\in V|~ \eta_p(t)\cdot v=\chi(t)v\}$.
It follows from the system of equations \eqref{eq:(FixedPts)} that $X_i(V_\chi)\subset V_{\chi-\varepsilon_i}$ and, similarly, the $Y_{i}$'s - to $Y_i(V_\chi)\subset V_{\chi+\varepsilon_i}$ (here multiplication of characters is written additively). As $im~j\neq 0$ due to the stability condition it follows from the last equation in \eqref{eq:(FixedPts)}  that  im$~j\in V_0$.
 \par
 Below we provide a description of the fixed points when dim$(V)\leq 3$.
\par
\textit{Case 1}. If dim$(V)=1$, the variety $\overline{\mathcal{M}}^{\theta}(1,\ell)$ is a single point.
\par
\textit{Case 2}. If dim$(V)=2$, we choose a cyclic vector  $v_{0}\in im~j$ as the first vector in the basis. Then at least one of the $X_{k}$ or $Y_{s}$ must act nontrivially on $v_{0}$ and the image is $v_{1}$ inside some $V_{\pm \varepsilon_i}$. The vectors $v_{0}$ and $v_{1}$ already span $V$ as they have different weights and cannot be collinear. We notice that $X_{s}v_{0}=v_{1}$ or $Y_{s}v_{0}=v_{1}$ immediately implies $X_{\neq s}v_{0}=Y_{\neq s}v_{0}=X_{\neq s}v_{1}=Y_{\neq s}v_{1}=0$ as all these vectors would lie in weight spaces different from $V_{0,\hdots,0}$ and $V_{0,\hdots,0,\pm1_{s},0,\hdots,0}$. It remains to notice that equation $\eqref{eq:(moment)}$ becomes $[X_{s},Y_{s}]+ji=0$, which shows that $X_{s}\neq 0$ implies $Y_{s}=0$ and vice versa. Therefore, there are $2\ell$ fixed points: $p_{s}=(X_{\neq s}=0, X_{s}=\left(\begin{array}{cc}
0 & 0 \\
1 & 0\end{array}\right), Y_{1}=0, \hdots, Y_{\ell}=0, i=0, j=\left(\begin{array}{c}
1  \\
0\end{array}\right))$, $p_{s+\ell}=(X_{1}=0, \hdots, X_{\ell}=0, Y_{\neq s}=0, Y_{s}=\left(\begin{array}{cc}
0 & 0 \\
1 & 0\end{array}\right), i=0, j=\left(\begin{array}{c}
1  \\
0\end{array}\right))$ , where $s \in \{1,\hdots, \ell\}$.
\par
\textit{Case 3}.
 Now dim$(V)=3$. Again let the cyclic vector $v_{0}:=im~j$ be the first vector in the basis. Now there are the following possibilities ($s,k \in \{1,\hdots,\ell\}$): 
 \par
 $\bullet$ for some $s,k$: $X_{s}v_{0}=v_{1}\neq 0$ and $Y_{k}v_{0}=v_{2}\neq 0$; 
  \par
 $\bullet$ for some $s\neq k$:   $X_{s}v_{0}=v_{1}\neq 0$ and $X_{k}v_{0}=v_{2}\neq 0$;
   \par
 $\bullet$ for some $s\neq k$: $Y_{s}v_{0}=v_{1}\neq 0$ and $Y_{k}v_{0}=v_{2}\neq 0$;
   \par
 $\bullet$ for some $s\neq k$:  $X_{s}v_{0}=v_{1}\neq 0$ and $Y_{k}v_{1}=v_{2}\neq 0$;
   \par
 $\bullet$ for some $s,k$: $X_{s}v_{0}=v_{1}\neq 0$ and $X_{k}v_{1}=v_{2}\neq 0$;
   \par
 $\bullet$ for some $s,k$: $Y_{s}v_{0}=v_{1}\neq 0$ and $Y_{k}v_{1}=v_{2}\neq 0$;
 \par
  In each of the cases above the vectors $v_{0}, v_{1}$ and $v_{2}$ are linearly independent and span $V$, while all the remaining $X$ and $Y$ coordinates of $p$ are zero. We verify it when $X_{s}v_{0}=v_{1}$ and $Y_{k}v_{0}=v_{2}$, the remaining cases being similar. 
\par
First, $X_{\neq k}$ and $Y_{\neq s}$ must be zero, as otherwise there would be vectors with weights different from those of  $v_{0}, v_{1}$   and $v_{2}$ and, therefore, linearly independent with them. For the same reason $X_{k}v_{0}=X_{k}v_{1}=X_{s}v_{1}=X_{s}v_{2}=Y_{s}v_{0}=Y_{s}v_{2}=Y_{k}v_{1}=Y_{k}v_{2}=0$. To show  $Y_{s}v_{1}=0$, we notice that equation $\eqref{eq:(moment)}$ reduces to $[X_{s}, Y_{s}]+[X_{k}, Y_{k}] + ji = 0$. Applying to $v_{1}$, we get $$X_{s}Y_{s}v_{1}+jiv_{1}=0,$$ and notice that $X_{s}Y_{s}v_{1} \in V_{0,\hdots,-1_{s},\hdots,0}$, while $ji v_{1} \in V_{0,\hdots,0_{s},\hdots,0}$. Thus, $jiv_{1}=0$ and $X_{s}Y_{s}v_{1}=0$ separately, so $Y_{s}v_{1}=0$ and $Y_{s}=0$. It is analogous to show $X_{k}=0$.
 \par
\begin{center}
\begin{tikzpicture}
\matrix(m)[matrix of math nodes,
row sep=3em, column sep=2.5em,
text height=1.5ex, text depth=0.25ex]
{v_{0} & v_{1}\\ v_{2} \\};
\path[->,font=\scriptsize]
(m-1-1)  edge node[right] {$Y_{k}$} (m-2-1)
(m-1-1)  edge node[above] {$X_{s}$} (m-1-2); 
\end{tikzpicture}
\end{center}
\end{proof}
 \begin{rmk} 
Next we show that when $n=4, \ell=2$  the subvariety of fixed points contains a copy of the projective line $\mathbb{CP}^{1}=\mathbb{C}[\mu_{1}:\mu_{2}]$. The operators below are presented in a weight basis with the first vector of weight $(0,0)$, the second $(-1,0)$, the third $(0,-1)$ and the fourth $(-1,-1)$, the action of the subgroup of $G$, preserving the weight decomposition, can only simultaneously rescale $\mu_{1} $ and $\mu_{2}$. The subvariety is given by
$$X_{1}=\left(\begin{array}{cccc}
0 & 0 & 0 & 0\\
1 & 0 & 0 & 0\\
0 & 0 & 0 & 0\\
0 & 0 & \mu_{1} & 0
\end{array}\right), ~X_{2}=\left(\begin{array}{cccc}
0 & 0 & 0 & 0\\
0 & 0 & 0 & 0\\
1 & 0 & 0 & 0\\
0 & \mu_{2} & 0 & 0
\end{array}\right),~Y_{1}=Y_{2}=0, ~i=0,~ j=\left(\begin{array}{c}
1 \\
0 \\
0\\
0 
\end{array}\right).$$
 \par
\begin{center}
\begin{tikzpicture}
\matrix(m)[matrix of math nodes,
row sep=3em, column sep=2.5em,
text height=1.5ex, text depth=0.25ex]
{& v_{0,0} \\ v_{-1,0} & & v_{0,-1}\\ & v_{-1,-1} \\};
\path[->,font=\scriptsize]
(m-1-2)  edge node[left] {$X_{1}$} (m-2-1)
(m-1-2)  edge node[right] {$X_{2}$} (m-2-3)
(m-2-1)  edge node[left] {$X_{2}$} (m-3-2)
(m-2-3)  edge node[right] {$X_{1}$} (m-3-2); 
\end{tikzpicture}
\end{center}
\end{rmk}
\begin{rmk}
	Both varieties $\overline{\mathcal{M}}^{\theta}(1,\ell)$ and $\overline{\mathcal{M}}(1,\ell)$ consist of a single point, therefore, we proceed with the case $dim V=2$.
\end{rmk}
The following fact is a particular case of the result established in Section $5$ of \cite{Los17} and will be used in the proof of Theorem \ref{HomsStand}. Suppose $\tilde{\nu}$ lies in the face of a chamber containing $\nu$. Then $\triangle_{\tilde{\nu}}$
restricts to an exact functor $\mathcal{O}_{\nu}(C_{\tilde{\nu}}(\mathcal{\overline{A}}_{\lambda}(2,\ell))) \rightarrow \mathcal{O}_{\nu}(\mathcal{\overline{A}}_{\lambda}(2,\ell))$. Moreover, there is an isomorphism of functors $\triangle_{\nu}=\triangle_{\tilde{\nu}}\circ \underline{\triangle}$, where $ \triangle_{\tilde{\nu}}: C_{\tilde{\nu}}(\mathcal{A}_{\lambda})\operatorname{-mod} \rightarrow \mathcal{A}_{\lambda}\operatorname{-mod}, \underline{\triangle}: C_{\nu}(\mathcal{A}_{\lambda})\operatorname{-mod} \rightarrow C_{\tilde{\nu}}(\mathcal{A}_{\lambda})\operatorname{-mod}$ and $\triangle_{\nu}$ is the standardization functor given by Definition \ref{StandFunctor}. This allows to study the functor $\triangle_{\tilde{\nu}}$ in stages.
\par 
We start by describing the fixed points loci $\overline{\mathcal{M}}^{\theta}(2,
\ell)^{\nu(\mathbb{C}^{*})}$ for certain one-parameter subgroups $\tilde{\nu}: \mathbb{C}^{*} \rightarrow T$ and the corresponding algebras $C_{\tilde{\nu}}(\mathcal{A}_{\lambda})$.
\begin{thm}
	\label{AllFP}
The fixed point set $\overline{\mathcal{M}}^{\theta}(2,
\ell)^{\tilde{\nu}(\mathbb{C}^{*})}$ for $\tilde{\nu}: \mathbb{C}^{*}\rightarrow T$ with $\tilde{\nu}(t)=(t^d,1,\hdots,1)$ and $d \in \mathbb{Z}_{>0}$ is $\overline{\mathcal{M}}^{\theta}(2,
\ell-1)\amalg  \mathbb{C}^{2\ell-2}\amalg  \mathbb{C}^{2\ell-2}$.
\end{thm}
\begin{proof}
The subset $ \overline{\mathcal{M}}^{\theta}(2,\ell)^{\tilde{\nu}(\mathbb{C}^{*})}$ is formed by the points $p=(X_{1},\hdots,X_{\ell},Y_{1},\hdots,Y_{\ell},i,j)$ which satisfy the system of equations \eqref{AdditionalFixedOnePPoints} below. These equations are obtained analogously to those in \eqref{eq:(FixedPts)} with $\eta_p$ standing for the composition $\mathbb{C}^*\stackrel{\tilde{\nu}}{\rightarrow}T \rightarrow G$, s.t. 
\begin{equation}
\label{AdditionalFixedOnePPoints}
 \begin{cases} t^{d}X_{1}=\eta_p(t)X_{1}\eta_p(t)^{-1} \\ X_{2}=\eta_p(t)X_{2}\eta_p(t)^{-1} \\ \hdots \\ X_{\ell}=\eta_p(t)X_{\ell}\eta_p(t)^{-1} \\ t^{-d}Y_{1}=\eta_p(t)Y_{1}\eta_p(t)^{-1} \\ Y_{2}=\eta_p(t)Y_{2}\eta_p(t)^{-1} \\ \hdots \\ Y_{\ell}=\eta_p(t)Y_{\ell}\eta_p(t)^{-1}\\ i=\eta_p(t)^{-1}i \\ j=\eta_p(t)j. \end{cases}
 \end{equation}
 and  $\eta_p$ is the same for points in the same connected component. Let $\tilde{\eta}_p$ be the one-parameter subgroup $(\tilde{\nu}, \eta_p)\subset T\times G$.  The irreducible components of $\overline{\mathcal{M}}^{\theta}(2,\ell)^{\tilde{\nu}(\mathbb{C}^{*})}$ can be recovered as the Hamiltonian reductions of the vector space $( T^*\overline{R})^{\tilde{\eta}_p} $ with respect to the action of $Z_{\eta_p}$ (the centralizer  of $\eta_p$ in $G$).
 \par
 There are two possible cases. First, if $X_1=Y_1=0$, it follows from \eqref{AdditionalFixedOnePPoints} and our choice of the stability condition that the entire $2$-dimensional vector space $V$ is of weight $0$ with respect to $\eta_p(t)$ and, hence, $\eta_p(t)=\left(\begin{array}{cc}
1 & 0 \\
0 & 1\end{array}\right)$. Such points form the fixed component $ \overline{\mathcal{M}}^{\theta}(2,
\ell-1)\subset \overline{\mathcal{M}}^{\theta}(2,\ell)^{\tilde{\nu}(\mathbb{C}^{*})} $. Indeed, $ ( T^*\overline{R})^{\tilde{\eta}_p}=\mathfrak{sl}_2^{\oplus 2\ell-2}\oplus V\oplus V^*$ and $Z_{\eta_p}=G$.
 
Next we treat the case when $(X_1,Y_1)\neq 0$. Let $v_0\in im~j$ be a cyclic vector.  Notice, that since dim$V=2$ and $X_1 v_0\subset V_{-d\varepsilon_1}$ while $Y_1 v_0\subset V_{d\varepsilon_1}$, we must have that at least one of the operators $X_1,Y_1$ is zero as well the remaining one squared. Therefore, the matrix of the nonzero operator is conjugate to $\left(\begin{array}{cc}
0 & 0 \\
1 & 0\end{array}\right)$.  One observes that  $X_1=\left(\begin{array}{cc}
0 & 0 \\
1 & 0\end{array}\right)$ and $Y_1=0$ implies the weight basis of $V$ consists of vectors with weights $0$ and $d$, while $\eta_p(t)=\left(\begin{array}{cc}
1 & 0 \\
0 & t^{d}\end{array}\right)$ in this basis, similarly, $\eta_p(t)=\left(\begin{array}{cc}
1 & 0 \\
0 & t^{-d}\end{array}\right)$ if $Y_1=\left(\begin{array}{cc}
0 & 0 \\
1 & 0\end{array}\right)$ and $X_1=0$. In either of the two cases $(T^*\overline{R})^{\tilde{\eta}_p}=\{X_2,\hdots,X_{\ell},Y_2,\hdots,Y_{\ell}~|~X_i,Y_j\in \mathfrak{h}\subset \mathfrak{sl}_2\}$ and the action of $Z_{\eta_p}=\left(\begin{array}{cc}
* & 0 \\
0 & *\end{array}\right)\subset G$ is trivial, hence the Hamiltonian reduction is isomorphic to $\mathbb{C}^{2\ell-2}$.
\end{proof}
 \begin{rmk}
	\label{SmTFPpp}
	The $T'\simeq(\mathbb{C}^{*})^{\ell-1}:=\{(t_1,\hdots,t_{\ell-1},t_{\ell})\subset T~|~t_{\ell}=1\}$ fixed points on  $\overline{\mathcal{M}}^{\theta}(2,\ell)$ are $\{p\in\overline{\mathcal{M}}^{\theta}(2,\ell)~|~X_1=\hdots =X_{\ell-1}=Y_1=\hdots =Y_{\ell-1}=0\}\simeq T^{*}\mathbb{P}^{1}$ and $2\ell-2$ copies of $\mathbb{C}^{2}$. Indeed, $X_{\ell}$ and $Y_{\ell}$ now preserve the weights of weight vectors. Therefore, there are two possibilities:
	\begin{enumerate}
		\item[(i)] the vector space $V=V_0$, so $X_{\neq \ell}=Y_{\neq \ell}=0$ and we arrive at  $T^{*}\mathbb{P}^{1}$ described above;
		\item[(ii)]
		$V$ is spanned by $v_{0}\in V_0$ and $v_1\in V_{\pm\varepsilon_s}$, in which case $X_{\ell}=\left(\begin{array}{cc}
		a & 0 \\
		0 & -a\end{array}\right), Y_{\ell}=\left(\begin{array}{cc}
		b & 0 \\
		0 & -b\end{array}\right)$, one of $X_{s}$ or $Y_{s}$ is $\left(\begin{array}{cc}
		0 & 0 \\
		1 & 0\end{array}\right)$ (depending on the sign of the corresponding weight of $v_{1}$), the other $X$'s and $Y$'s as well as $i$ are $0$ and $j=\left(\begin{array}{c}
		1 \\
		0 \end{array}\right)$. Since the remaining action of $G$ is trivial and $s\in \{1,\hdots,\ell-1\}$, this gives rise to  $2\ell-2$ copies of $\mathbb{C}^{2}$.
	\end{enumerate}
\end{rmk}
\begin{prop}
\label{RestrParam'}
Let $\nu_0$ and $\nu'$ be the one-parameter subgroups from Theorem \ref{AllFP}.
\begin{enumerate}
\item[$(a)$] We have an isomorphism of algebras $C_{\nu_0}(\overline{\mathcal{A}}_{\lambda}(2,\ell))\simeq \overline{\mathcal{A}}_{\lambda}(2,\ell-1)\oplus \mathcal{D}(\mathbb{C}^{2\ell-2})\oplus \mathcal{D}(\mathbb{C}^{2\ell-2})$, where $\overline{\mathcal{A}}_{\lambda}(2,\ell-1)$ is a quantization of $Z=\overline{\mathcal{M}}^{\theta}(2,\ell-1)$.
\par
\item[$(b)$]  Similarly, $C_{\nu'}(\overline{\mathcal{A}}_{\lambda}(2,\ell))\simeq\mathcal{A}^{Z_1}_{\lambda+1-\frac{\ell}{2}}\oplus\mathcal{A}^{Z_2}_{\lambda+\frac{\ell}{2}}$, where $Z_1, Z_2$ are the fixed components for $\nu'$ and $\mathcal{A}^{Z_i}_{\mu}$ stands for the quantization of $Z_i$ with period $\mu$.
\end{enumerate}
\end{prop}
\begin{proof}
Proposition $2.2$ \cite{Los16} asserts that $C_{\nu_0}(\overline{\mathcal{A}}^{sym}_{\lambda}(2,\ell))=\underset{k}{\oplus}\mathcal{A}^{Z_k}_{i_{Z_k}^{*}(\lambda)-\rho_{Z_k}}$, where $Z_k$'s are the irreducible components of  $\overline{\mathcal{M}}^{\theta}(2,\ell)^{\nu_0}$ and $\mathcal{A}^{Z_k}_{i_{Z_k}^{*}(\lambda)-\rho_{Z_k}}$ stands for the algebra of global sections of the filtered quantization of $Z_k$ with period ${i_{Z_k}^{*}(\lambda)-\rho_{Z_k}}$. Here  $i_{Z}^{*}$ is the pull-back map $H^2(\overline{\mathcal{M}}^{\theta}(2,\ell),\mathbb{C})\rightarrow H^2(Z,\mathbb{C})$ and $\rho_{Z_k}$ equals half of the 1st Chern class of the contracting bundle of $Z_k$. We start with describing this bundle in our case. For the general description of tangent spaces to quiver varieties we refer to  Lemma $3.10$ and Corollary $3.12$ in \cite{Nak}. The tangent bundle
 descends from the $G$-module ker $\beta$ / im $\alpha$, where $\alpha$ and $\beta$ are in the following complex:
\begin{equation}
\label{tangent}
\Hom(V,V) \stackrel{\alpha}{\hookrightarrow} \mathfrak{sl}_2\otimes \mathbb{C}^{2\ell}\oplus V\oplus V^* \overset{\beta}{\twoheadrightarrow} \Hom(V,V),
\end{equation}
here  $\alpha$ stands for the differential of the $G$-action and  $\beta$ is the differential of the moment map at that fixed point.
\par
It is not hard to observe that the sequence \eqref{tangent} is equivariant with respect to the $(\mathbb{C}^{*})^{\ell}$-action with $\beta$ surjective and $\alpha$ injective. 
\par
We proceed with verifying the assertion of $(a)$. As every bundle over the $\mathbb{C}^{2\ell-2}$ component of $Z$ is trivial, we look at the restriction of the contracting bundle to $\overline{\mathcal{M}}^{\theta}(2,\ell-1)$. 
\par

  It follows from the description of the tangent bundle as the middle cohomology of the complex \eqref{tangent}  that the contracting bundle descends under $G$-action from $T^*\overline{R}^{\tilde{\eta}_p, >0}$ modulo two copies of $\mathfrak{g}^{\tilde{\eta}_p, >0}$. In our case $(T^*\overline{R})^{\tilde{\eta}_p, >0}=H$ is the three-dimensional space $Vec(X_1)$, while  $\mathfrak{g}$ is pointwise fixed under the action of $\tilde{\eta}_p$, hence, the contracting bundle descends from  $H$.

 The top exterior power of the vector bundle $\widetilde{H}$ descending from  $H$ under $G$-action is trivial, since $G$ acts trivially on the top exterior power of $H$. By \cite{LosQuant}, Section $5$,
 the period of a quantization $\overline{\mathcal{A}}_{\lambda}(2,\ell)$ is $\lambda-\zeta$, where $\zeta$ is half the character of the action of $G$ on
 $\Lambda^{top}\overline{R}$. Thus the periods of the quantizations  $\overline{\mathcal{A}}_{\lambda}(2,\ell)$ and  $\overline{\mathcal{A}}_{\lambda}(2,\ell-1)$
 are both equal to $\lambda+\frac{1}{2}$, the first claim of the proposition follows.

We verify the claim in $(b)$ for $Z_1$.  There is a line subbundle $L_{triv} \subset \tilde{V}$ with the fiber over a point $p\in Z_1$ being $\mbox{im}~j$. It is trivial, since for a fixed $0\neq w\in W$ one has a nowhere vanishing section $j(w)$ of $\tilde{V}$. Using the splitting principle, we write   $\tilde{V}=L_{triv}\oplus L_1$ with $c_1(L_{triv})=0$ and $c_1(L_1)=c_{Z_1}$, where $c_{Z_1}$ is the generator of $H^2(Z_1)$.  In this case $V=\mathbb{C}\langle v_0,v_1\rangle$ with $v_0$ of weight $0$ and $v_1$ of weight $d$, in other words, $\eta_p=\left(\begin{array}{cc}
1 & 0 \\
0 & t^d \end{array}\right)$ in the basis $\langle v_0, v_1\rangle$. This implies that the bundle on $Z_1$ descending from $\mathfrak{sl}_2$ is $L_{triv}\oplus L_1\oplus L_1^*$. 
Let $\tilde{\eta}_p(t)=(\nu'(t),\eta_p(t))\subset T\times G$, then $U^{\tilde{\eta}_p, >0}=(z_1,\hdots, z_{\ell},v_1)$, where $z_s$ is the $12$-entry (first row and second column) of the matrix $X_s$, while $\mathfrak{g}^{\tilde{\eta}_p, >0}$ consists of the $12$-entry of the corresponding matrix. Hence, the nontrivial part of the contracting bundle is $L_1\otimes\mathbb{C}^{\ell-1}$. Thus we  conclude
 that $\rho_{Z_1}=\frac{\ell-1}{2}c_{Z_1}$. 
\par
 Analogously one can show that the nontrivial part of the contracting bundle on $Z_2$ is $L_1^*\otimes\mathbb{C}^{\ell-1}$ and  $\rho_{Z_2}=\frac{1-\ell}{2}c_{Z_2}$. The maps $i_{Z_1}^{*}$ and $i_{Z_2}^{*}$ send $c_1(\tilde{V})\in H^2(\overline{\mathcal{M}}^{\theta}(2,\ell))$ to the generators $c_{Z_1}\in H^2(Z_1)$ and $c_{Z_2}\in H^2(Z_2)$. The claim in $(b)$ follows.
\end{proof}
\begin{rmk}
\label{PeriodsOnProjSpace}
The quantizations $\mathcal{A}^{Z_1}_{\lambda+1-\frac{\ell}{2}}\mbox{ and }\mathcal{A}^{Z_2}_{\lambda+\frac{\ell}{2}}$ are isomorphic to $\mathcal{D}^{\lambda-\ell+1}(\mathbb{P}^{\ell-1})$ and $\mathcal{D}^{\lambda}(\mathbb{P}^{\ell-1})$ (the algebras of twisted differential operators on projective spaces).
\end{rmk}
The next lemma provides some information on the preimages of zero under $\bar{\rho}:\overline{\mathcal{M}}^{\theta}(n,\ell)\rightarrow \overline{\mathcal{M}}(n,\ell)$ (central fibers) in $\overline{\mathcal{M}}^{\theta}(n,\ell)$. 
\begin{lem} 
\label{Preimage}
\begin{enumerate}
	\item[(a)]  The preimage of $0$ in $\overline{\mathcal{M}}^{\theta}(2,\ell)$ is $\bar{\rho}^{-1}(0)=\mathbb{P}^{2\ell-1}$.
 \item[(b)]  Let $n,\ell>1$, then $\mbox{dim}(\bar{\rho}^{-1}(0))\leq\frac{1}{2}\mbox{dim}(\overline{\mathcal{M}}^{\theta}(n,\ell))$.
 \end{enumerate} 
\end{lem}
\begin{proof}
An application of the Hilbert-Mumford criterion shows (the argument is analogous to the one in Proposition $9.7.4.$ in \cite{DW})  that $p\in \overline{\mathcal{M}}^{\theta}(n,\ell)$ lies in $\bar{\rho}^{-1}(0)$ if and only if on the corresponding representation there exists a filtration $0 = L_0 \subset L_1 \subset  L_2 \subset \hdots \subset  L_n  = r_p \in T^*\overline{R}$ by subrepresentations such that each quotient $L_i/L_{i-1}$ for $i<n$  is isomorphic to a simple representation (of the framed quiver $\widehat{B}_{2\ell}$) with dimension vector $\left(\begin{array}{c}
0 \\
1\end{array}\right)$ and $L_n/L_{n-1}$ is isomorphic to the simple representation with dimension vector $\left(\begin{array}{c}
1 \\
0\end{array}\right)$ (here the top coordinate corresponds to the dimension of framing and the bottom to the dimension of $V$). This implies that all the $\mathfrak{sl}_n$-components of $p$ must be strictly upper-triangular matrices. It follows from equation $\eqref{eq:(moment)}$ and our choice of stability condition, that $i=0$.

$(a)$ Pick a vector $0\neq h \in im~j$. As $h$ is a cyclic vector, it must have a nontrivial projection onto $V/V_1$.  The action by matrices of the form $\left(\begin{array}{cc}
1 & \alpha \\
0 & 1\end{array}\right)$ (conjugation by which does not change any of the $2\times 2$  matrices of $p$) allows to assume that the component of $h$ along the first vector is zero. Acting by  $\left(\begin{array}{cc}
1 & 0 \\
0 & *\end{array}\right) \subset GL_{2}$, allows to  pick a representative of $p$ with $j=\left(\begin{array}{c}
0 \\
1\end{array}\right) $  and the action by $\left(\begin{array}{cc}
* & 0 \\
0 & 1\end{array}\right) \subset GL_{2}$ to simultaneously rescale the $2\times 2$  matrices of $p$.
We conclude that 
$p=(X_{1}=\left(\begin{array}{cc}
0 & a_{1} \\
0 & 0\end{array}\right), \hdots, X_{\ell}=\left(\begin{array}{cc}
0 & a_{\ell} \\
0 & 0\end{array}\right),Y_{1}=\left(\begin{array}{cc}
0 & a_{\ell+1} \\
0 & 0\end{array}\right), \hdots, Y_{\ell}=\left(\begin{array}{cc}
0 & a_{2\ell} \\
0 & 0\end{array}\right), i=0, j=\left(\begin{array}{c}
0  \\
1\end{array}\right))$ with at least one of $X_{k}$'s and $Y_{s}$'s, being nonzero due to the stability condition, up to simultaneous dilations of $X_{k}$'s and $Y_{s}$'s, which shows the claim, stated in $(a)$.
 \par
Now we show the claim in $(b)$. Acting by matrices of the form $\left(\begin{array}{cccccc}
1 & 0 & \hdots & 0 & * \\
0 & 1 & \hdots & 0 & *  \\ 
\vdots & \vdots & \ddots & \vdots & \vdots \\
0 & 0 & \hdots & 1 & * \\
0 & 0 & \hdots & 0 & 1 \\
\end{array}\right)$, we can assume that $h$ is proportional to the last vector in the basis. The action by the subgroup $$\left(\begin{array}{cccccc}
1 & 0 & \hdots & 0 & 0 \\
0 & 1 & \hdots & 0 & 0  \\ 
\vdots & \vdots & \ddots & \vdots & \vdots \\
0 & 0 & \hdots & 1 & 0 \\
0 & 0 & \hdots & 0 & * \\
\end{array}\right)\subset GL_{n}$$ allows to assume $ j=\left(\begin{array}{c}
0\\
\vdots\\
0\\
1\end{array}\right)$. 
\par
Since $i=0$, the moment equation $\eqref{eq:(moment)}$ reduces to $\sum\limits_{k=0}^{\ell}[X_k,Y_k]=0$ and as each of the commutators is a matrix of the form $$[X_k,Y_k]=\left(\begin{array}{cccccc}
0 & 0 & * & * & \hdots   & * \\
0 & 0 & 0 & * &\hdots & *  \\ 
\vdots & \vdots  & \vdots & \ddots & \ddots & \vdots \\
0 & 0 & 0 & 0 &\hdots & *  \\ 
0 & 0 & 0 & 0 &\hdots & 0  \\ 
\end{array}\right),$$
equation $\eqref{eq:(moment)}$ imposes $\frac{(n-1)(n-2)}{2}$ independent conditions on the coordinates of $p\in \bar{\rho}^{-1}(0)$. The action of matrices of the form 
$$\left(\begin{array}{cccccc}
* & * & \hdots & * & 0 \\
0 & * & \hdots & * & 0  \\ 
\vdots & \vdots & \ddots & \vdots & \vdots \\
0 & 0 & \hdots & * & 0 \\
0 & 0 & \hdots & 0 & 1 \\
\end{array}\right)$$
preserves both $j$ and the strictly upper-triangular matrices and reduces the dimension  by $\frac{n(n-1)}{2}$. Therefore, we have established that $$\mbox{dim}(\bar{\rho}^{-1}(0))\leq \frac{n(n-1)}{2}2\ell-\frac{(n-1)(n-2)}{2}-\frac{n(n-1)}{2}=(n^2-n)\ell-n^2+2n-1$$
and a straightforward computation shows that $(n^2-n)\ell-n^2+2n-1<(\ell-1)n^2-\ell+n=\frac{1}{2}\mbox{dim}(\overline{\mathcal{M}}^{\theta}(n,\ell))$ provided $n,\ell>1$.
\end{proof}
 \begin{cor}
	\label{LagrnotIsotr}
	Assume $n,\ell>1$, then the central fiber $\bar{\rho}^{-1}(0)\subset \overline{\mathcal{M}}^{\theta}(n,\ell)$ is an isotropic but not Lagrangian subvariety.
\end{cor}
\begin{rmk}
 The $T$ - fixed points for the action on $\overline{\mathcal{M}}^{\theta}(2,\ell)$ (see Theorem $2.1$)  lie on $\overline{\rho}^{-1}(0)=\mathbb{P}^{2\ell-1}$.
 \end{rmk}
 \par
 \begin{cor}
 \label{H2}
 $H^2(\overline{\mathcal{M}}^{\theta}(2,\ell))\simeq \mathbb{C}$.
 \end{cor}
\begin{cor} There are no finite dimensional $\overline{\mathcal{A}}_{\lambda}(n,\ell)$-modules for  $n, \ell>1$  and generic $\nu$.
\end{cor}
\begin{proof}
The support of a finite dimensional module must be  $0 \in \overline{\mathcal{M}}(n,\ell)$ (since $0$ is the only fixed point of $\overline{\mathcal{M}}(n,\ell)$ for the scaling $\mathbb{C}^{*}$-action, the support is $\mathbb{C}^{*}$-stable and the module is finite dimensional). Then the support of $Loc_{\lambda}^{\theta}(M)$ must be contained in $\bar{\rho}^{-1}(0) \subset \overline{\mathcal{M}^{\theta}}(n,\ell)$. On the other hand, due to Gabber's involutivity theorem, the support of a  coherent module must be a coisotropic subvariety of $\overline{\mathcal{M}}^{\theta}(n,\ell)$.  However, this is impossible for dimension reasons (dim$(\overline{\mathcal{M}}^{\theta}(2,\ell))=6\ell-4$ and dim$(\overline{\mathcal{M}}^{\theta}(3,\ell))=16\ell-12>2(6\ell-3)$ for $\ell >1$).
\end{proof}
\par
 I would like to thank Pavel Etingof for bringing my attention to the following fact. 
\begin{rmk}
	\label{CounterEx}
The resolutions $\overline{\rho}:\overline{\mathcal{M}}^{\theta}(n,\ell)\rightarrow \overline{\mathcal{M}}(n,\ell)$ serve as counterexamples to Conjecture $1.3.1$ in  \cite{ES}. Indeed, $H^{top}(\overline{\mathcal{M}}^{\theta}(2,\ell),\mathbb{C})=H^{3\ell-2}(\mathbb{P}^{2\ell-1},\mathbb{C})=0$ and, in general, $H^{top}(\overline{\mathcal{M}}^{\theta}(n,\ell),\mathbb{C})=H^{\frac{1}{2}\mbox{dim}(\overline{\mathcal{M}}^{\theta}(n,\ell))}(\overline{\mathcal{M}}^{\theta}(n,\ell),\mathbb{C})=0$, since the variety $\overline{\mathcal{M}}^{\theta}(n,\ell)$ is homotopy equivalent to $\overline{\rho}^{-1}(0)$ (via the contracting $\mathbb{C}^*$-action) and this variety has dimension strictly less than $\frac{1}{2}\mbox{dim}(\overline{\mathcal{M}}^{\theta}(n,\ell))$ as shown in Lemma \ref{Preimage} (b). On the other hand, the point $0$ is a symplectic leaf in affine Poisson varieties $\overline{\mathcal{M}}(n,\ell)$. This is true, since the Poisson bracket is of degree $-2$ and there are no invariant functions of degree one in $\mathbb{C}[T^*\bar{R}]^G$, hence, the maximal ideal of $0$ is Poisson. From this it follows that the vector spaces  $HP_0(\mathcal{O}(\overline{\mathcal{M}}(n,\ell)))$ and are at least $1-$dimensional. Here $HP_0(\mathcal{O}(X))=\mathcal{O}(X)/\{\mathcal{O}(X),\mathcal{O}(X)\}$ stands for the zeroth Poisson homology of an affine Poisson variety $X$. Therefore, $H^{top}(\overline{\mathcal{M}}(n,\ell))\neq HP_0(\overline{\mathcal{M}}(n,\ell))$, contradicting the assertion in part $(a)$ of the conjecture.
\end{rmk}

\section{Symplectic leaves and slices}
\subsection{Symplectic leaves}
\par
First we describe the symplectic leaves and slices to them for the Poisson varieties $\overline{\mathcal{M}}(2,\ell)$ and $\overline{\mathcal{M}}(3,\ell)$. The general description was given by Nakajima, it can also be found in Section $2$ of \cite{BezrLos}. In particular (Section $6$ of \cite{Nak94} or Section $3$ of \cite{Nak}), it was shown that $$\overline{\mathcal{M}}(n,\ell)=\underset{\hat{G}\subseteq G}{\bigcup}\overline{\mathcal{M}}(n,\ell)_{\hat{G}},$$
where the strata are parametrized by reductive subgroups $\hat{G}\subseteq G$ and $\overline{\mathcal{M}}(n,\ell)_{\hat{G}}$ stands for the locus of isomorphism classes of semisimple representations, whose stabilizer is conjugate to $\hat{G}$. A semisimple representation $r \in T^{*}R$ is in  $\overline{\mathcal{M}}(n,\ell)_{\hat{G}}$, if it can be decomposed as $r=r^{0}\underset{i=1}{\overset{n}{\oplus}} r^{i}\otimes U_{i}$, where $r_{i}$'s are simple and pairwise nonisomorphic with zero-dimensional framing and $U_{i}$'s are their multiplicity spaces, and $\hat{G}$ is conjugate to $\prod GL(U_{i})$. Moreover, according to Theorem $1.3$ of \cite{Cr-B}, the stratum $\overline{\mathcal{M}}(n,\ell)_{\hat{G}}$ is an irreducible locally closed subset of $\overline{\mathcal{M}}(n,\ell)$. Each stratum $\overline{\mathcal{M}}(n,\ell)_{\hat{G}}$, being irreducible, must be a symplectic leaf. The information about the symplectic leaves of $\overline{\mathcal{M}}(2,\ell)$ and $\overline{\mathcal{M}}(3,\ell)$  is summarized in the tables below. 
\begin{rmk}
We would like to notice that  there are no irreducible representations with dimension vector $(1,1)$, as each summand $[X_{k},Y_{k}]$ in equation $\eqref{eq:(moment)}$ equals zero and, therefore, $ji=0$ as well, forcing $i=0$ or $j=0$ (or  $i=j=0$) and making the representation with dimension vector $(1,0)$ (zero-dimensional framing) in the former case and with dimension vector $(0,1)$ in the latter a subrepresentation.
\end{rmk}
The third leaf in Table $1$ corresponds to representations $r=r^0\oplus r^1\oplus r^2$, while the fourth $r=r^0\oplus r^1\otimes \mathbb{C}^2$, the multiplicities in Table $2$ are indicated in the second column therein.
\begin{rmk}
	\label{A_1SingSlice}
Since $\overline{\mathcal{M}}(2,\ell)$ has a unique symplectic leaf of codimension $2$, the slice to which is an $A_1$ singularity the Namikawa Weyl group (see \cite{NamMain}) of $\overline{\mathcal{M}}(2,\ell)$ is $\mathbb{Z}/2\mathbb{Z}$. As there are no symplectic leaves of codimension $2$ in $\overline{\mathcal{M}}(3,\ell)$, the corresponding Namikawa Weyl group is trivial.
\end{rmk}
\begin{table}[ht]
	\label{SliceTable}
\begin{center}
\begin{tabular}{ |c|c|c|c| } 
 \hline
 type & dim vector & dim of leaf  & stabilizer (in $GL_{2}$)\\ 
  \hline
 1 & (2,1) & $6\ell-4$ & \{id\}\\ 
  \hline
  2 & $(2,0)\oplus(0,1)$ & $6\ell-6$  & $\mathbb{C}^{*}\cdot$id\\ 
 \hline
3 & $(1,0)\oplus(1,0)\oplus(0,1)$ & $2\ell$  & $\left(\begin{array}{cc}
\lambda &  0  \\
0 &   \mu
\end{array}\right)$, $\lambda, \mu \in \mathbb{C}^{*}$\\ 
 \hline
4 & $(1,0)^{\oplus 2}\oplus(0,1)$ & 0  & $GL_{2}$\\ 
 \hline
\end{tabular}
\caption{Symplectic leaves of $\overline{\mathcal{M}}(2,\ell)$}
\end{center}
\end{table}
\begin{table}[ht]
\begin{center}
\begin{tabular}{ |c|c|c|c| } 
 \hline
 type & dim vector & dim of leaf  & stabilizer (in $GL_{3}$)\\ 
  \hline
 1 & (3,1) & $16\ell-12$ & \{id\}\\ 
  \hline
  2 & $(3,0)\oplus(0,1)$ & $16\ell-16$  & $\mathbb{C}^{*}\cdot$id\\ 
 \hline
3 & $(2,1)\oplus(1,0)$ & $6\ell-4$ & $\left(\begin{array}{ccc}
1 &  0 & 0 \\
0 & 1 &  0 \\
0 & 0 &  \nu
\end{array}\right)$, $ \nu \in \mathbb{C}^{*}$\\ 
 \hline
4 & $(2,0)\oplus(1,0)\oplus(0,1)$ & $6\ell-6$   & $\left(\begin{array}{ccc}
\lambda &  0 & 0 \\
0 & \lambda &  0 \\
0 & 0 &  \mu
\end{array}\right)$, $\lambda, \mu \in \mathbb{C}^{*}$\\ 
 \hline
5 & $(1,0)\oplus(1,0)\oplus(1,0)\oplus(0,1)$ & $4\ell$  & $\left(\begin{array}{ccc}
\lambda &  0 & 0 \\
0 & \nu &  0 \\
0 & 0 &  \mu
\end{array}\right)$, $\lambda, \nu, \mu \in \mathbb{C}^{*}$\\ 
 \hline
6 & $(1,0)^{\oplus 2}\oplus(1,0)\oplus(0,1)$ & $2\ell$  & $\left(\begin{array}{ccc}
* &  * & 0 \\
* & * &  0 \\
0 & 0 &  \mu
\end{array}\right)$, $\mu \in \mathbb{C}^{*}$\\ 
 \hline
7 & $(1,0)^{\oplus 3}\oplus(0,1)$ & 0  & $GL_{3}$\\ 
 \hline
\end{tabular}
\caption{Symplectic leaves of $\overline{\mathcal{M}}(3,\ell)$}
\end{center}
\end{table}
\subsection{Fixed points on the slice}\label{SliceSection}
\par
Next we study the slice taken at some point of the leaf of type $3$ in Table $1$ above. This slice is the quiver variety  on the picture below with $k,s \in \{1,\hdots,\ell-1,\ell+1,\hdots,2\ell-1\}$. The dimension vector is $(1,1)$ and the framing is also one-dimensional.
\par
\begin{center}
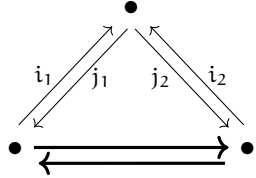
\begin{figure}[h!]
\begin{tikzpicture}
\matrix(m)[matrix of math nodes,
row sep=3em, column sep=2.5em,
text height=1.5ex, text depth=0.25ex]
{& \bullet\\ \bullet & & \bullet\\};
\path[->,font=\scriptsize]
(m-1-2) +(-0.05,-0.3) edge node[right] {$j_{1}$} (m-2-1)
(m-2-1) +(0.05,0.3) edge node[left] {$i_{1}$} (m-1-2)  
(m-1-2) + (0.05,-0.3) edge node[left] {$j_{2}$} (m-2-3)
(m-2-3) +(-0.05,0.3) edge node[right] {$i_{2}$} (m-1-2) +(0.05,0.3)
(m-2-1) edge  [very thick]   (m-2-3) 
(m-2-3)+(-0.3,-0.21) edge [very thick]  (-1.23,-1.15);
\end{tikzpicture}
\caption{Slice quiver with the maps corresponding to thick edges from left to right being $x_{1},\hdots,x_{\ell-1},y_{\ell},\hdots,y_{2\ell-2}$ and from right to left $y_{1},\hdots,y_{\ell-1},x_{\ell},\hdots,x_{2\ell-2}$.}
\label{Slice Quiver}
\end{figure}
\end{center}
\par
We consider the  point $p=(X_{\ell}=\left(\begin{array}{cc}
1 & 0 \\
0 & -1\end{array}\right),X_{\neq \ell}=0, Y_{k}=0,i=0,j=0)$. As the representation is semisimple, the $G$ orbit through $p$ in $T^*\overline{R}$ is closed and slightly abusing notation we will refer to the corresponding point in $\overline{\mathcal{M}}(2,\ell)$ as $p$ as well. The slice to the symplectic leaf at $p$ will be denoted by $\mathcal{SL}_{p}$. The description of slices as quiver varieties can be found in Section $2$ of \cite{BezrLos}.
In our case the slice  $\mathcal{SL}_{p}$ is the hypertoric variety obtained from the $(\mathbb{C}^*)^2$-action on $\mathbb{C}^{2\ell}$. In the basis $\langle x_{1}, x_{2},\hdots, x_{2\ell-2}, i_{1}, i_{2}\rangle$ the weights are  $(t_{1}^{-1}t_{2},\hdots,t_{1}^{-1}t_{2},t_{1}t_{2}^{-1},\hdots,t_{1}t_{2}^{-1}, t_{1}^{-1},t_{2}^{-1})$. It is the quiver variety for the underlying graph depicted on Figure \ref{Slice Quiver} with one-dimensional vector spaces assigned to the vertices and one-dimensional framing. We denote by $\rho_{s}$ the map $\mathcal{SL}_{p}^{\theta} \rightarrow \mathcal{SL}_{p}$ and fix $\theta=(-1,-1)$. The preimage of zero $\rho_{s}^{-1}(0)$  and the fixed points for the $T'\simeq (\mathbb{C}^{*})^{\ell-1}$-action on $\mathcal{SL}_{p}^{\theta}$ are described in the proposition below.
\par
\begin{prop}
$(a)$ $\rho_{s}^{-1}(0)\cong\mathbb{CP}^{2\ell-2}\cup~ \mathbb{CP}^{2\ell-2}$ consists of two irreducible components, intersecting in a single point $(x_{s}=y_{k}=i_{1}=i_{2}=0, j_{1}=j_{2}=1)$.
\par
$(b)$ There are $4\ell-3$ fixed points on $\mathcal{SL}_p^{\theta}$ for the $T'$-action. These points are (the $(\mathbb{C}^*)^2$- orbits of) $(x_{i}=1, x_{\neq i}=y_{j}=i_{1}=i_{2}=j_{2}=0, j_{1}=1)$, $(y_{j}=1, x_{i}=y_{\neq j}=i_{1}=i_{2}=j_{1}=0, j_{2}=1)$ and $(x_{s}=y_{k}=i_{1}=i_{2}=0, j_{1}=j_{2}=1)$.
\end{prop}
\begin{proof}
To see that $(a)$ is true, we first notice that for $(\mathbf{x},\mathbf{y},i,j)\in \rho_s^{-1}(0)$ we have either all $x_{k}=0$ or all $y_{s}=0$ (use the Hilbert-Mumford criterion in a similar way to the proof of Lemma $2.2$). In the former case the stability condition guarantees $j_{2}\neq0$ and $j_{1}$ or at least one of $y_{s}$'s is nonzero. Therefore, the first equation in \eqref{eq:(slicemoment)} below immediately implies that $i_{2}=0$. To see that $i_{1}=0$ as well, notice that the one-dimensional torus, acting on the vector space assigned to the left vertex, acts on $i_{1}$ and $y_{s}$ with $j_{1}$ with opposite weights.  We look at the space $\mathbb{C}^{2\ell-1}\setminus \{0\}$, formed by $y_{s}$'s and $j_{1}$. The $\mathbb{C}^{*}$-action on the one-dimensional framing attached to the right vertex allows to assume $j_{2}=1$. Observing that the action of the remaining $\mathbb{C}^{*}$ simultaneously rescales the vectors in $\mathbb{C}^{2\ell-1}\setminus \{0\}$, we recover the first $\mathbb{CP}^{2\ell-2}$ component in $\rho_s^{-1}(0)$. Similarly, if all $y_{s}=0$, one comes up with $\mathbb{CP}^{2\ell-2}$ with coordinates $x_{k}$ and $j_{2}$. It remains to notice that the projective spaces have exactly one point of intersection, $(x_{s}=y_{k}=i_{1}=i_{2}=0, j_{1}=j_{2}=1)$.
\par
Next we verify the assertion of $(b)$. The moment map equations considered separately for the two vertices are equivalent to
\begin{equation}
 \label{eq:(slicemoment)}
\begin{cases} \sum\limits_{i=1}^{\ell-1} (x_{i}y_{\ell+i} + x_{\ell+i}y_{i})+j_{1}i_{1}=0 \\ j_{1}i_{1} = j_{2}i_{2}.  \end{cases}
 \end{equation}
\par
Recall that $\theta = (-1,-1)$. Then the $\theta$-semistable locus consists of all representations for which at least one of $j_{1}, j_{2}$ is not equal to zero and
\begin{itemize}
	\item if $j_{1}\neq 0$ and $j_2=0$ there exists an  $i$ such that  $x_{i}\neq 0$;
	\item if $j_{2}\neq 0$ and $j_1=0$ there exists a  $j$ such that  $y_{j}\neq 0$.
\end{itemize}
  The formulas for the torus action below are derived from the fact that $x_{i}\in\mbox{Hom}(r_{1},r_{2})$ and $y_{i}\in\mbox{Hom}(r_{2},r_{1})$ are the elements above and below diagonal in the $i$th matrix of our quiver variety, where $r=r_{0}\oplus r_{1}\oplus r_{2}$ is the decomposition of the representation into simples.
\begin{equation}
\label{eq:(sliceFixedPoints)}
\begin{cases} t'_{1}x_{1}=t_1 x_{1}t_2^{-1} \\ \hdots  \\ t'^{-1}_{\ell-1}x_{2\ell-1}=t_1 x_{2\ell-1}t_2^{-1} 
\\t'_{1}y_{1}=t_1^{-1} y_{1}t_2 \\ \hdots \\ t'^{-1}_{\ell-1}y_{2\ell-1}=t_1^{-1} y_{2\ell-1}t_2
\\ j_{1}=t_1^{-1}j_{1} \\  j_{2}=t_2^{-1}j_{2} \\ i_{1}=t_1 i_{1} \\  i_{2}=t_2 i_{2}, 
 \end{cases}
 \end{equation}
\par
here $(t'_1,\hdots,t'_{\ell-1})\in T'$ and $(t_1,t_2)\in (\mathbb{C}^*)^2$. We first notice that it is not possible for both $i_{s}$ and $j_{s}$ to be nonzero ($s\in\{1,2\}$), as otherwise the second equation of  \eqref{eq:(slicemoment)} would imply all $i_{s}$, $j_{s}$ ($s\in\{1,2\}$) were nonzero  and consequently $t_1=t_2=1$, implying all $x_{k}=y_{c}=0$, hence, contradicting  the first equation of  \eqref{eq:(slicemoment)}. It follows from $(a)$ that $i_{1}=i_{2}=0$. From the system of equalities \eqref{eq:(sliceFixedPoints)} it also follows that we must have one of the following
\begin{itemize}
	\item $x_{i}\neq 0, y_{l+i}\neq 0$ and  $y_{i}\neq 0$ with $i \in \{2,\hdots,\ell\}$;
	\item  $x_{l+i}\neq 0$ and $y_{i}\neq 0$ with $i \in \{2,\hdots,\ell\}$;
	\item all $x_{i}$ and all $y_{j}$ are zero with $j_{1}=j_{2}=1$ and $i_{1}=i_{2}=0$.
\end{itemize}

In each of the former two cases  \eqref{eq:(slicemoment)} reduces to either $x_{i}y_{\ell+i}= 0$ or $x_{\ell+i}y_{i} = 0$, then the claim of the proposition easily follows from the description of semistable points.  

\end{proof}
 \begin{rmk}
	\label{SmTFP}
	 The slice $\mathcal{SL}_{p}\subset\overline{\mathcal{M}}^{\theta}(2,\ell)$ is a formal subscheme (formal neighborhood of the point $p$). We describe the intersection of the fixed point loci $\mathcal{SL}^{T'}_{p}\cap\overline{\mathcal{M}}^{\theta}(2,\ell)^{T'}$ (the latter was found in Remark \ref{SmTFPpp}). Each fixed point $(x_{s}=1, x_{\neq s}=y_{j}=i_{1}=i_{2}=j_{2}=0, j_{1}=1)$ on the slice  with $s \in \{1,\hdots,\ell-1\}$ is the fixed point $(X_{\ell}=\left(\begin{array}{cc}
	1 & 0 \\
	0 & -1\end{array}\right), Y_{\ell}=0,$ $ X_{s} = \left(\begin{array}{cc}
	0 & 1 \\
	0 & 0\end{array}\right), X_{\neq s}=Y_k=0, i=0, j=\left(\begin{array}{c}
	1 \\
	0\end{array}\right))$  on $\overline{\mathcal{M}}^{\theta}(2,\ell)^{T'}$; $(y_{s}=1, x_{i}=y_{\neq s}=i_{1}=i_{2}=j_{1}=0, j_{2}=1)$ is $(X_{\ell}=\left(\begin{array}{cc}
	-1 & 0 \\
	0 & 1\end{array}\right), Y_{\ell}=0,$ $ X_{s} = \left(\begin{array}{cc}
	0 & 1 \\
	0 & 0\end{array}\right), X_{\neq s}=Y_k=0, i=0, j=\left(\begin{array}{c}
	1 \\
	0\end{array}\right))$. Notice that these points are respectively the points $(1,0)$ and $(-1,0)$ on $\mathbb{C}^2_s\subset\overline{\mathcal{M}}^{\theta}(2,\ell)^{T'}$ (see Remark \ref{SmTFPpp}). In case $s \in \{\ell+1,\hdots,2\ell\}$ the fixed points on the slice are $(X_{\ell}=\left(\begin{array}{cc}
	1 & 0 \\
	0 & -1\end{array}\right), Y_{\ell}=0,$ $Y_{s}=\left(\begin{array}{cc}
	0 & 1 \\
	0 & 0\end{array}\right),  X_{\neq s}=Y_k=0, i=0, j=\left(\begin{array}{c}
	1 \\
	0\end{array}\right))$ and  $(X_{\ell}=\left(\begin{array}{cc}
	-1 & 0 \\
	0 & 1\end{array}\right), Y_{\ell}=0,$ $ Y_{s} = \left(\begin{array}{cc}
	0 & 1 \\
	0 & 0\end{array}\right),  X_{\neq s}=Y_k=0, i=0, j=\left(\begin{array}{c}
	1 \\
	0\end{array}\right))$, finally,
	$(x_{s}=y_{k}=i_{1}=i_{2}=0, j_{1}=j_{2}=1)$ becomes $(X_{\ell}=\left(\begin{array}{cc}
	1 & 0 \\
	0 & -1\end{array}\right),X_{\neq \ell}=Y_{k}=0, i=0, j=\left(\begin{array}{c}
	1 \\
	1\end{array}\right)) \in T^{*}\mathbb{P}^{1}$.
\end{rmk}


\section{Category $\mathcal{O}_{\xi}(\mathcal{\overline{S}}_{\lambda}(2,\ell))$ for the slice $\mathcal{SL}_{p}$}

The main goal of this section is to provide a description of the category $\mathcal{O}_{\nu}(\mathcal{\overline{SL}}_{\lambda}(2,\ell))$ for the slice $\mathcal{SL}_{p}$. These results will be used in the next section for the study of the category $\mathcal{O}_{\nu}(\mathcal{\overline{A}}_{\lambda}(2,\ell))$. As $\mathcal{SL}_{p}$ is a hypertoric variety, we use the results of \cite{BLPW2} and \cite{BLPW1}, where analogous categories were explicitly described in a more general setting.

We start by briefly recalling the basic definitions, notions and results (for a more detailed exposition see  \cite{BLPW2} and \cite{BLPW1}).
\subsection{Hypertoric varieties (a brief overview)}
Consider the moment map for the action of the torus $K\subset \tilde{T}=(\mathbb{C}^*)^n$ on the variety $T^*\mathbb{C}^n$, i.e. $$\mu: T^*\mathbb{C}^n \rightarrow \mathfrak{k}^*.$$ 
\par
Fix a direct summand $\Lambda_0 \subset W_\mathbb{Z}$, let $W_{\mathbb{R}}:=W_{\mathbb{Z}}\otimes_{\mathbb{Z}}\mathbb{R}, V_{0,\mathbb{R}}=\mathbb{R}\Lambda_0, V_0:=\mathbb{C}\Lambda_0\subset W\cong t^*$, $\mathfrak{k}=V_0^{\perp}$ and $K \subset \widetilde{T}$ be the connected subtorus with Lie algebra $\mathfrak{k}$. Thus $\Lambda_0$ may be identified with the character lattice of $\widetilde{T}/K$ and $W_\mathbb{Z}/\Lambda_0$ may be identified with the character lattice of $K$. 
\begin{defn}
	The \textit{hypertoric variety} associated to the triple $X = (\Lambda_0, \eta, \xi)$ with $\eta$ a $\Lambda_0$-orbit in $W_{\mathbb{Z}}$ is $\mathfrak{M}(X):=\mu^{-1}(0)^{\eta-ss}//K$. Also define $\mathfrak{M}_0(X):=\mu^{-1}(0)//K$. We consider the categorical quotient in both cases. The projective map $\mathfrak{M}(X)\rightarrow \mathfrak{M}_0(X)$ will be denoted by $\kappa$. We will denote the subspace $\eta+V_{0,\mathbb{R}}\subset W_{\mathbb{R}}$ by $V_{\eta}$. The triple $X = (\Lambda_0, \eta, \xi)$ is called a \textit{polarized arrangement}.
\end{defn}

For a sign vector $\alpha\in\{+,-\}^{n}$ define the chamber $P_{\alpha,0}$  to be the subset of the affine space $V_{\nu}:=\{v+\nu~|~v\in V_{0,\mathbb{R}}\}$ cut out by the inequalities
$$h_i \geq 0 \mbox{ for all } i \in I_{\mathbf{\Lambda}} \mbox{ with } \alpha(i) = + \mbox{ and } h_i \leq 0 \mbox{ for all } i \in I_{\mathbf{\Lambda}} \mbox{ with } \alpha(i) = -.$$ If $P_{\alpha,0}\neq \varnothing$ we say that $\alpha$ is \textit{feasible} for $X$ and let $\mathcal{F}_{\eta}$ be the set of feasible sign vectors.

  \begin{rmk}
	The hypertoric variety $\mathfrak{M}_0(X)$ is affine, and for any central character $\lambda$ of the
	hypertoric enveloping algebra $U$ there is a natural isomorphism
	$\mbox{gr}  U_\lambda \simeq \mathbb{C}[\mathfrak{M}_0] \simeq  \mathbb{C}[\mathfrak{M}] $ (Proposition $5.2$ in \cite{BLPW1}).
\end{rmk}
	\par
	Let $\mathbb{S} := \mathbb{C}^*$ act on $T^*\mathbb{C}^n$ by inverse scalar multiplication i.e. $s \cdot (z, w) := (s^{-1}z, s^{-1}w)$. This induces an $\mathbb{S}$-action on both $\mathfrak{M}(X)$ and $\mathfrak{M}_0(X)$, and the map $\kappa$ is $\mathbb{S}$-equivariant. We have that $\kappa:\mathfrak{M}(X)\rightarrow \mathfrak{M}_0(X)$ is a conical symplecic resolution. The symplectic form $\omega$ has weight $2$ w.r.t. the aforemented $\mathbb{S}$-action.
\subsection{Hypertoric category $\mathcal{O}$}
\par
Let $\mathbb{D}$ be the Weyl algebra of polynomial differential operators on $\mathbb{C}^n$, i.e. $$\mathbb{D} =\mathbb{C}\langle x_1,\partial_1,\hdots x_n,\partial_n\rangle,$$ with $[x_i,x_j]=[\partial_i,\partial_j]=0$ and $[\partial_i,x_j]=\delta_{ij}$. The action of the torus $\widetilde{T}=(\mathbb{C}^*)^n$ on $\mathbb{C}^n$ induces an action on $\mathbb{D}$. This provides the $\mathbb{Z}^n$- grading  $$\mathbb{D}=\underset{z\in W_\mathbb{Z}}{\bigoplus}\mathbb{D}_z,$$  where $W_\mathbb{Z}$ is the character lattice of $\widetilde{T}$, deg$(x_i)=-\mbox{deg}(\partial_i)=(0,\hdots,0,\underset{i}{1},0,\hdots,0)$ and $\mathbb{D}_z:=\{a\in \mathbb{D}~|~ t\cdot a=t_1^{z_1}\hdots t_n^{z_n}a ~\forall~ t\in \widetilde{T}\}$.
\par
Observe that the $0^{\mbox{th}}$ graded piece is $\mathbb{D}^{\widetilde{T}}=\mathbb{C}[x_1\partial_1,\hdots, x_n\partial_n]$ and define $h_i^-:=\partial_ix_i$ and   $h_i^+:=x_i\partial_i$ with $h_i^--h_i^+=1$. We consider the Bernstein filtration on  $\mathbb{D}$ (here deg$(x_i)=$ deg$(\partial_j)=1$)  and let $H:=\mbox{gr}(\mathbb{D}_0)=\mathbb{C}[h_1,\hdots,h_n]$, where $h_i:=h_i^++F_0(\mathbb{D}_0)=h_i^-+F_0(\mathbb{D}_0)$.


\begin{defn}
\textit{The hypertoric enveloping algebra associated to $\Lambda_0$} is the ring of $K$-invariants
$U := \mathbb{D}^K =\underset{z\in \Lambda_0}{\bigoplus}\mathbb{D}_z.$
\end{defn}
Consider a module $M\in U\operatorname{-mod}$. For a point $v\in W$, let $\mathcal{J}_v$ denote the corresponding maximal ideal. Then the generalized $v$-weight space of $M$ is defined as $$M_v:=\{m\in M~|~\mathcal{J}^k_v m=0 \mbox{ for } k\gg 0 \}.$$

The support of $M$ is defined by $$\mbox{Supp }M:=\{v \in W~|~M_v\neq  0 \}.$$
We will use the notation $U\operatorname{-mod}_{\Lambda}$ for $M \in U\operatorname{-mod}$ with $\mbox{Supp }M\subset \Lambda$.
\par
Let $Z(U)$ denote  the center of $U$. It is not hard to show that $Z(U)$  is the subalgebra isomorphic to the image of $S[\mathfrak{k}]$ under the quantum comoment map (Section $3.2$ of \cite{BLPW1}). Let $\lambda: Z(U)\rightarrow \mathbb{C}$ be a central character. Notice that the isomorphism $Z(U)\simeq S[\mathfrak{k}]$ allows to think of $\lambda$  as an element of $\mathfrak{k}^*$. We will denote by $U_\lambda:=U/\langle ker(\lambda)\rangle U$ the corresponding central quotient. Set $V_{\lambda} := \lambda + V_0 = \lambda + \mathbb{C}\Lambda_0, V_{\lambda,\mathbb{R}} :=  \lambda + \mathbb{R}\Lambda_0$ and let $\mathbf{\Lambda}$ be a $\Lambda_0$-orbit.
\par
Choose a generic element $\xi \in \Lambda_0^*\simeq (t/\mathfrak{k})^*$, the action of $\xi$ lifts to $U$ and produces a grading given by $$U:=\underset{\xi(z)=k}{\bigoplus}U_z.$$ 
\par
Set $$U^+:=\underset{k\geq0}{\bigoplus}U^k \mbox{ and } U^-:=\underset{k\leq0}{\bigoplus}U^k,$$
similarly, $U^+_\lambda$ and $U^-_\lambda$ are the images of $U^+$ and $U^-$ under the quotient map $U \rightarrow U_\lambda$.
\begin{defn}
The \textit{ hypertoric category} $\mathcal{O}$ is the full subcategory of $U$-mod consisting
of modules that are $U^+$ - locally finite and semisimple over the center $Z(U)$. Define $\mathcal{O}_{\lambda}$
to be the full subcategory of $\mathcal{O}$ consisting of modules on which $U$ acts with central character $\lambda$. Equivalently, it is as the full subcategory of $U_\lambda$-mod consisting of modules that are $U^+_\lambda$ - locally finite. Finally, define $\mathcal{O}(\Lambda_0, \Lambda, \xi)$ to be the full subcategory of $\mathcal{O}_{\lambda}$ consisting of modules supported in $\Lambda$; equivalently, the full subcategory of $U_\lambda\operatorname{-mod}_{\Lambda}$ consisting of modules that are $U^+_\lambda$ - locally finite. The triple $\mathbf{X}:=(\Lambda_0, \Lambda, \xi)$ is called a \textit{quantized polarized arrangement}.
\end{defn}
\par
Similarly to category $\mathcal{O}$ of a semisimple Lie algebra, we have the direct sum decompositions 
\begin{equation}
\label{BlockDecomp}
\begin{aligned}
\mathcal{O} = \underset{\Lambda \in W/\Lambda_0}{\bigoplus}
\mathcal{O}(\Lambda_0, \Lambda, \xi) \mbox{ and }\\
\mathcal{O}_{\lambda} = \underset{\Lambda' \in V_{\lambda}/\Lambda_0}{\bigoplus}
\mathcal{O}(\Lambda_0, \Lambda', \xi).
\end{aligned}
\end{equation}
The summands in the decompositions above are blocks, i.e. they are the smallest possible direct summands (see Section $4.1$ of \cite{BLPW1} for details).
\par
Let $I_{\mathbf{\Lambda}}$ be the set of indices $i \in \{1,\hdots, n\}$ for which $h^+_i (\Lambda) \subset \mathbb{Z}$ (or equivalently $h^-_i (\Lambda) \subset \mathbb{Z}$).
For a sign vector $\alpha\in\{+,-\}^{n}$ define the chamber $P_\alpha$  to be the subset of the affine space $V_{\lambda}:=\{v+\lambda~|~v\in V\}$ cut out by the inequalities
$$h^+_i \geq 0 \mbox{ for all } i \in I_{\mathbf{\Lambda}} \mbox{ with } \alpha(i) = + \mbox{ and } h^-_i \leq 0 \mbox{ for all } i \in I_{\mathbf{\Lambda}} \mbox{ with } \alpha(i) = -.$$
If $P_\alpha\cap \Lambda$ is nonempty, we say that $\alpha$ is \textit{feasible} for $\Lambda$. We call $\alpha$ \textit{bounded for $\xi$} if the restriction of $\xi$ is proper and bounded above on $P_\alpha$. The set of feasible sign vectors will be denoted by $\mathcal{F}_{\Lambda}$, the set of bounded vectors by $ \mathcal{B}_{\xi}$ and the set of bounded feasible vectors by $\mathcal{P}_{\Lambda,\xi}:=\mathcal{F}_{\Lambda}\cap\mathcal{B}_{\xi}$.
\begin{ex}
	In case $\ell=2$, the slice  $\mathcal{SL}_{p}$ is the hypertoric variety obtained from the $K=(\mathbb{C}^*)^2$-action on $\mathbb{C}^4$ via $$t\cdot (x_{1}, x_{2}, i_{1}, i_{2})=(t_{1}^{-1}t_{2}x_1,t_{1}t_{2}^{-1}x_2,t_{1}^{-1}i_1,t_{2}^{-1}i_2).$$ Notice that $\mathfrak{k}\hookrightarrow Lie(\widetilde{T})$ and the image is span$((-1,1,-1,0),(1,-1,0,-1))$, set $L:=\mbox{span}_{\mathbb{R}}((-1,1,-1,0),(1,-1,0,-1))$. Then $V_{0,\mathbb{R}}=\mbox{span}_{\mathbb{R}}((1,1,0,0),(0,1,1,-1))$ (the subspace of $W_{\mathbb{R}}$ orthogonal to $L$) and we consider  $\Lambda_0=V_{0,\mathbb{R}}\cap W_{\mathbb{Z}}$ and the central character $\lambda: S[\mathfrak{k}]\rightarrow \mathbb{C}$ defined by $\lambda(t_1,t_2)=(\tilde{\lambda},\tilde{\lambda})$ for $\tilde{\lambda} \in \mathbb{C}$.  We take $\eta=(-1,-1)$ to be the restriction of the character $\theta$ of $G$.


	Then $V_{\lambda}$ is cut out in $W$ (or $V_{\lambda,\mathbb{R}}$ inside $W_{\mathbb{R}}$) by the following equations:
	
	$$\begin{cases} -x_{1}+x_{2}-i_{1}=\tilde{\lambda} \\ \quad x_{1}-x_{2}-i_{2}=\tilde{\lambda} \end{cases},$$
	equivalently,
	$$\begin{cases} \; i_{1}+i_{2}=-2\tilde{\lambda} \\ x_{1}-x_{2}=i_{2}+\tilde{\lambda} \end{cases}.$$
	This is a $2$-dimensional affine subspace of $W$. We identify $V_{\lambda}$ with $\mathbb{C}^2$ (or $V_{\lambda,\mathbb{R}}$ with $\mathbb{R}^2$) by choosing the origin of $V_{\lambda}$ to be the point $(0,0,-\tilde{\lambda},-\tilde{\lambda})$ and the basis $u_1:=(1,1,0,0), u_2:=(0,1,1,-1)$. Next we pick a one-parameter subgroup $\xi=(2,1)$. In case $\tilde{\lambda} \in \mathbb{Z}$, we have  $\mathcal{P}_{\lambda,\xi}=\{+---,-+--,----,---+,--+-\}$ (see Figure \ref{ChambersL2}).
	\begin{center}
		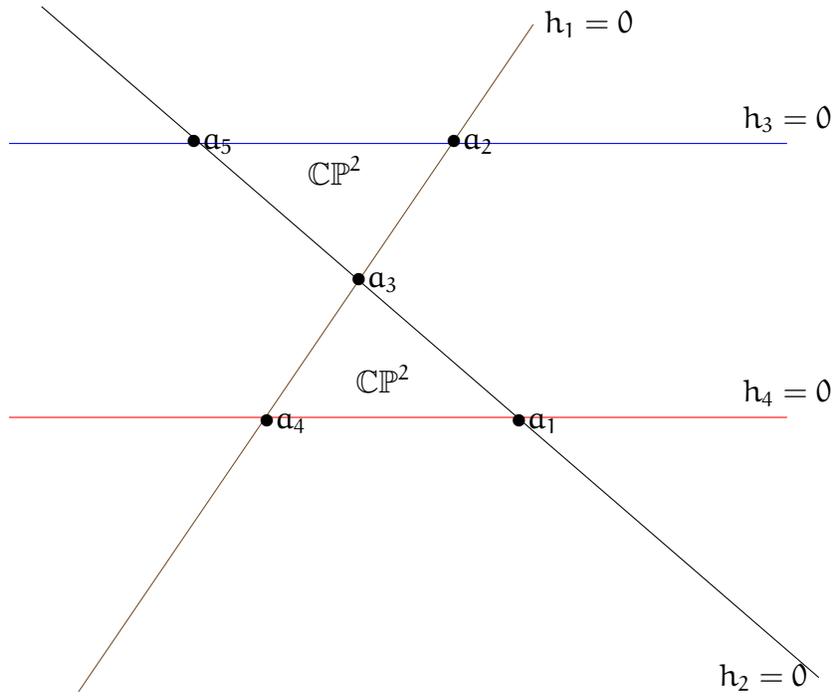
\begin{figure}
			\begin{tikzpicture}
			\begin{axis}[scale=1.6,axis lines=none, ticks=none,xmin=-11, xmax=15, ymin=-10,ymax=15,
			samples =2,no marks]
			\addplot+[domain=-13:13.5] {10} node[above,text=black] {$h_3=0$};
			\addplot+[domain=-13:13.5] {0} node[above,text=black] {$h_4=0$};
			\addplot+[domain=-13:5.5] {1.7*x+5} node[right,text=black] {$h_1=0$};
			\addplot+[domain=-10:14.5] {5-x} node[left,text=black] {$h_2=0$} ;  
			\end{axis}

			\node at (1.67*2.17,1.55*2.3){$\bullet a_4$};
			\node at (1.6*3.03,1.6*3.4){$\bullet a_3$};
			\node at (1.6*3.82,1.6*4.55){$\bullet a_2$};
			\node at (1.66*1.6,1.6*4.55){$\bullet a_5$};
			\node at (1.6*4.36,1.55*2.3){$\bullet a_1$};

			\node at (1.6*3.1,1.6*2.6){$\mathbb{CP}^2$};
			\node at (1.6*2.7,1.6*4.33){$\mathbb{CP}^2$};
			\end{tikzpicture}
			\caption{Polarized arrangement for $\ell=2$}
			\label{Polarized arrangement}
		\end{figure}
	\end{center}
\end{ex}
\begin{center}
	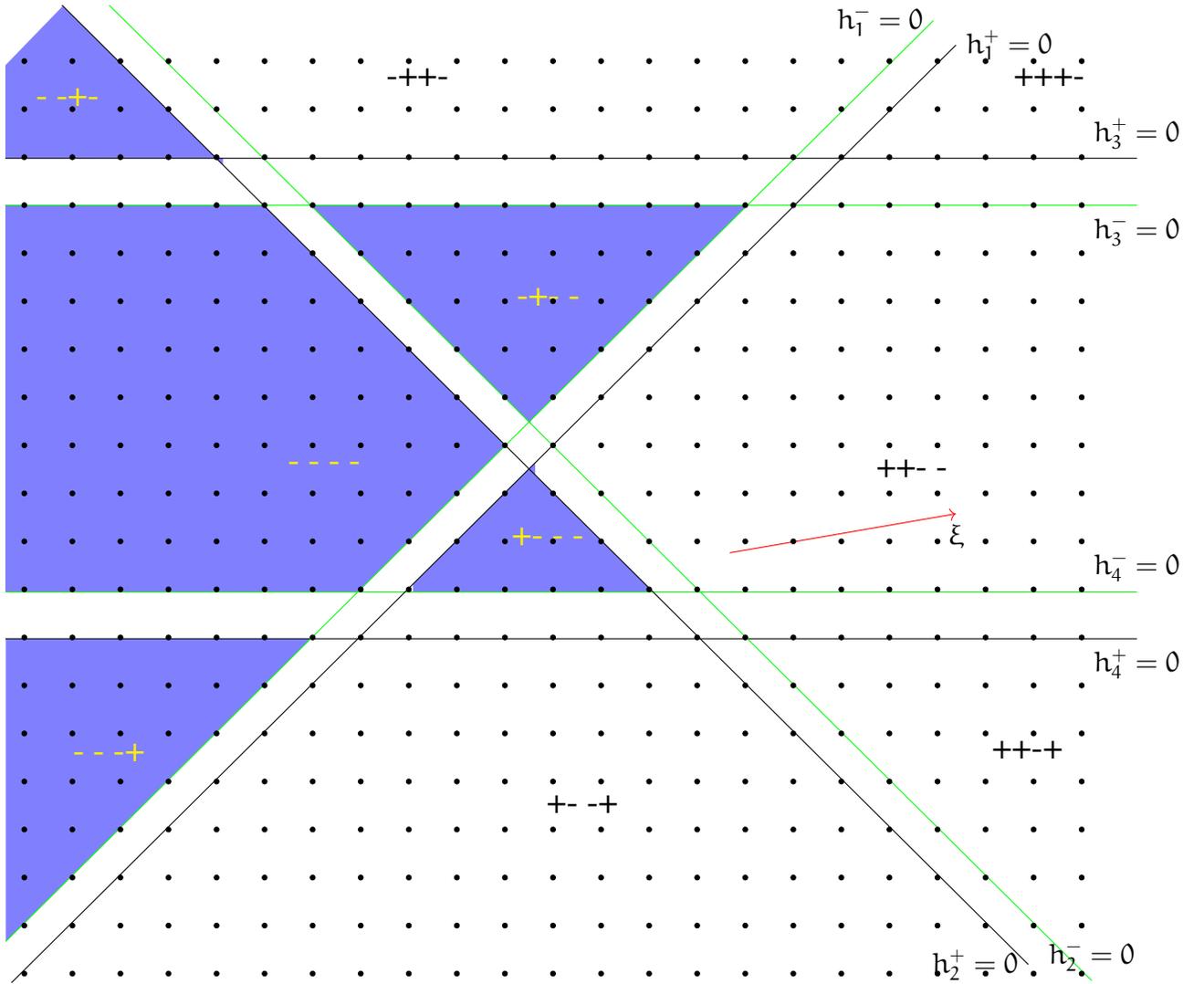
\begin{figure}
		\label{ChambersL2}
		\begin{tikzpicture}
		\begin{axis}[scale=2.5,axis lines=none,ticks=none,xmin=-11, xmax=15, ymin=-10,ymax=15,
		, samples =2,,no marks]
		\addplot[name path=C1,domain=-13:14,green] {9.9} node[below,text=black] {$h^-_3=0$};
		\addplot[name path=C2,domain=-13:14,black] {11.1} node[above,text=black] {$h^+_3=0$};
		\addplot[name path=D2,domain=-13:14,black] {-1.2} node[below,text=black] {$h^+_4=0$};
		\addplot[name path=D1,domain=-13:14,green] {0} node[above,text=black] {$h^-_4=0$};
		\addplot[name path=B2,domain=-10.5:11.6,black] {3.8-1.15*x}node[left,text=black] {$h^+_2=0$};
		\addplot[name path=B1,domain=-10:13,green]{5-1.15*x}node[above,text=black] {$h^-_2=0$};
		\addplot[name path=A1,domain=-11:9.5,green] {3.7+1.15*x}node[left,text=black] {$h^-_1=0$};
		\addplot[name path=A2,domain=-11:10,black] {2.5+1.15*x}node[right,text=black] {$h^+_1=0$};
		\addplot[->,domain=5:10,red] {0.2*x} node[below,text=black] {$\xi$};

		\addplot[blue!50]fill between[of=D1 and C1, soft clip={domain=-11:-5.25}];
		\addplot[blue!50]fill between[of=A1 and D2, soft clip={domain=-11:-4.25}];
		\addplot[blue!50]fill between[of=D1 and B2, soft clip={domain=-5.3:-3.14}];       \addplot[blue!50]fill between[of=A1 and B2, soft clip={domain=-3.15:0}];
		\addplot[blue!50]fill between[of=B1 and C1, soft clip={domain=-4.3:0.6}];
		\addplot[blue!50]fill between[of=A1 and C1, soft clip={domain=0.56:5.4}];
		\addplot[blue!50]fill between[of=B2 and C2, soft clip={domain=-115:-6.2}];
		\addplot[blue!50]fill between[of=A2 and D1, soft clip={domain=-2:0.7}];
		\addplot[blue!50]fill between[of=B2 and D1, soft clip={domain=0.69:3.2}];
		\end{axis}
		\node at (1.5,1.6*2.1){\large{{\color{yellow}{- - -+}}}};
		\node at (1.6*2.9,1.6*4.7){\large{{\color{yellow}{- - - -}}}};;
		\node at (2.4*3.5,2.6){\large{+- -+}};
		\node at (2.4*6.2,3.4){\large{++-+}};
		
		\node at (2.43*3.25,2.5*2.6){\large{{\color{yellow}{+- - -}}}};
		\node at (4*3.3,3*2.5){\large{++- -}};
		\node at (4*3.8,3*4.4){\large{+++-}};
		\node at (6,3*4.4){\large{-++-}};
		\node at (0.9,3*4.3){\large{{\color{yellow}{- -+-}}}};
		\node at (3.2*2.47,2.5*4){\large{{\color{yellow}{-+- -}}}};
		
		\foreach \x in {0,1,...,22}{
			\foreach \y in {0,1,...,19}{
				\node[draw,circle,inner sep=.7pt,fill] at (.7*\x+0.27,.7*\y+0.13) {};
			}
		}
		
		\end{tikzpicture}
		\caption{Chambers and sign vectors for  $\ell=2$}
		\label{QPolarized arrangement}
	\end{figure}
\end{center}

\begin{rmk}
    If $\alpha \in \mathcal{P}_{\Lambda,\xi}$ and  $\xi$ is a generic character then the differential of $\xi$ attains its maximal value at a single point of $P_\alpha$. This point will be denoted by $a_{\alpha}$. It is the intersection of $\mbox{dim}(V_{\lambda})$ hyperplanes from
    $$h^+_i = 0 \mbox{ for all } i \in I_{\mathbf{\Lambda}} \mbox{ with } \alpha(i) = + \mbox{ and } h^-_i = 0 \mbox{ for all } i \in I_{\mathbf{\Lambda}} \mbox{ with } \alpha(i) = -.$$
     Let $C_\alpha$ be the unique polyhedral cone cut out in $V_{\lambda}$  by $\mbox{dim }V_{\lambda}$ inequalities 
     $$h^+_i \geq 0 \mbox{ for all } i \in I_{\mathbf{\Lambda}}, h^+_i(a_{\alpha})=0  \mbox{ with } \alpha(i) = + \mbox{ and }$$ $$ h^-_i \leq 0 \mbox{ for all } i \in I_{\mathbf{\Lambda}}, h^-_i(a_{\alpha})=0 \mbox{ with } \alpha(i) = -.$$  Notice that $P_\alpha\subset C_\alpha$ and the differential of $\xi$ is negative on the extremal rays of $C_\alpha$.
\end{rmk}



Next we describe the standard objects of $\mathcal{O}(\mathbf{X})$.
For any sign vector $\alpha\in\{+,-\}^{I_{\mathbf{\Lambda}}}$, consider the $\mathbb{D}$-module $$\triangle_\alpha := \mathbb{D}/I_{\alpha},$$
where $I_{\alpha}$ is the left ideal generated by the elements
\begin{flalign*}
&\bullet \partial_i, h_i^+(a_{\alpha})=0,\\
&\bullet x_i, h_i^-(a_{\alpha})=0,\\
&\bullet h^+_i -h^+_i (a_{\alpha}), i\not\in I_{\mathbf{\Lambda}}. 
\end{flalign*}
Define $\triangle_\alpha^\Lambda:=\underset{v\in \mathbb{\Lambda}}{\bigoplus} (\triangle_\alpha )_v$, then the standard objects of $\mathcal{O}(\Lambda_0, \Lambda, \xi)$ are $\triangle_\alpha^\Lambda$ for $\alpha \in \mathcal{P}_{\Lambda,\xi}$   (see Section $4.4$ in \cite{BLPW1}). Let  $S_\alpha^\Lambda$ denote the unique simple quotient of $\triangle_\alpha^\Lambda$.
\par
We will need one more definition.
 \begin{defn}
 \label{LinkedArrangements}
	The quantized polarized arrangement $\mathbf{X}=(\Lambda_0,\Lambda,\xi)$ and polarized arrangement $X=(\Lambda_0,\eta,\xi)$ are said to be \textit{linked} if $\pi(\mathcal{F}_{\Lambda})=\mathcal{F}_{\eta}$ for the projection $\pi: \{1,\hdots,n\}\rightarrow I_{\Lambda}$.
\end{defn}

\begin{rmk}
   	The hypertoric category $\mathcal{O}_{\lambda}$ is a category $\mathcal{O}_{\xi}$ for $U_\lambda$ in the sense of Definition \ref{CatO}.
\end{rmk} 
\begin{rmk}
	\label{highestWeight}
	If $\mathbf{X}=(\Lambda_0, \Lambda, \xi) $ is regular, then the category $\mathcal{O}(\mathbf{X}) $ is highest weight and Koszul (see Definition $2.10$ and Corollary $4.10$ in \cite{BLPW1}).
\end{rmk}

\subsection{Hypertoric category $\mathcal{O}$ for the slice $\mathcal{SL}_{p}$}
The slice  $\mathcal{SL}_{p}$ is the hypertoric variety obtained from the $K=(\mathbb{C}^*)^2$-action on $\mathbb{C}^{2\ell}$ via 
$$t\cdot (x_{1}, \hdots, x_{\ell-1}, x_{\ell}, \hdots, x_{2\ell-2}, i_{1}, i_{2})=(t_{1}^{-1}t_{2}x_1,\hdots, t_{1}^{-1}t_{2}x_{\ell-1}, t_{1}t_{2}^{-1}x_{\ell},\hdots,t_{1}t_{2}^{-1}x_{2\ell-2}t_{1}^{-1}i_1,t_{2}^{-1}i_2),$$ it can be also viewed as a quiver variety (see Figure \ref{Slice Quiver}). This is the toric variety $\mathfrak{M}(X)$ for the polarized arrangement $X=(\Lambda_0, \eta,\xi)$. Let $\Lambda_0=V_{0,\mathbb{R}}\cap W_{\mathbb{Z}}$ for $V_{0,\mathbb{R}}=\mbox{span}_{\mathbb{R}}(u_1,\hdots,u_{2\ell-2})$ (where the vectors $u_1,\hdots,u_{2\ell-2}$ are defined below) and the same character $\lambda$ and same $\eta$ as for $\ell=2$ above, then $V_{\lambda}$ is the affine subset of $W$ given by

$$\begin{cases} -\sum\limits_{k=1}^{\ell-1}x_{k}+\sum\limits_{k=1}^{\ell-1}x_{\ell-1+k}-i_{1}=\tilde{\lambda} \\ \quad \sum\limits_{k=1}^{\ell-1}x_{k}-\sum\limits_{k=1}^{\ell-1}x_{\ell-1+k}-i_{2}=\tilde{\lambda} \end{cases},$$
equivalently,
$$\begin{cases} \; i_{1}+i_{2}=-2\tilde{\lambda} \\\sum\limits_{k=1}^{\ell-1}x_{k}-\sum\limits_{k=1}^{\ell-1}x_{\ell-1+k}=i_{2}+\tilde{\lambda} \end{cases}.$$
 This is a $2\ell-2$-dimensional affine subspace of $W$. We choose the origin to be the point $(0,\hdots,0,0,-\tilde{\lambda},-\tilde{\lambda})$ and the basis 
 \begin{flalign*}
 &u_1:=(1,-1,0,\hdots,0,0)\\
 &u_2:=(1,0,-1,\hdots,0,0)\\
 &~~~~~~~~~~~ \hdots\\
 &u_{\ell-1}:=(1,0,\hdots,-1,0,\hdots,0,0)\\
 &u_{\ell}:=(1,0,\hdots,0,1,0,\hdots,0,0)\\
 &u_{\ell+1}:=(1,0,\hdots,0,0,1,0,\hdots,0,0)\\
 &u_{2\ell-3}:=(1,0,\hdots,0,1,0,0)\\
 &u_{2\ell-2}:=(0,\hdots,0,1,1,-1). 
 \end{flalign*}
  One convenient choice of the character is  $\xi=(1,\hdots,\ell-2,\ell,\hdots,2(\ell-1),\ell-1)$.
\par 
\begin{defn}
 The algebra $\mathcal{\overline{S}}_{\lambda}(2,\ell)$ will stand for the quantization of the slice $\mathcal{SL}_{p}$ with period $(\lambda+\frac{1}{2},\lambda+\frac{1}{2})$. 
\end{defn}

 \begin{rmk}
 	\label{LambdaRestriction}
 	The restriction of the quantization $\overline{\mathcal{A}}_{\lambda}(2,\ell)$ to $\mathcal{SL}_{p}$ is $\mathcal{\overline{S}}_{\lambda}(2,\ell)$. This is true since the map $\hat{r}$ from Section $5.4$ of \cite{BezrLos} sends $\lambda$ to $(\lambda,\lambda)$. Indeed, $\hat{r}(\lambda)=r(\lambda-\zeta)+\tilde{\zeta}$, where $\zeta=-\frac{1}{2}$ is the character for the action of $G$ on $\Lambda^{top}{\bar{R}}$, $\tilde{\zeta}=(-\frac{1}{2},-\frac{1}{2})$ is the character for the action of $K$ on $\Lambda^{top}{\mathbb{C}^{2\ell}}$ and $r(\nu)=(\nu,\nu)$ is the restriction.
 \end{rmk}
\begin{prop}
	Pick a central character $\lambda: Z(U)\rightarrow \mathbb{C}$ with $\tilde{\lambda} \in (-\infty;1-\ell)\cup (\ell-2;\infty)$, let $\tilde{\Lambda}:=\{v\in W_{\mathbb{Z}} ~|~ h^+_{2\ell-1}(v)+h^+_{2\ell}(v)=-2\tilde{\lambda},~ \sum\limits_{k=1}^{\ell-1}h^+_{k}(v)-\sum\limits_{k=1}^{\ell-1}h^+_{\ell-1+k}(v)=h^+_{2\ell}(v)+\tilde{\lambda} \}$. The quantized polarized arrangement $\mathbf{X}=(\Lambda_0,\tilde{\Lambda},\xi)$ is regular. 
\end{prop}
\begin{rmk}
	\label{AgreementOnRegLambda}
Henceforth, unless stated explicitly otherwise, we work with $\lambda$ with corresponding  $\tilde{\lambda}$ regular. 	
\end{rmk}
\begin{rmk}
\label{AbLocHypT}
Abelian localization holds for the algebra $\mathcal{\overline{S}}_{\lambda}(2,\ell)$  for  $\lambda<1-\ell$ (it is easy to see that $\mathcal{F}_{\Lambda}=\mathcal{F}_{\Lambda+r\eta}$ with $r\in \mathbb{Z}_{\geq 0}$, so the conditions of Theorem $6.1$ in \cite{BLPW1} are met).
\end{rmk}
\begin{prop}
\label{SignVectors}
Pick a central character $\lambda: Z(U)\rightarrow \mathbb{C}$ with $\tilde{\lambda} \in \mathbb{Z}_{<0}+1-\ell$, let $\mathbf{X}=(\Lambda_0,\tilde{\Lambda},\xi)$ be the quantized polarized arrangement. Assume, in addition, $\mathbf{X}$ is linked to $X$ (see Definition \ref{LinkedArrangements}).
\par
\begin{enumerate}
	\item[(a)]
There is an equivalence of categories $\mathcal{O}_{\xi}(\mathcal{\overline{S}}_{\lambda}(2,\ell))=\mathcal{O}(\mathbf{X})$.
\item[(b)] The set of feasible bounded vectors $\mathcal{P}_{\Lambda,\xi}$ consists of the following $4\ell-3$ sign vectors (notice that the sign vector $\alpha_{mid}=\underset{\ell-1}{\underbrace{---\hdots- }} \underset{\ell-1}{\underbrace{---\hdots-}}--$ appears in both sets below for convenience but is counted once only)

$$2\ell-1 \begin{cases}\alpha_{2\ell-1}=\underset{\ell-1}{\underbrace{-\hdots---}}\underset{\ell-1}{\underbrace{+\hdots++-}}-+\\ 
\alpha_{2\ell-2}=\underset{\ell-1}{\underbrace{-\hdots---}}\underset{\ell-1}{\underbrace{+\hdots+--}}-+\\
 \hdots\\
 \alpha_{\ell+1}=\underset{\ell-1}{\underbrace{-\hdots--- }} \underset{\ell-1}{\underbrace{-\hdots---}}-+\\  
 \alpha_{mid}=\underset{\ell-1}{\underbrace{-\hdots--- }} \underset{\ell-1}{\underbrace{-\hdots---}}--\\
\alpha_{\ell-1}=\underset{\ell-1}{\underbrace{-\hdots--+ }} \underset{\ell-1}{\underbrace{-\hdots---}}--\\
\alpha_{\ell-2}=\underset{\ell-1}{\underbrace{-\hdots-++ }} \underset{\ell-1}{\underbrace{-\hdots---}}--\\ \hdots\\
\alpha_{1}=\underset{\ell-1}{\underbrace{+\hdots+++ }} \underset{\ell-1}{\underbrace{-\hdots---}}--\\ \end{cases},$$

$$2\ell-1 \begin{cases}\beta_{2\ell-1}=\underset{\ell-1}{\underbrace{-++\hdots+}}\underset{\ell-1}{\underbrace{---\hdots-}}+-\\ 
\beta_{2\ell-2}=\underset{\ell-1}{\underbrace{--+\hdots+}}\underset{\ell-1}{\underbrace{---\hdots-}}+-\\ 
\hdots\\
\beta_{\ell+1}=\underset{\ell-1}{\underbrace{---\hdots- }}\underset{\ell-1}{\underbrace{---\hdots-}}+-\\  
\alpha_{mid}=\underset{\ell-1}{\underbrace{---\hdots- }} \underset{\ell-1}{\underbrace{---\hdots-}}--\\
\beta_{\ell-1}=\underset{\ell-1}{\underbrace{---\hdots- }} \underset{\ell-1}{\underbrace{+--\hdots-}}--\\
\beta_{\ell-2}=\underset{\ell-1}{\underbrace{---\hdots- }} \underset{\ell-1}{\underbrace{++-\hdots-}}--\\ \hdots\\
\beta_{1}=\underset{\ell-1}{\underbrace{---\hdots-}} \underset{\ell-1}{\underbrace{+++\hdots+}}--\\ \end{cases}.$$

\item[(c)] The simple and standard objects in the category $\mathcal{O}_{\xi}(\mathcal{\overline{S}}_{\lambda}(2,\ell))$ are  indexed by the  sign vectors in (b). We have the short exact sequences $0\rightarrow S_{\alpha_{i+1}}^{\Lambda}\rightarrow\Delta_{\alpha_i}^{\Lambda}\rightarrow S_{\alpha_i}^{\Lambda}\rightarrow 0$   (resp. $0\rightarrow S_{\beta_{i+1}}^{\Lambda}\rightarrow\Delta_{\beta_i}^{\Lambda}\rightarrow S_{\beta_i}^{\Lambda}\rightarrow 0$ ). The socle filtration of $\Delta_{\alpha_{mid}}^{\Lambda}$ has subquotients $S_{\alpha_{mid}}^{\Lambda}, S_{\alpha_{\ell+1}}^{\Lambda}$ and $S_{\beta_{\ell+1}}^{\Lambda}$. Finally, if $1\leq i<\ell$, we have that $\Delta_{\alpha_i}^{\Lambda}$ (resp. $\Delta_{\beta_i}^{\Lambda}$) have a socle filtration with subquotients $S_{\alpha_i}^{\Lambda}, S_{\alpha_{i+1}}^{\Lambda}$, $S_{\beta_{2\ell-i}}^{\Lambda}$ and $S_{\beta_{2\ell-i+1}}^{\Lambda}$ (resp. $S_{\beta_i}^{\Lambda}, S_{\beta_{i+1}}^{\Lambda}$, $S_{\alpha_{2\ell-i}}^{\Lambda}$ and $S_{\alpha_{2\ell-i+1}}^{\Lambda}$).

  
\item[(d)]
We have dim(Hom($\Delta_\gamma^{\Lambda},\Delta_{\alpha}^{\Lambda}))=1$, if $S_\gamma^{\Lambda}$ appears as a subquotient in filtration of $\Delta_{\alpha}^{\Lambda}$ and dim(Hom($\Delta_\beta^{\Lambda},\Delta_{\alpha}^{\Lambda}))=0$ otherwise as determined in (c).
\end{enumerate}

\end{prop}
 \begin{proof}
Since $\Lambda_0$ is unimodular and $\mathbf{X}$ is integral, i.e. $\Lambda\subset W_{\mathbb{Z}}$,  $(a)$ follows from Remark $4.2$ of \cite{BLPW1}. To determine the sign vector $\alpha$ of each chamber, we first notice that $\xi$ is maximized at one of the vertices. The vertex is formed by the intersection of $2\ell-2$ hyperplanes in the arrangement (see Table \ref{Chambers}). The corresponding $2\ell-2$ coordinates of $\alpha$ are derived from the decomposition of $\xi$ in terms of the normal vectors to the $2\ell-2$ hyperplanes (in Table \ref{HyperplNormV} the direction of each normal vector $\eta_i$ is chosen so that the corresponding coordinate $x_i$ increases along $\eta_i$). There is a unique way to choose the polyhedral cone $C_{\alpha}$ so that the dot product of any vector inside the cone with $\xi$ is negative. The remaining two coordinates are determined by the coordinates of the vertex itself (see Tables \ref{HyperplNormV} and \ref{Chambers}).
\par
We proceed with verifying the assertions in (c) and (d). The appearance of $S_{\gamma}^{\Lambda}$ in the composition series of $\triangle_{\alpha}^{\Lambda}$ is equivalent to the containment $P_{\gamma}\subset C_{\alpha}$ (see Proposition $4.15$ in \cite{BLPW1}). This, in turn, means that the $2\ell-2$ coordinates of the sign vectors $\gamma$ and $\alpha$ corresponding to the defining hyperplanes of $a_{\alpha}$ coincide (here $a_{\alpha}$ is the point of maximum of $\xi$ on the chamber $P_{\alpha}$). The result follows from the explicit description provided in Table \ref{Chambers}. 

 \end{proof}
\begin{prop}
	\label{SignVectors2}
	If $\tilde{\lambda} \in \mathbb{Z}_{<0}+\frac{3}{2}-\ell$, we have  $\mathcal{O}_{\xi}(\mathcal{\overline{S}}_{\lambda}(2,\ell))=\overset{\ell-1}{\underset{i=1}{\bigoplus}}( \mathcal{O}_{\{\alpha_i,\beta_{2\ell-i}\}}\oplus \mathcal{O}_{\{\beta_i,\alpha_{2\ell-i}\}})\oplus \mathcal{O}_{\alpha_{mid}}$. Each block of the form $\mathcal{O}_{\{a,b\}}$ is equivalent to the principal block $\mathcal{O}_0$ in the BGG category $\mathcal{O}$ for $\mathfrak{sl}_2$.  In case $\tilde{\lambda}\not\in \frac{\mathbb{Z}}{2}$, the category $\mathcal{O}_{\xi}(\mathcal{\overline{S}}_{\lambda}(2,\ell))$ is semisimple.
\end{prop}
\begin{proof}
    According to the general result on block decompositions of hypertoric categories $\mathcal{O}$ (see  \eqref{BlockDecomp}), it is sufficient to notice that the partition of points  $a_{\gamma}$ corresponding to sign vectors  $\gamma$ according to $\Lambda_0$-orbits in which they lie, is the same as for corresponding sign vectors in the proposed block decompositions (see Table \ref{HyperplNormV}).
\end{proof}
\begin{ex}
	We illustrate the results for $\ell=2$ (see also Figure \ref{ChambersL2}). In case $\tilde{\lambda} \in \mathbb{Z}$ we have  $\mathcal{P}_{\lambda,\xi}=\{1=+---,2=-+--,3=----,4=---+\mbox{ and } 5=--+-\}$. The standards  in $\mathcal{O}_{\nu}(\mathcal{\overline{S}}_{\lambda}(2,2))$ are filtered as shown in the table below. 
	\begin{table}[ht]
		\begin{center}
			\begin{tabular}{ |c|c|c|c|c| } 
				\hline
				$\Delta_1$ & $\Delta_2$ & $\Delta_3$ & $\Delta_4$ & $\Delta_5$\\ 
				\hline
				$S_1$ & $S_2$ & $S_3$ &$S_4$ & $S_5$\\ 
				\hline
				$S_3$ & $S_3$ & $S_4$  & &  \\ 
				\hline
				$S_5$ &  $S_4$ & $S_5$   & & \\ 
				\hline
			\end{tabular}
			\caption{Multiplicities of simples in standards for $\ell=2$}
		\end{center}
	\end{table}
    In case $\tilde{\lambda} \in \mathbb{Z}+\frac{1}{2}$, we have $\mathcal{O}_{\xi}(\mathcal{\overline{S}}_{\lambda}(2,2))= \mathcal{O}_{\{1,5\}}\oplus\mathcal{O}_{\{2,4\}}\oplus \mathcal{O}_{3}$. Finally, if $\tilde{\lambda} \not\in \frac{\mathbb{Z}}{2}$, the category $\mathcal{O}_{\xi}(\mathcal{\overline{S}}_{\lambda}(2,2))= \overset{5}{\underset{i=1}{\bigoplus}}\mathcal{O}_{i}$ is semisimple.

\end{ex}

  \begin{table}[ht]
  \label{HyperplNormV}
\begin{center}
\begin{tabular}{ |c|c|c|} 
 \hline
Hyperplane  & $h_i=0\cap V_{\lambda}$ & Normal vector \\ 
  \hline
 $h_1=0$  & $\sum\limits_{i=1}^{2\ell-3}u_i=0$ & $\eta_1=(1,\hdots,1,0)$ \\ 
  \hline
$h_2=0$  & $u_1=0$ & $\eta_2=(-1,0,\hdots,0)$ \\ 
  \hline
$h_3=0$  & $u_2=0$ & $\eta_3=(0,-1,0\hdots,0)$ \\ 
  \hline
  $\hdots$  & $\hdots$ & $\hdots$ \\ 
     \hline
  $h_{2\ell-3}=0$  & $u_{2\ell-4}=0$ & $\eta_{2\ell-3}=(0,\hdots,0,-1,0,0)$ \\ 
    \hline
$h_{2\ell-2}=0$  & $u_{2\ell-3}+u_{2\ell-2}=0$ & $\eta_{2\ell-2}=(0,\hdots,0,-1,-1)$ \\ 
  \hline
 $h_{2\ell-1}=0$ & $u_{2\ell-2}=-a$ & $\eta_{2\ell-1}=(0,\hdots,0,1)$ \\ 
   \hline
  $h_{2\ell}=0$  & $u_{2\ell-2}=a$ & $\eta_{2\ell}=(0,\hdots,0,-1)$ \\ 
  \hline
\end{tabular}
\caption{Collection of hyperplanes and normal vectors}
\end{center}
\end{table}

 \begin{table}[ht]
 \label{Chambers}
\begin{center}
\begin{tabular}{ |c|c|c|} 
 \hline
Sign vector $\gamma \in \mathcal{P}_{\Lambda,\xi}$  & Coordinates of $a_{\gamma}$ in $W$& Hyperplanes $H$, s.t. $a_{\alpha}\not\in H$ \\ 
  \hline
$\alpha_{1}$ & $(\lambda,0,\hdots,0,-2\lambda,0)$ &  $h_1=0, h_{2\ell-1}=0$ \\ 
  \hline
$\beta_{1}$ & $(0,0,\hdots,0,\lambda,0,-2\lambda)$ &  $h_{2\ell-2}=0, h_{2\ell}=0$  \\ 
    \hline
$\hdots$ & $\hdots$ & $\hdots$ \\ 
     \hline
  $\alpha_{mid}$ & $(0,0,\hdots,0,-\lambda,-\lambda)$ & $h_{2\ell-1}=0,h_{2\ell}=0$ \\ 
      \hline
 $\hdots$ & $\hdots$ & $\hdots$ \\ 
    \hline
 $\alpha_{2\ell-1}$  & $(0,\hdots,0, \lambda,-2\lambda,0)$ & $h_{2\ell-2}=0, h_{2\ell-1}=0$ \\ 
  \hline
 $\beta_{2\ell-1}$ & $(\lambda,0,\hdots,0,-2\lambda)$ & $h_{1}=0, h_{2\ell}=0$ \\ 
  \hline
\end{tabular}
\caption{Sign vectors, walls of chambers and points of maximum of $\xi$}
\end{center}
\end{table}

 \begin{center}
 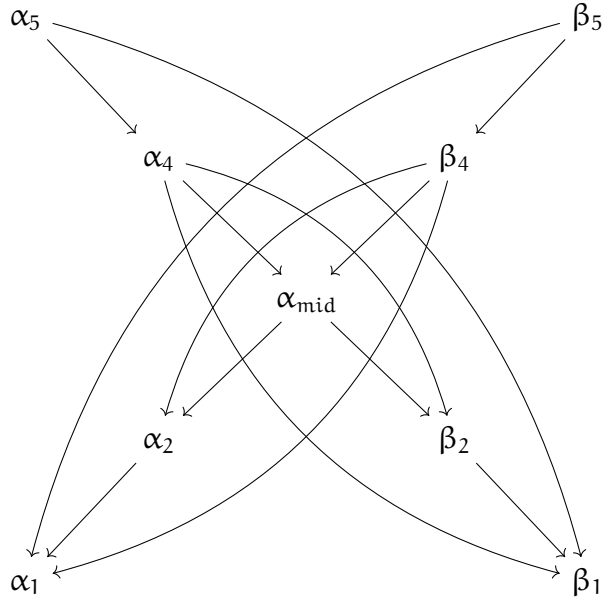
\begin{figure}
\begin{tikzpicture}
\matrix(m)[matrix of math nodes,
row sep=3em, column sep=2.5em,
text height=1.5ex, text depth=0.25ex]
	{ \alpha_5 & & & & \beta_5 \\ & \alpha_4 &  & \beta_4 & \\ &  & \alpha_{mid} & &  \\  & \alpha_2 & & \beta_2 & \\ \alpha_1 & &  & & \beta_1\\};
\path[->,font=\scriptsize]
(m-1-1)  edge   (m-2-2)
(m-1-5)  edge  (m-2-4)
(m-2-2)  edge  (m-3-3)
(m-2-4)  edge  (m-3-3)
(m-3-3)  edge  (m-4-2)
(m-3-3)  edge  (m-4-4)
(m-4-2)  edge  (m-5-1)
(m-4-4)  edge  (m-5-5)
(m-1-1)  edge [bend left]  (m-5-5)
(m-1-5)  edge  [bend right] (m-5-1)
(m-2-2)  edge [bend right]  (m-5-5)
(m-2-4)  edge  [bend left] (m-5-1)
(m-2-2)  edge [bend left]  (m-4-4)
(m-2-4)  edge  [bend right] (m-4-2); 
\end{tikzpicture}
\caption{Homs between standards in $\mathcal{O}_{\nu}(\mathcal{\overline{S}}_{\lambda}(2, 3))$ for $\lambda \in -2+\mathbb{Z}_{<0}\cup \mathbb{Z}_{>0}+1$}
\end{figure}
\end{center}

\begin{center}
	\begin{figure}
		\begin{tikzpicture}
		\matrix(m)[matrix of math nodes,
		row sep=3em, column sep=2.5em,
		text height=1.5ex, text depth=0.25ex]
		{ \alpha_5 & & & & \beta_5 \\ & \alpha_4 &  & \beta_4 & \\ &  & \alpha_{mid} & &  \\  & \alpha_2 & & \beta_2 & \\ \alpha_1 & &  & & \beta_1\\};
		\path[->,font=\scriptsize]
		
		(m-1-1)  edge [bend left]  (m-5-5)
		(m-1-5)  edge  [bend right] (m-5-1)

		(m-2-2)  edge [bend left]  (m-4-4)
		(m-2-4)  edge  [bend right] (m-4-2); 
		\end{tikzpicture}
		\caption{Homs between standards in $\mathcal{O}_{\nu}(\mathcal{\overline{S}}_{\lambda}(2, 3))$ for $\lambda\in-\frac{5}{2}+ \mathbb{Z}_{<0}\cup \mathbb{Z}_{>0}+\frac{3}{2}$}
	\end{figure}
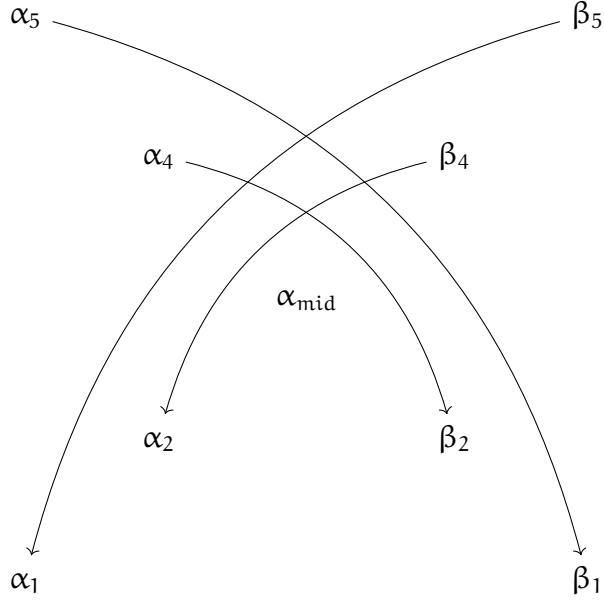
\end{center}

\section{Harish-Chandra bimodules, ideals and localisation theorems} 
Let $\overline{\mathcal{A}}_{\lambda}(n, \ell), \overline{\mathcal{A}}_{\lambda'}(n, \ell)$ be two quantizations of $A=\mathbb{C}[\overline{\mathcal{M}}(n,\ell)]$. An $\overline{\mathcal{A}}_{\lambda}(n, \ell) - \overline{\mathcal{A}}_{\lambda'}(n, \ell)$-bimodule $\mathcal{B}$ is  \textit{Harish-Chandra} (HC) provided there exists a filtration on $\mathcal{B}$, s.t. the induced left and right actions of $A$ on $gr \mathcal{B}$ coincide and $gr \mathcal{B}$ is a finitely generated $A$-module. Such filtrations will be referred to as \textit{good}.
\par
Pick a point $x\in \overline{\mathcal{M}}(2,\ell)$ on a symplectic leaf $\mathcal{L}$. Then for the slice $\mathcal{SL}_x$ at $x$ we have a restriction functor (see Section $5.4$ of \cite{BezrLos}) $$Res_{\dagger,x}:HC(\overline{\mathcal{A}}_{\lambda}(2, \ell) - \overline{\mathcal{A}}_{\lambda'}(2, \ell))\rightarrow HC(\mathcal{SL}_{x,\tilde{\lambda}} -\mathcal{SL}_{x,\tilde{\lambda'}}).$$
\par
This functor is exact. 
\begin{thm}
	If  $\lambda\in (-\infty;1-\ell)\cup (\ell-2;+\infty)$ is not an integer or half-integer, then the algebra $ \overline{\mathcal{A}}_{\lambda}(2, \ell)$ has no proper two-sided ideals. 
\end{thm}
\begin{proof}
	Assume $\mathcal{I}$ is a proper two-sided ideal in $ \overline{\mathcal{A}}_{\lambda}(2, \ell)$. Then pick a point $x$ in an open symplectic leaf in $V(\overline{\mathcal{A}}_{\lambda}(2, \ell)/\mathcal{I})$, so $Res_{\dagger,x}(\mathcal{I})$ is an ideal in the algebra $\mathcal{SL}_{x,\tilde{\lambda}}$. Since for $\lambda$ as in the statement of the theorem there no finite-dimensional representations neither in the category $\mathcal{S}_{\lambda}$-mod nor the category of finitely generated modules over the corresponding quantization of the $2$-dimensional slice, the type $A_1$  Kleinian singularity $\mathbb{C}^2/\mathbb{Z}_2$ (see Remark \ref{A_1SingSlice}), the argument is concluded by contradiction.
\end{proof}
Similarly, we prove the following.
\begin{thm}
	\label{AbLocHolds}
	Abelian localisation holds for $(\lambda,\theta)$ with $\theta<0$ and $\lambda<-\ell$ or $\theta>0$ and $\lambda>\ell-1$.
\end{thm}
\begin{proof}
	The argument is completely anagolous to the one in the proof of Lemma $5.3$ in \cite{Los17_1}, which can be briefly summarized as follows. The abelian localization holds for $\lambda$ if and only if the natural homomorphisms for $m\gg0$ and $\chi=det$
	\begin{equation}
	\label{AbLocCrit}
	\begin{split}
	\overline{\mathcal{A}}_{\lambda+(m+1)\chi,-\chi}(2, \ell)\otimes_{\overline{\mathcal{A}}_{\lambda+(m+1)\chi}}\overline{\mathcal{A}}_{\lambda+m\chi,\chi}(2, \ell)\rightarrow \overline{\mathcal{A}}_{\lambda+m\chi}(2, \ell) \\
	\overline{\mathcal{A}}_{\lambda+m\chi,\chi}(2, \ell)\otimes_{\overline{\mathcal{A}}_{\lambda+m\chi}}\overline{\mathcal{A}}_{\lambda+(m+1)\chi,-\chi}(2, \ell)\rightarrow \overline{\mathcal{A}}_{\lambda+(m+1)\chi}(2, \ell),
	\end{split}
	\end{equation}
	with $\overline{\mathcal{A}}_{\lambda+m\chi,\chi}(2,\ell):=(D(\overline{R})/[D(\overline{R})\{\Phi(x)-(\lambda+m\chi)(x), x \in \mathfrak{g}\}])^{G,\chi}$ the $\overline{\mathcal{A}}_{\lambda+(m+1)\chi}-\overline{\mathcal{A}}_{\chi}$-bimodule, are isomorphisms. Assuming that this is not the case, there must be a nontrivial module $M$ in the kernel or cokernel of the first or the second map. Then the support of $M$ must be $\overline{\mathcal{L}}$, the closure of a symplectic leaf $\mathcal{L}$. Applying the functor $Res_{\dagger,x}$ to \eqref{AbLocCrit} with $x\in \mathcal{L}$, we again get natural homomorphisms. Furthermore, since the order on the leaves is linear and $\mathcal{L}\neq o$ (otherwise $M$ would be of finite dimension, which is impossible due to Corollary \ref{LagrnotIsotr}), we can pick $x$ to be on the $2\ell$-dimensional leaf (the one with number $3$ in Table \ref{SliceTable}).   Since the slice $\mathcal{SL}_{x}$ is the hypertoric variety $\mathcal{SL}_{p}$ and abelian localization holds for the algebra $\mathcal{\overline{S}}_{\lambda}(2,\ell)$  for  $\lambda<-\ell$ (see Remark \ref{AbLocHypT}), the restricted homomorphisms must be isomorphisms. As the module $Res_{\dagger,x}(M)$ is nonzero, we obtain a contradiction.
	
	The assertion for $\theta>0$ and $\lambda>\ell-1$ follows from the isomorphism $\mathcal{\overline{A}}_{\lambda}(n,\ell)\cong\mathcal{\overline{A}}_{-\lambda-1}(n,\ell)$ (see Lemma \ref{Symplectomorphism}).
\end{proof}
\begin{cor}
\label{FinHomDim}
	If  $\lambda\in (-\infty;-\ell)\cup (\ell-1;+\infty)$, then the algebra $ \overline{\mathcal{A}}_{\lambda}(2, \ell)$ has finite homological dimension. 
\end{cor}
\begin{proof}
	 Theorem $1.1$ of \cite{MN2} asserts that the derived localisation holds for $\lambda$ if and only if $\overline{\mathcal{A}}_{\lambda}(2, \ell)$ is of finite homological dimension.
\end{proof}
\section{Structure of the category $\mathcal{O}_{\nu}(\mathcal{\overline{A}}_{\lambda}(2,\ell))$}
\par The main goal of the present section is to present a proof of Theorem \ref{HomsStand} and Corollary \ref{MultiplicityCor}, which provide a complete description of homomorphisms between standard objects in $\mathcal{O}_{\nu}(\mathcal{\overline{A}}_{\lambda}(2,\ell))$ and multiplicities of simple objects in the standard ones. In order to accomplish this goal we make an extensive use of methods and results introduced in \cite{BLPW} and \cite{Los17}. A brief overview of these techniques will be given in Subsections $5.1$ through $5.3$ after which the section will conclude with the proof of Theorem \ref{HomsStand}. 
\subsection{Parabolic induction functor} Let $\rho: X\rightarrow X_0$ be a conical symplectic resolution equipped with a Hamiltonian action of a torus $T$. Following Section $5.5$ of \cite{Los17}, introduce a pre-order $\prec^{\lambda}$ on $Hom(\mathbb{C}^*,T)$, the one-parameter subgroups of $T$, via $\nu'\prec^{\lambda}\nu$, if
\begin{enumerate}
	\item[$\bullet$] $\mathcal{A}_{\lambda}(\mathcal{A}_{\lambda}^{>0,\nu'}+(\mathcal{A}_{\lambda}^{\nu'})^{>0,\nu})=\mathcal{A}_{\lambda}\mathcal{A}_{\lambda}^{>0,\nu}$;
	\item[$\bullet$] the natural action of $\nu'(\mathbb{C}^*)$ on $C_{\nu}(\mathcal{A}_{\lambda})$ is trivial.
\end{enumerate}
\par 
The following result was established in \cite{Los17} (see Lemma $5.8$ therein).
\begin{lem} 
	\label{ParabInd}
Consider two elements $\nu,\nu' \in Hom(\mathbb{C}^*,T)$, s.t. $\nu'\prec^{\lambda}\nu$. Then $C_{\nu}(\mathcal{A}_{\lambda})=C_{\nu}(C_{\nu'}(\mathcal{A}_{\lambda}))$. Furthermore, there is an isomorphism of functors $\triangle_{\nu}=\triangle_{\nu'}\circ \underline{\triangle}$, where $ \triangle_{\nu'}: C_{\nu'}(\mathcal{A}_{\lambda})\operatorname{-mod} \rightarrow \mathcal{A}_{\lambda}\operatorname{-mod}, \underline{\triangle}: C_{\nu}(\mathcal{A}_{\lambda})\operatorname{-mod} \rightarrow C_{\nu'}(\mathcal{A}_{\lambda})\operatorname{-mod}$ and $\triangle_{\nu}$ is the standardization functor given by Definition \ref{StandFunctor}.
\end{lem}
\begin{prop}
\label{ParabIndNu}
Let $X= \overline{\mathcal{M}}^{\theta}(2,\ell)$ and consider the one-parameter subgroups  $\nu=(t^{d_1},t^{d_2},\hdots,t^{d_{\ell}})$ with $d_1\gg d_2>d_3>\hdots>d_{\ell}>0$ and $\tilde{\nu}=(t^d,1,\hdots,1)$ with $d>0$. Then we have $\tilde{\nu}\prec^{\lambda}\nu$.
\end{prop}
\begin{proof}
We are going to find the sufficient condition on $\ell$-tuples of weights $(d_1,d_2,\hdots,d_{\ell})$, so that for the corresponding one-parameter subgroup $\nu=(t^{d_1},t^{d_2},\hdots,t^{d_{\ell}})$	the following containments hold
\begin{equation}
\label{Cond1}
\mathcal{\overline{A}}_{\lambda}(2,\ell)^{>0,\tilde{\nu}}\subseteq \mathcal{\overline{A}}_{\lambda}(2,\ell)^{>0,\nu},
\end{equation}
For verifying the reverse containment it suffices to check that 
\begin{equation}
\label{Cond2}\mathcal{\overline{A}}_{\lambda}(2,\ell)^{>0,\nu}\subseteq \mathcal{\overline{A}}_{\lambda}(2,\ell)^{>0,\tilde{\nu}}+(\mathcal{\overline{A}}_{\lambda}(2,\ell)^{\tilde{\nu}(\mathbb{C}^*)})^{>0,\nu} \mbox{ and }
\end{equation}
\begin{equation}
\label{Cond3}\mathcal{\overline{A}}_{\lambda}(2,\ell)^{\geq 0,\nu}\subseteq \mathcal{\overline{A}}_{\lambda}(2,\ell)^{>0,\tilde{\nu}}+(\mathcal{\overline{A}}_{\lambda}(2,\ell)^{\tilde{\nu}(\mathbb{C}^*)})^{\geq 0,\nu}.
\end{equation}
	
Clearly, for such $\nu$ the equality \begin{equation}
\label{Cond4}
\mathcal{\overline{A}}_{\lambda}(2,\ell)(\mathcal{\overline{A}}_{\lambda}(2,\ell)^{>0,\tilde{\nu}}+(\mathcal{\overline{A}}_{\lambda}(2,\ell)^{{\tilde{\nu}(\mathbb{C}^*)}})^{>0,\nu})=\mathcal{\overline{A}}_{\lambda}(2,\ell)\mathcal{\overline{A}}_{\lambda}(2,\ell)^{>0,\nu}
\end{equation}
 holds. Recall that $C_{\nu}(\mathcal{\overline{A}}_{\lambda}(2,\ell))=\mathcal{\overline{A}}_{\lambda}(2,\ell)^{\geq 0,\nu}/\left(\mathcal{\overline{A}}_{\lambda}(2,\ell)^{\geq 0,\nu}\cap\mathcal{\overline{A}}_{\lambda}(2,\ell)\mathcal{\overline{A}}_{\lambda}(2,\ell)^{>0,\nu}\right)$ and the latter is equal to $\mathcal{\overline{A}}_{\lambda}(2,\ell)^{\geq 0,\nu}/\left(\mathcal{\overline{A}}_{\lambda}(2,\ell)^{\geq 0,\nu}\cap\mathcal{\overline{A}}_{\lambda}(2,\ell)(\mathcal{\overline{A}}_{\lambda}(2,\ell)^{>0,\tilde{\nu}}+(\mathcal{\overline{A}}_{\lambda}(2,\ell)^{{\tilde{\nu}(\mathbb{C}^*)}})^{>0,\nu})\right)$ due to equality \eqref{Cond4}, where the action of $\tilde{\nu}(\mathbb{C}^*)$ is trivial thanks to \eqref{Cond3}.
	
It remains to construct the tuples of numbers $(d_1,\hdots, d_{\ell})$, s.t. the containments \eqref{Cond1}-\eqref{Cond3} hold. The algebra of semiinvariants $\mathbb{C}[\overline{\mathcal{M}}(2,\ell)]^{\geq 0,\tilde{\nu}}=\mbox{gr}(\mathcal{\overline{A}}_{\lambda}(2,\ell)^{\geq 0,\tilde{\nu}})$ is finitely generated (see Lemma $3.1.2$ in \cite{LG}). Thus we can choose finitely many $T$-semiinvariant generators $f_1,\hdots,f_s$ of the ideal $\mathbb{C}[\overline{\mathcal{M}}(2,\ell)]^{> 0,\tilde{\nu}}$  with $f_i\in \mathbb{C}[\overline{\mathcal{M}}(2,\ell)]_{\chi_i}$ a $T$-semiinvariant of weight $\chi_i=(a_1^i,\hdots,a_{\ell}^i)$. Let $\tilde{f}_1,\hdots,\tilde{f}_s$ denote the lifts of the generators to $\mathcal{\overline{A}}_{\lambda}(2,\ell)^{\geq 0,\tilde{\nu}}$. These lifts generate $\mathcal{\overline{A}}_{\lambda}(2,\ell)^{> 0,\tilde{\nu}}$. Fix the collection of numbers $d_2>d_3>\hdots>d_{\ell}>0$, denote $\mathfrak{a}_i:=min_j\{a_i^j\}$ and pick $\nu'=(t^{d'_1},t^{d_2},\hdots,t^{d_{\ell}})$ with 
 \begin{equation}
 \label{CondWeights}
 d'_1>\mbox{max}(d_2,-\sum\limits_{1<i\leq \ell,\mathfrak{a}_i<0}\mathfrak{a}_id_i).
 \end{equation} We see that $\tilde{f}_i$ being in $\mathcal{A}_{\lambda}^{>0,\tilde{\nu}}$ imposes $a_1^i>0$ for all $i \in \{1,\hdots,s\}$, hence, $\tilde{f}_i\in \mathcal{A}_{\lambda}^{>0,\nu}$ due to \eqref{CondWeights}, so the containment \eqref{Cond1} holds for $\nu'$ in place of $\nu$.

Similarly, let $g_1,\hdots,g_k$ be the $T$-semiinvariant generators of the algebra $\mathbb{C}[\overline{\mathcal{M}}(2,\ell)]^{\geq 0,\nu}$ with $g_j\in \mathbb{C}[\overline{\mathcal{M}}(2,\ell)]_{\theta_j}$ a $T$-semiinvariant of weight $\theta_j=(b_1^j,\hdots,b_{\ell}^j)$.  Introduce $\mathfrak{b}_i:=min_j\{b_i^j\}$ and pick $\nu''=(t^{d''_1},t^{d_2},\hdots,t^{d_{\ell}})$ with 
\begin{equation}
\label{CondWeights2}
d''_1>\mbox{max}(d_2,-\sum\limits_{1<i\leq \ell,\mathfrak{b}_i<0}\mathfrak{b}_id_i).
\end{equation}
Notice that due to inequality \eqref{CondWeights2} for all $\tilde{g}_j$ (lifts of $g_j$'s, which generate  $\mathcal{A}_{\lambda}^{\geq 0,\nu}$) $\tilde{g}_j\in\mathcal{\overline{A}}_{\lambda}(2,\ell)^{>0,\nu}$ implies $\tilde{g}_j\in\mathcal{\overline{A}}_{\lambda}(2,\ell)^{>0,\tilde{\nu}}+(\mathcal{\overline{A}}_{\lambda}(2,\ell)^{\tilde{\nu}(\mathbb{C}^*)})^{>0,\nu}$, while $\tilde{g}_j\in\mathcal{\overline{A}}_{\lambda}(2,\ell)^{\geq0,\nu}$ implies $\tilde{g}_j\in\mathcal{\overline{A}}_{\lambda}(2,\ell)^{>0,\tilde{\nu}}+(\mathcal{\overline{A}}_{\lambda}(2,\ell)^{\tilde{\nu}(\mathbb{C}^*)})^{\geq 0,\nu}$ and, therefore, containments \eqref{Cond2} and \eqref{Cond3} hold for $\nu''$ in place of $\nu$. 

Finally, we put $d_1>\mbox{max}(d'_1,d''_1)$ so that the conditions \eqref{Cond1}-\eqref{Cond3} all hold true simultaneously for $\nu=(t^{d_1},t^{d_2},\hdots,t^{d_{\ell}})$.
\end{proof}
\begin{rmk}
\label{Nu0precedsNu}
Consider the pair of one-parameter subgroups  $\nu=(t^{d_1},t^{d_2},\hdots,t^{d_{\ell}})$ with $d_1> d_2>d_3>\hdots>d_{\ell-1}\gg d_{\ell}>0$ and $\nu_0=(t^{d_1},t^{d_2},\hdots,t^{d_{\ell-1}},1)$. Similarly to the argument presented in the proof of Proposition \ref{ParabIndNu}, one shows that the containments \eqref{Cond1}-\eqref{Cond3} hold and hence $\nu_0\prec^{\lambda}\nu$.	
\end{rmk}

\subsection{Restriction functor}
Following \cite{B-Et} and \cite{Los17_1}, we define the restriction functor $Res: \mathcal{O}_{\nu}(\mathcal{\overline{A}}_{\lambda}(2,\ell))\rightarrow \mathcal{O}_{\nu}(\mathcal{\overline{S}}_{\lambda}(2,\ell))$.  Set $\mathcal{\overline{A}}_{\lambda}(2,\ell)^{\Lambda_p}:=\mathbb{C}[\overline{R}//G]^{\Lambda_p}\otimes_{\mathbb{C}[\overline{R}]^G}\mathcal{\overline{A}}_{\lambda}(2,\ell)$ and $\mathcal{\overline{S}}_{\lambda}(2,\ell)^{\Lambda_0}:=\mathbb{C}[\mathbb{C}^{2\ell}//K]^{\Lambda_0}\otimes_{\mathbb{C}[\mathbb{C}^{2\ell}]^K}\mathcal{\overline{S}}_{\lambda}(2,\ell)$, then anagolously to Lemma $6.4$ in \cite{Los17_1} there is a $G$-equivariant isomorphism $\Theta: \mathcal{\overline{A}}_{\lambda}(2,\ell)^{\Lambda_p}\rightarrow \mathcal{\overline{S}}_{\lambda}(2,\ell)^{\Lambda_0}$ of filtered algebras. It is the quantization of the Nakajima isomorphism of formal neighborhoods (see Section $5.4$ of \cite{BezrLos} for details). Let $\nu_0=(t^{d_1},t^{d_2}, \hdots,t^{d_{\ell-1}},1)$ with $d_1>d_2>\hdots>d_{\ell-1}>0$ be a one-parameter subgroup. Consider the category $\mathcal{O}_{\nu}(\mathcal{\overline{S}}_{\lambda}(2,\ell))^{\Lambda_0}$ consisting of all finitely generated $\mathcal{\overline{S}}_{\lambda}(2,\ell)^{\Lambda_0}$-modules such that 
\begin{enumerate}
	\item{}
	$h_0 = d_e\nu_0$ (the differential of $\nu_0$ at $e=(1,\hdots,1)$)  acts locally finitely with eigenvalues bounded from above; 
	
	\item{}
	the generalized $h_0$-eigenspaces are finitely generated over $\mathbb{C}[\mathcal{S}_{p}]^{\Lambda_0}$.
\end{enumerate}
We get an exact functor $$\mathcal{\overline{S}}_{\lambda}(2,\ell)\rightarrow \mathcal{\overline{S}}_{\lambda}(2,\ell)^{\Lambda_0}, N \mapsto \mathbb{C}[\mathbb{C}^{2\ell}//K]^{\Lambda_0}\otimes_{\mathbb{C}[\mathbb{C}^{2\ell}]^K} N.$$
\par
Let  $h$ be the image of $1$ under the quantum comoment map for $t \mapsto \nu(t)\nu_0(t)^{-1}$. For $N \in \mathcal{\overline{S}}_{\lambda}(2,\ell)^{\Lambda_0}\operatorname{-mod}$ denote by $N_{fin}$ the subspace of $h$-finite elements. The statement and proof of the following lemma is analogous to Lemma $6.5$ in \cite{Los17_1}.
\begin{lem}
The functor $\bullet^{\Lambda_0}$ is a category equivalence. A quasi-inverse functor is given by $N\mapsto N_{fin}$.
\end{lem}
\par
 Finally, define $$Res(N):=[\Theta_{*}(\mathbb{C}[\overline{R}//G]^{\Lambda_p}\otimes_{\mathbb{C}[\overline{R}]^G}N)]_{fin}.$$

\par
The following isomorphism of functors will be of crucial importance. It was established in Lemma $6.7$ of \cite{Los17_1}.
\begin{equation}
\label{ResFunctor}
Res\circ\triangle_{\nu_0}\cong \triangle_{\nu_0}\circ\underline{Res},
\end{equation}
where $\underline{Res}$ is the   functor    $ \mathcal{O}_{\nu}(C_{\nu_0}(\mathcal{\overline{A}}_{\lambda}(2,\ell)))\rightarrow  \mathcal{O}_{\nu}(C_{\nu_0}(\mathcal{\overline{S}}_{\lambda}(2,\ell)))$ defined analogously to $Res$.
\begin{rmk}
\label{SliceQuantAlg}
Let $\lambda<1-\ell$. As $\mathcal{SL}_{p}^{\nu_0}$ consists of $4\ell-3$ points and abelian localisation holds for $\lambda$ (Remark \ref{AbLocHypT}), we have $C_{\nu_0}(\mathcal{\overline{S}}_{\lambda}(2,\ell))\cong \mathbb{C}^{4\ell-3}$. The variety of fixed points of $\overline{\mathcal{M}}^{\theta}(2,\ell)^{\nu_0}$ is $T^*\mathbb{CP}^1$ together with the disjoint union of $2(\ell-1)$ copies of $\mathbb{C}^{2}$ (see Remark \ref{SmTFPpp}).  Arguing analogously to the proofs of Theorem \ref{AllFP} and Proposition \ref{RestrParam'} one checks that  $C_{\nu_0}(\mathcal{\overline{A}}_{\lambda}(2,\ell))=\mathcal{\overline{A}}_{\lambda}(2,1)\oplus \mathcal{D}(\mathbb{C}^{2})^{\oplus 2\ell-2}\simeq \mathcal{D}^{\lambda}(\mathbb{CP}^{1})\oplus \mathcal{D}(\mathbb{C}^{2})^{\oplus 2\ell-2}$. 
\end{rmk}
\begin{cor}
	\label{ResIm}
	Let $\lambda\in \mathbb{Z}_{<0}+1-\ell\cup \mathbb{Z}_{>0}+\ell-2$, then the images of standard and simple objects in $ \mathcal{O}_{\nu}(C_{\nu_0}(\mathcal{\overline{A}}_{\lambda}(2,\ell)))$ are given by
 \begin{flalign*}
&Res(\triangle_i)=\triangle_{\alpha_{i}}\oplus \triangle_{\beta_{i}},i>\ell+1\\
&Res(\triangle_i)=\triangle_{\alpha_{i+1}}\oplus \triangle_{\beta_{i+1}}, i<\ell,\\
&Res(\triangle_j)=\triangle_{\alpha_{mid}}, j\in \{\ell,\ell+1\},\\
&Res(S_i)=S_{\alpha_{i}}\oplus S_{\beta_{i}}, i>\ell+1,\\
&Res(S_i)=S_{\alpha_{i+1}}\oplus S_{\beta_{i+1}}, i<\ell,\\
&Res(S_{j})=S_{\alpha_{mid}}, j\in \{\ell,\ell+1\}\\
\end{flalign*}
\par
In case $\lambda \in \mathbb{Z}_{<0}+\frac{1}{2}-\ell\cup \mathbb{Z}_{>0}+\ell-\frac{3}{2}$, the only difference is that $Res(S_{\ell})=0$.
\end{cor}
\begin{proof}  
We show that the analog of the assertion of the corollary holds for $\underline{Res}$ in place of $Res$, then the result is a direct consequence of Lemma \ref{ParabInd} and equality \eqref{ResFunctor} (as $\nu_0\prec^{\lambda}\nu$ due to Remark \ref{Nu0precedsNu}). The standard objects of $\mathcal{O}_{\nu}(C_{\nu_0}(\mathcal{\overline{A}}_{\lambda}(2,\ell)))$ are $\underline{\triangle}(N_s)$, where $N_s$ is the one-dimensional irreducible representation of $C_{\nu}(\mathcal{\overline{A}}_{\lambda}(2,\ell))\simeq \mathbb{C}^{2\ell}$ with the action given by $ (a_1,\hdots,a_{2\ell})\cdot w:=a_sw$ for $ (a_1,\hdots,a_{2\ell})\in \mathbb{C}^{2\ell} \mbox{ and } 0\neq w\in N_s$. First, let us consider $i \not\in \{\ell,\ell+1\}$, then
 \begin{flalign*}
&\underline{Res}(\underline{\triangle}(N_i))=\underline{Res}(C_{\nu_0}(\mathcal{\overline{A}}_{\lambda}(2,\ell))/\mathcal{I}^{>0,\nu}\otimes_{\mathbb{C}^{2\ell}}N_i)=\\
=&\underline{Res}((\mathcal{D}^{\lambda}(\mathbb{CP}^{1})\oplus \mathcal{D}(\mathbb{C}^{2})^{\oplus 2\ell-2})/\mathcal{I}^{>0,\nu}\otimes_{\mathbb{C}^{2\ell}}N_i)=\underline{Res}(\mathcal{D}(\mathbb{C}^{2}_s)/\mathcal{\tilde{I}}^{>0,\nu})=\\
=&\underline{Res}(\mathbb{C}[x_i,y_i])\xrightarrow[\varphi]{\thicksim} M_{\alpha_k}\oplus M_{\beta_k},
\end{flalign*}
where the map $\varphi$ is the evaluation at points $(1,0)$ and $(-1,0)\in \mathbb{C}^{2}_s$, the two points on the $s^{\mbox{th}}$ copy of $\mathbb{C}^2$ which are the $\nu_0(\mathbb{C}^*)$-fixed points with indices $\alpha_k$ and $\beta_k$ on the slice (see Remark \ref{SmTFP} for details). Here $\mathcal{D}(\mathbb{C}^{2}_s)$ stands for the algebra of differential operators on the $s$th copy of $\mathbb{C}^2$, while  $\mathcal{I}^{>0,\nu}:=C_{\nu_0}(\mathcal{\overline{A}}_{\lambda}(2,\ell))C_{\nu_0}(\mathcal{\overline{A}}_{\lambda}(2,\ell))^{>0,\nu},$ $\mathcal{\tilde{I}}^{>0,\nu}=\mathcal{I}^{>0,\nu}\cap \mathcal{D}_i(\mathbb{C}^{2})$ and $M_{\alpha_k}, M_{\beta_k}$ are the one-dimensional irreducibles in $ \mathcal{O}_{\nu}(C_{\nu_0}(\mathcal{\overline{S}}_{\lambda}(2,\ell)))$ with $k$ as given in the statement of the corollary. In case $i \in \{\ell,\ell+1\}$, it is analogous to check that $\underline{Res}(\underline{\triangle}(N_i))=M_{\alpha_{mid}}$. This completes verification of the assertion on the images of standards. 
\par 
Next we verify the assertion on the images of simples. Let $M\in \mathcal{O}_{\nu}(C_{\nu_0}(\mathcal{\overline{A}}_{\lambda}(2,\ell)))$, we write $L_{\nu_0}(M)$ for the maximal quotient of $\triangle_{\nu_0}(M)$
that does not intersect the highest weight subspace. Analogously to Corollary $6.8$ in \cite{Los17_1},
one checks that $Res(L_{\nu_0}(M)) = L_{\nu_0}(\underline{Res}(M))$. Therefore, if  $\lambda \not\in \mathbb{Z}_{<0}+\frac{1}{2}-\ell\cup \mathbb{Z}_{>0}+\ell-\frac{3}{2}$, then each $L_{\nu_0}(\underline{\triangle}(N_i))$ is irreducible, so $L_{\nu}(\underline{\triangle}(N_i))=L_{\nu_0}(\underline{\triangle}(N_i))$ and we have $Res(S_i)=Res(L_{\nu}(\underline{\triangle}(N_i)))=Res(L_{\nu_0}(\underline{\triangle}(N_i)))=L_{\nu_0}(M_{\alpha_k}\oplus M_{\beta_k})=S_{\alpha_k}\oplus S_{\beta_k}$ if $i \not\in \{\ell,\ell+1\}$ and $Res(S_{\ell})=Res(S_{\ell+1})=S_{\alpha_{mid}}$.
\par
Notice, that in case $\lambda \in \mathbb{Z}_{<0}+\frac{1}{2}-\ell\cup \mathbb{Z}_{>0}+\ell-\frac{3}{2}$, we have $\underline{\triangle}(N_{\ell+1})\subset\underline{\triangle}(N_{\ell})$ (see Remark \ref{SliceQuantAlg}) and, hence, $S_{\ell}\neq L_{\nu_0}(\underline{\triangle}(N_{\ell}))$, instead, $S_{\ell}=L_{\nu_0}(\underline{\triangle}(N_{\ell})/\underline{\triangle}(N_{\ell+1}))$, while $\underline{Res}(\underline{\triangle}(N_{\ell})/\underline{\triangle}(N_{\ell+1}))$ is already equal to zero.
\end{proof}
\begin{cor}
\label{SoclePreservation}
The restriction functor $Res$ maps socles of standard objects to socles of their images.
\end{cor}
\begin{proof}
We start by noticing that Corollary \ref{ResIm} implies that $Res$ maps simple objects to semisimple, hence, the containment $Res(Soc(\triangle_{\nu_i}))\subseteq Soc(Res(\triangle_{\nu_i}))$ follows.	
The reverse inclusion is a consequence of part $(c)$, Proposition \ref{SignVectors}. Namely, it provides an explicit description of socles of standards in the target category, i.e.

 \begin{flalign*}
 &Soc(\triangle_{\alpha_k})=S_{\alpha_{k+1}}, Soc(\triangle_{\beta_k})=S_{\beta_{k+1}} \mbox{ for } \ell<k<2\ell-2, \\
  &Soc(\triangle_{\alpha_{mid}})=S_{\alpha_{\ell+1}}\oplus S_{\beta_{\ell+1}},\\
 &Soc(\triangle_{\alpha_k})=S_{\beta_{2\ell-i}}, Soc(\triangle_{\beta_k})=S_{\alpha_{2\ell-i}} \mbox{ for } 1\leq k<\ell.
 \end{flalign*}
\par
Combining the above with the statement of Corollary \ref{ResIm}, allows to conclude
 \begin{flalign*}
&0\subsetneq Res(Soc(\triangle_{k}))\subseteq Res(S_{k+1})=S_{\alpha_{k+1}}\oplus S_{\beta_{k+1}} \mbox{ for } \ell<k\leq 2\ell, \\
&0\subsetneq Res(Soc(\triangle_{\ell}))=Res(Soc(\triangle_{\ell+1}))\subseteq Res(S_{\ell+2})=S_{\alpha_{\ell+1}}\oplus S_{\beta_{\ell+1}},\\
&0\subsetneq Res(Soc(\triangle_k))\subseteq Res(S_{k+1})=S_{\beta_{2\ell-k}}\oplus S_{\alpha_{2\ell-k}} \mbox{ for } 1\leq k<\ell,
\end{flalign*}
so the nonstrict containments in every row must be equalities and the result follows.
\end{proof}
\begin{cor}
	\label{SocOfStand}
	The socles of standards in $\mathcal{O}_{\nu}(\mathcal{\overline{A}}_{\lambda}(2,\ell))$ are as follows:
	\begin{enumerate}
		\item if $\lambda\in \mathbb{Z}_{<0}+1-\ell\cup \mathbb{Z}_{>0}+\ell-2$
		\begin{flalign*}
		&Soc(\triangle_{k})=S_{k+1} \mbox{ for } \ell+1<k<2\ell, \\
		&Soc(\triangle_{\ell+1})=Soc(\triangle_{\ell})=S_{\ell+2},\\
		&Soc(\triangle_{k})=S_{2\ell-k+1}, \mbox{ for } 1\leq k<\ell.
		\end{flalign*}
		\item  if $\lambda \in \mathbb{Z}_{<0}+\frac{1}{2}-\ell\cup \mathbb{Z}_{>0}+\ell-\frac{3}{2}$
		\begin{flalign*}
		&Soc(\triangle_{k})=\triangle_{k} \mbox{ for } \ell<k\leq 2\ell, \\
		&Soc(\triangle_{k})=S_{2\ell-k+1}, \mbox{ for } 1\leq k\leq \ell.
		\end{flalign*}
		\item otherwise, if $\lambda\in (-\infty;1-\ell)\cup (\ell-2;+\infty)$ is neither an integer nor a half-integer
		\begin{flalign*}
		&Soc(\triangle_{k})=\triangle_{k} \mbox{ for } 1\leq k\leq 2\ell. \\
		\end{flalign*}
	\end{enumerate}
	
\end{cor}
\begin{rmk}
	\label{SocInclusion}
	Since $Res$ maps simple objects to semisimple, the containment $Res(Soc(M))\subseteq Soc(Res(M))$ is true for any $M \in \mathcal{O}_{\nu}(C_{\nu_0}(\mathcal{\overline{A}}_{\lambda}(2,\ell)))$. 
\end{rmk}

\begin{prop}
\label{SuppDim}
 Let $\lambda\in \mathbb{Z}_{<0}+1-\ell\cup \mathbb{Z}_{>0}+\ell-2$, then the support of the simple module $S_{1}$  has dimension $2\ell$, supports of simple modules $S_2,\hdots, S_{\ell+1}$ have dimension $4\ell-3$, supports of simple modules $S_{\ell+2},\hdots, S_{2\ell}$ are of dimension $4\ell-2$.

\end{prop}
\begin{proof}
By Theorem $1.2$ in \cite{LosB} all irreducible components of Supp$(S_i)$ have the same dimension (the arithmetic fundamental groups are finite due to the general result of \cite{Nam}). If $Res(S_i)\neq 0$, there exists an irreducible component of Supp$(S_i)$, containing the point $p$ and, therefore, the symplectic leaf through it. Hence, codim Supp$(S_i)$ in $\overline{\mathcal{M}}(2,\ell)$ is equal to codim Supp$(Res(S_i))$ in $\mathcal{SL}_{p}$.  It remains to compute the dimensions of Supp$(Res(S_{\alpha}))$'s.  It follows from Proposition $5.5$ in \cite{BLPW1} that the variety Supp$(S_{\alpha})$ is determined by the sign vector $\alpha$ corresponding to $S_{\alpha}$. Namely, Supp$(S_{\alpha})$ is cut out in $\mathcal{SL}_p$ by the equations $$x_s=0 \mbox{ if } \alpha_i(s)=- \mbox{ and } y_s=0 \mbox{ if }  \alpha(s)=+ \mbox{ for } s \in \{1,\hdots, 2\ell-2\},$$ $$i_k=0 \mbox{ if } \alpha(k)=- \mbox{ and } j_k=0 \mbox{ if }  \alpha(k)=+ \mbox{ for } k \in \{2\ell-1, 2\ell\}.$$
\par
The sign vectors for simple modules were provided in Proposition \ref{SignVectors}.
\par
If $\alpha=\underset{\ell-1}{\underbrace{-\hdots-\overset{a}{\overbrace{+\hdots+} } }} \underset{\ell-1}{\underbrace{-\hdots---}}-- $ the coordinate ring $\mathbb{C}[\mbox{Supp}(S_{\alpha}))]$ is generated by $u_{ij}:=x_iy_j \mbox{ and } v_{js}:=y_jy_s $ for $i\in \{\ell-a,\hdots, \ell-1\}, j \in \{1,\hdots, \ell-a-1\},  s \in \{\ell,\hdots, 2\ell-2\}$ subject to relations: 
  \begin{flalign*}
 &u_{ij}u_{mn}=u_{mj}u_{in}\\
 &u_{ij}v_{kl}=u_{ik}v_{jl}\\
 &v_{ij}v_{kl}=v_{kj}v_{il}. 
 \end{flalign*}
\par
Therefore, dim Supp$(S_{\alpha})=\ell-a-1+\ell-1+a-1=2\ell-3$. The case $\alpha=\underset{\ell-1}{\underbrace{-\hdots--- }} \underset{\ell-1}{\underbrace{\overset{a}{\overbrace{+\hdots+} }-\hdots-}}--$ is completely analogous.
\par
 If $\alpha=\underset{\ell-1}{\underbrace{\overset{a}{\overbrace{+\hdots+} }-\hdots-}}\underset{\ell-1}{\underbrace{-\hdots---}}+-$,  the coordinate ring $\mathbb{C}[\mbox{Supp}(S_{\alpha})]$ is generated by polynomials in $u_{ij}, v_{kl} w_s=i_1y_sj_2$ with $i,j,s, u_{ij} \mbox{ and } v_{js}$ as above. It is direct to check that dim Supp$(S_{\alpha})=2\ell-2$. The case $\alpha=\underset{\ell-1}{\underbrace{-\hdots---}}\underset{\ell-1}{\underbrace{\overset{a}{\overbrace{+\hdots+} }-\hdots-}}-+$ is analogous.
\par
 Finally, if  $\alpha =\underset{\ell-1}{\underbrace{+\hdots+++ }} \underset{\ell-1}{\underbrace{-\hdots---}}-- \mbox{ or } \underset{\ell-1}{\underbrace{-\hdots--- }} \underset{\ell-1}{\underbrace{+\hdots+++}}--$, then the coordinate ring $\mathbb{C}[\mbox{Supp}(S_{\alpha})]=\mathbb{C}[x_1,\hdots,x_{\ell-1},y_{\ell},\hdots,y_{2\ell-2},j_1,j_2]^{\mathbb{C}^*\times \mathbb{C}^*}=$
 \newline $\mathbb{C}[y_1,\hdots,y_{\ell-1},x_{\ell},\hdots,x_{2\ell-2},,j_1,j_2]^{\mathbb{C}^*\times \mathbb{C}^*}=\mathbb{C}$, so, $Supp(S_{\alpha})$ is a point.  
\end{proof}

\begin{ex}
If $\ell=2$, then dim Supp$(S_1)=4$, dim Supp$(S_2)=$dim Supp$(S_3)=5$ and dim Supp$(S_4)=6$.
\end{ex}

\par
 \par

 \par

 

\subsection{Cross-walling functors and $W$-action}
It was checked in Section $5$ of \cite{BLPW} that the natural functor $\iota_{\nu}:D^b( \mathcal{O}_{\nu}(\mathcal{\overline{A}}_{\lambda}(2,\ell)))\hookrightarrow D^b(Coh(\mathcal{\overline{A}}^{\theta}_{\lambda}(2,\ell)))$ is a full embedding.
As shown in  Proposition $8.7$ in \cite{BLPW}, the functor $\iota_{\nu}$ admits both left and right adjoints to be denoted by $\iota^{!}_{\nu}$ and $\iota^*_{\nu}$ respectively.
\begin{defn}
	Let $\nu, \nu'$ be generic one-parameter subgroups. The \textit{cross-walling functor} is given by $$\mathfrak{CW}_{\nu\rightarrow \nu'}:=\iota^{!}_{\nu'}\circ \iota_{\nu}.$$
	The functor $\mathfrak{CW}_{\nu\rightarrow \nu'}$ has a right adjoint $\mathfrak{CW}^*_{\nu\rightarrow \nu'}$ given by $\iota_{\nu}\circ \iota^*_{\nu'}$.
\end{defn}
\par
We need to recall one more concept prior to formulating the  property of cross-walling functors relevant for the purposes of the exposition. Let $\mathcal{C}_1$, $\mathcal{C}_2$ be two highest weight categories. Consider the full subcategories $\mathcal{C}_1^{\triangle}\subset \mathcal{C}_1$ and  $ \mathcal{C}_2^{\nabla}\subset \mathcal{C}_2$ of standardly and costandardly filtered objects. We say that  $\mathcal{C}_2$ is \textit{Ringel dual} to  $\mathcal{C}_1$ if there exists an equivalence  $\mathcal{C}_1^{\triangle} \xrightarrow{\sim} \mathcal{C}_2^{\nabla}$
of exact categories. This equivalence is known to extend to a derived equivalence $\mathcal{R}: D^b(\mathcal{C}_1) \xrightarrow{\sim} D^b(\mathcal{C}_2)$ to be called a \textit{Ringel duality functor}. 
\par
The following result is obtained via a direct apllication of part $2$ of Proposition $7.4$ in \cite{Los17}.
\begin{prop}
	\label{RingelDuality}
The functor $\mathfrak{CW}_{\nu\rightarrow -\nu}[2-3\ell]$ is a Ringel duality functor that maps $\triangle^{\nu}(p)$ to $\nabla^{-\nu}(p)$ for all $p\in \overline{\mathcal{M}}^{\theta}(2,\ell)^T$.
\end{prop}

Let $W = N_G(T)/T \subset Sp_{2\ell}(\mathbb{C})$ be the Weyl group. The action of $W$ on $\mathbb{C}[\overline{\mathcal{M}}^{\theta}(n,\ell)]$ lifts to an action on the quantization $\overline{\mathcal{A}}_{\lambda}(2,\ell)$. This gives rise to the functor $\Phi_w: \mathcal{O}_{\nu'}(\overline{\mathcal{A}}_{\lambda}(2,\ell))\rightarrow\mathcal{O}_{\nu}(\overline{\mathcal{A}}_{\lambda}(2,\ell))$, where $w\cdot \nu = \nu'$ (here we consider the action of $W$ via conjugation, i.e. $w \cdot \nu =w\nu w^{-1}$). The functor $\Phi_w$ maps an object $N$ to itself with the twisted action of $\overline{\mathcal{A}}_{\lambda}(2,\ell)$. More
precisely, $$a \cdot n := (wa)n,$$
with the ordinary action of $\overline{\mathcal{A}}_{\lambda}(2,\ell)$ on the r.h.s.
\par
We conclude with an important result concerning the faithfulness of the functor $Res$.
\begin{prop}
	\label{Faithfulness}
	The restriction of the functor $Res$ to $ \mathcal{O}_{\nu}(\mathcal{\overline{A}}_{\lambda}(2,\ell))^{\triangle}$ is faithful.
\end{prop}
\begin{proof}
The functor $Res$ is exact (see Section $6.2$ of \cite{Los17}). Since it also preserves socles of the objects in $ \mathcal{O}_{\nu}(\mathcal{\overline{A}}_{\lambda}(2,\ell))^{\triangle}$ (see Corollary \ref{SoclePreservation}), it is sufficient that $Res$ does not kill socles of standard objects to conclude that the functor is faithful (the socle of the image of a nontrivial homomorphism in $ \mathcal{O}_{\nu}(\mathcal{\overline{A}}_{\lambda}(2,\ell))^{\triangle}$ is nonzero). As established in Corollary \ref{ResIm} this is the case for $\lambda \not\in\mathbb{Z}+\frac{1}{2}$, since $Res(S_i)\neq 0$ for all $i$.
\par
  In case $\lambda \in\mathbb{Z}+\frac{1}{2}$, we have $Res(S_{\ell})=0$ (here $S_{\ell}$ is the unique simple annihilated by $Res$ as shown in Corollary \ref{ResIm}), however, $S_{\ell}$ does not lie in the socle of any standard object in $ \mathcal{O}_{\nu}(\mathcal{\overline{A}}_{\lambda}(2,\ell))$ (see Corollary \ref{ResIm}). 
\end{proof}
\subsection{Main theorem}

The results obtained above  allow to describe the $\Hom$ spaces between standards in $\mathcal{O}_{\nu}(\mathcal{\overline{A}}_{\lambda}(2,\ell))^{\triangle}$.
\begin{thm}
	\label{HomsStand}
	Let $\lambda\in \mathbb{Z}_{<0}+1-\ell\cup \mathbb{Z}_{>0}+\ell-2$ and $\nu=(t^{d_1},t^{d_2},\hdots,t^{d_{\ell}})$ with $d_1\gg d_2\gg d_3\gg\hdots\gg d_{\ell}>0$. The nontrivial Homs in $\mathcal{O}_{\nu}(\mathcal{\overline{A}}_{\lambda}(2,\ell))^{\triangle}$ are 
	
	\begin{enumerate}
		\item{}
		Hom($\triangle_i,\triangle_{i-1}$), where $i \in \{2,\hdots,2\ell\}, i\neq \ell+1$; 
		
		\item{}
		Hom($\triangle_{\ell+2},\triangle_{\ell}$), Hom($\triangle_{\ell+1},\triangle_{\ell-1}$);
		
		\item{}
		Hom($\triangle_{2\ell-i},\triangle_{i+1}$) with $i \in \{0,\hdots,\ell-2\}$.
	\end{enumerate}
	Let $\lambda\in \mathbb{Z}_{<0}+\frac{1}{2}-\ell\cup \mathbb{Z}_{>0}+\ell-\frac{3}{2}$. The nontrivial Homs between standards in $\mathcal{O}_{\nu}(\mathcal{\overline{A}}_{\lambda}(2,\ell))^{\triangle}$ are Hom($\triangle_{2\ell-i},\triangle_{i+1}$) with $i \in \{0,\hdots,\ell-1\}$. 
	
	
	All the Hom spaces are one-dimensional. 
	\par
	Finally, if $\lambda\in (-\infty;1-\ell)\cup (\ell-2;+\infty)$ is none of the above, the category $\mathcal{O}_{\nu}(\mathcal{\overline{A}}_{\lambda}(2,\ell))$ is semisimple.
\end{thm} 
\begin{proof}
First consider $\lambda\in \mathbb{Z}_{<0}+1-\ell \cup \mathbb{Z}_{>0}+\ell-2$. For convenience of the exposition the proof will be broken down into several steps.
\par
\textit{Step} $1$.	Notice that $Soc(\triangle_{2\ell-1})=Soc(\triangle_{2})=Soc(\triangle_{1})=\triangle_{2\ell}$ (Corollary \ref{SocOfStand}). Hence, Hom$(\triangle_{2\ell},\triangle_{2\ell-1})$, Hom$(\triangle_{2\ell},\triangle_{2})$ and Hom$(\triangle_{2\ell},\triangle_{1})$ do not vanish. 

\par
\textit{Step} $2$. Let $w_0\in W$ be the longest element and consider the functor $\mathcal{F}_{w_0}:=\Phi_{w_0}\circ \mathfrak{CW}_{\nu\rightarrow -\nu}$. Notice that $w_0\cdot\nu=-\nu$ and the order on the $T$-fixed points corresponding to $-\nu$ is in reverse to the one associated with $\nu$. Thus the functor $\mathcal{F}_{w_0}$ is an autoequivalence on $\mathcal{O}_{\nu}(\mathcal{\overline{A}}_{\lambda}(2,\ell))^{\triangle}$ with $\mathcal{F}_{w_0}(\triangle_i)=\triangle_{2\ell-i+1}$ (see Proposition \ref{RingelDuality}). Hence, we see that Hom$(\triangle_{2\ell-1},\triangle_{1})=$ Hom$(\triangle_{2\ell},\triangle_{2})=0$ for $\ell\geq 3$ since Soc$(\triangle_{2})=S_{2\ell-1}$ does not contain $\triangle_{2\ell}$ (see Corollary \ref{SocOfStand}). On the other hand if $\ell=2$, then Soc$(\triangle_{2})=\triangle_4$, so Hom$(\triangle_{3},\triangle_{1})$ does not vanish. Similarly,  one shows that Hom$(\triangle_{2},\triangle_{1})=$ Hom$(\triangle_{2\ell},\triangle_{2\ell-1})$ does not vanish either.

\par
\textit{Step} $3$.
We complete the proof for integral $\lambda$ arguing by induction on the number of loops $\ell$ with $\ell=2$ being the base. Assume the assertion holds for the variety $\overline{\mathcal{M}}^{\theta}(2,\ell)$  and take $\tilde{\nu}=(t^d,1,\hdots,1)$ with $d>0$. Notice that $\overline{\mathcal{M}}^{\theta}(2,\ell)\subset\overline{\mathcal{M}}^{\theta}(2,\ell+1)^{\tilde{\nu}}$ as a component. Since $\tilde{\nu}\prec^{\lambda} \nu$ (see Proposition \ref{ParabIndNu}), Lemma \ref{ParabInd} combined with the assumption that induction hypothesis holds in case of $\ell$ loops assure the existense of required homomorphisms between $\triangle_i$'s with indices $i\in\{2,\hdots,2\ell\}$  in $\mathcal{O}_{\nu}(\mathcal{\overline{A}}_{\lambda}(2,\ell+1))$. The remaining Homs between standard objects (not appearing in \textit{Steps} $1,2$ above) vanish, since so do the Homs between their images in the category  $\mathcal{O}_{\nu}(\mathcal{\overline{S}}_{\lambda}(2,\ell))$ and the functor $Res$ is faithful (see Proposition \ref{Faithfulness}). 
\par 
\textit{Step} $4$. 
In case $\lambda \in \mathbb{Z}_{<0}+\frac{1}{2}-\ell\cup \mathbb{Z}_{>0}+\ell-\frac{3}{2}$ using that Soc($\triangle_{i+1}$) with $i \in \{0,\hdots,\ell-1\}$ is $S_{2\ell-i}$ (see Corollary \ref{SocOfStand}), we establish the nonvanishing of Hom spaces in the statement of the theorem.  
Again the remaining Homs vanish since so do their images in the category  $\mathcal{O}_{\nu}(\mathcal{\overline{S}}_{\lambda}(2,\ell))$ and the functor $Res$ is faithful on standardly filtered objects (see Proposition \ref{Faithfulness}).
\par
\textit{Step} $5$. 
Finally, if $\lambda\in (-\infty;1-\ell)\cup (\ell-2;+\infty)$ is neither an integer nor a half-integer, Corollary \ref{SocOfStand} asserts that all standards $\triangle_i$ in $\mathcal{O}_{\nu}(\mathcal{\overline{A}}_{\lambda}(2,\ell))$ are irreducible. Since for $\lambda$ as above abelian localisation holds (Theorem \ref{AbLocHolds}), the classes of standard and costandard objects in $K_0(\mathcal{O}_{\nu}(\mathcal{\overline{A}}_{\lambda}(2,\ell)))$ coincide (Corollary $6.4$ in \cite{BLPW}), so we have that $\nabla_i$'s are simple as well. In particular, every simple lies in the head of a costandard object. The last condition is equivalent to $\mathcal{O}_{\nu}(\mathcal{\overline{A}}_{\lambda}(2,\ell))$ being semisimple (see Lemma $4.2$ in \cite{Los17_1}).
\end{proof}
 
 \begin{cor}
 	\label{MultiplicityCor}
 	Let $\lambda\in \mathbb{Z}_{<0}+1-\ell\cup \mathbb{Z}_{>0}+\ell-2$. Then
 	\begin{enumerate}
 		\item{}
 		$\triangle_{2\ell}=S_{2\ell}$; 
 		
 		\item{}
 	$\triangle_i$  with $\ell+1<i<2\ell$  has a socle filtration with subquotients $S_i \mbox{ and }S_{i+1}$;
 		
 		\item{}
 		$\triangle_{i}$ with $i \in \{\ell,\ell+1\}$ has a socle filtration with subquotients $S_{i}$ and $S_{\ell+2}$;
 			\item{}
 		$\triangle_{\ell-1}$ has a socle filtration with subquotients $S_{\ell-1}, S_{\ell}, S_{\ell+1}$ and $S_{\ell+2}$;
 			\item{}
 		Finally, $\triangle_{i}$ with $i <\ell-1$ has a socle filtration with subquotients $S_{i}, S_{i+1}$ and $S_{2\ell+1-i}$;
 	\end{enumerate}
 	Let $\lambda\in \mathbb{Z}_{<0}+\frac{1}{2}-\ell\cup \mathbb{Z}_{>0}+\ell-\frac{3}{2}$. 
 	Then
 	\begin{enumerate}
 		\item{}
 		$\triangle_{i}=S_{i}$ for $i>\ell$; 
 		
 		\item{}
 		$\triangle_i$  with $i \leq \ell$  has a socle filtration with subquotients $S_i \mbox{ and }S_{2\ell-i+1}$.
 	\end{enumerate}
 \par
The multiplicity of each subquotient is equal to $1$. 
\end{cor}
\begin{proof}
	We check the assertion for $\ell+1<i<2\ell$, the remaining cases are established analogously. Let $0=M_0\subset M_1\
	\subset\hdots\subset M_j=\triangle_i$ be a socle filtration. Notice, that $M_1=Soc(\triangle_i)$, so $Res(M_1)=Soc(\triangle_{\alpha_i}\oplus\triangle_{\beta_i})=S_{\alpha_{i+1}}\oplus S_{\beta_{i+1}}$. Next, $Res(M_2/M_1)=Res(Soc(\triangle_i/M_1))\subseteq Soc((\triangle_{\alpha_i}\oplus\triangle_{\beta_i})/(S_{\alpha_{i+1}}\oplus S_{\beta_{i+1}})))$  (see Remark \ref{SocInclusion}), but the latter is equal to $Res(S_{i})$ (see (c) of Proposition \ref{SignVectors}), hence, the nonstrict containment above must be an equality, so $j=2$ and $M_2=\triangle_i$, concluding verification of the claim. 
\end{proof}
 \par
            


\begin{center}
\begin{figure}
\label{HomsInO}
\begin{tikzpicture}
\matrix(m)[matrix of math nodes,
row sep=3em, column sep=2.5em,
text height=1.5ex, text depth=0.25ex]
{  \Delta_6 & \Delta_5 &  \Delta_4 &  \Delta_3 & \Delta_2 & \Delta_1 \\};
\path[->,font=\scriptsize]
(m-1-1)  edge[bend right=65]   (m-1-6)
(m-1-2)  edge[bend right=45]   (m-1-4)
(m-1-2)  edge[bend right=55]   (m-1-5)
(m-1-3)  edge[bend right=45]   (m-1-5)
(m-1-1)  edge   (m-1-2)
(m-1-2)  edge  (m-1-3)

(m-1-4)  edge  (m-1-5)
(m-1-5)  edge  (m-1-6); 
\end{tikzpicture}
\caption{Homs between standard objects in $\mathcal{O}_{\nu}(\mathcal{\overline{A}}_{\lambda}(2,\ell))$ for $\ell=3$ and $\lambda \in -2+\mathbb{Z}_{<0}\cup \mathbb{Z}_{>0}+1$}
\end{figure}
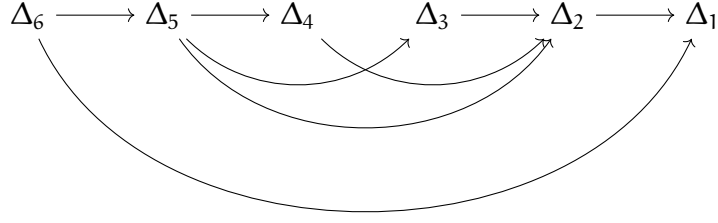
\end{center}

\begin{center}
	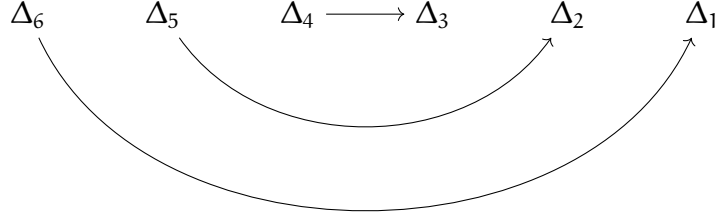
\begin{figure}
		\label{HomsInO}
		\begin{tikzpicture}
		\matrix(m)[matrix of math nodes,
		row sep=3em, column sep=2.5em,
		text height=1.5ex, text depth=0.25ex]
		{  \Delta_6 & \Delta_5 &  \Delta_4 &  \Delta_3 & \Delta_2 & \Delta_1 \\};
		\path[->,font=\scriptsize]
		(m-1-1)  edge[bend right=65]   (m-1-6)
	
		(m-1-2)  edge[bend right=55]   (m-1-5)

		(m-1-3)  edge  (m-1-4); 
		\end{tikzpicture}
		\caption{Homs between standard objects in $\mathcal{O}_{\nu}(\mathcal{\overline{A}}_{\lambda}(2,\ell))$ for $\ell=3$ and $\lambda\in-\frac{5}{2}+ \mathbb{Z}_{<0}\cup \mathbb{Z}_{>0}+\frac{3}{2}$}
	\end{figure}
\end{center}

\begin{table}[ht]
\begin{center}
\begin{tabular}{ |c|c|c|c| } 
 \hline
$\Delta_{\rom{1}}$ & $\Delta_{\rom{2}}$ & $\Delta_{\rom{3}}$ & $\Delta_{\rom{4}}$ \\ 
  \hline
 $S_{\rom{1}}$ & $S_{\rom{2}}$ & $S_{\rom{3}}$ &$S_{\rom{4}}$ \\ 
  \hline
 $S_{\rom{2}}$ & $S_{\rom{4}}$ & $S_{\rom{4}}$  &   \\ 
 \hline
$S_{\rom{3}}$ &   &    &  \\ 
 \hline
 $S_{\rom{4}}$ &   &    &  \\ 
 \hline
\end{tabular}
\caption{Multiplicities of simples in standards in $\mathcal{O}_{\nu}(\mathcal{\overline{A}}_{\lambda}(2,\ell))$ for $\ell=2$ and $\lambda\in \mathbb{Z}_{<0}-1\cup \mathbb{Z}_{>0}$}
\end{center}
\end{table}

\section{Singular parameters} 

In this section we will combine the results of McGerty and Nevins from \cite{MN1} and \cite{MN2} to show that certain quantization parameters $\lambda$ are \textit{singular}, by which we understand that the derived localization does not hold. The following definitions are due. 
\begin{defn}
	\label{CFdef}
	Let $M$ be a $D(\overline{R})$-module equipped with a rational action of $G$. This action gives rise to the map $\mathfrak{g}\rightarrow \mbox{End}(M)$ with $x\mapsto x_M$. Recall that $x_{\overline{R}}$ stands for the image of $x$ under the comoment map $\mathfrak{g}\rightarrow D(\overline{R})$. Then $M$ is said to be a \textit{$(G,\lambda)$-equivariant} $D(\overline{R})$-module provided $x_{M}m=x_{\overline{R}}m-\lambda(x)m$ for all $x \in \mathfrak{g}, m \in M$. The category of finitely generated $(G,\lambda)$-equivariant $D(\overline{R})$ modules will be denoted by $D(\overline{R}) -\mbox{mod}^{G,\lambda}$.
	\par
	We have the functor $\pi_{\lambda}:D(\overline{R}) \operatorname{-mod}^{G,\lambda} \rightarrow \overline{\mathcal{A}}_{\lambda}(n, \ell)\operatorname{-mod}$ of taking $G$ - invariants and the functor $\pi_{\lambda}^{\theta}:D_{\overline{R}} \operatorname{-mod}^{G,\lambda} \rightarrow \overline{\mathcal{A}}_{\lambda}^{\theta}(n, \ell)\operatorname{-mod}$ (the latter category is the category of coherent  $\overline{\mathcal{A}}_{\lambda}^{\theta}(n, \ell)$-modules) defined by first microlocalizing to the $\theta$ - semistable locus and then taking $G$ - invariants. 
\end{defn}
\begin{prop}
	The inclusion ker $\pi^{det}_{\lambda}\subset $ ker $\pi_{\lambda}$, where $\pi_{\lambda}:D_{\overline{R}} \operatorname{-mod}^{G,\lambda}  \twoheadrightarrow \overline{\mathcal{A}}_{\lambda}(n, \ell)\operatorname{-mod}$  and $\pi_{\lambda}^{det}: D_{\overline{R}} \operatorname{-mod}^{G,\lambda}  \twoheadrightarrow \mathcal{A}^{det}_{\lambda}(n,\ell)\operatorname{-mod}$
	holds for $\lambda$, provided $\lambda \notin \frac{\mathbb{Z}_{\leq 0}}{k}+(\ell-1)(n-k)-1, k \in \{1,\hdots, n\}$. We also have ker $\pi^{det^{-1}}_{\lambda}\subset $ ker $\pi_{\lambda}$, whenever $\lambda \notin \frac{\mathbb{Z}_{\geq 0}}{k}+(\ell-1)(n-k), k \in \{1,\hdots, n\}$. Moreover, for $\lambda$ as above the functor of global sections $\Gamma_{\lambda}$ is exact.
\end{prop}
\begin{ex}
	\label{n2MN}
	In case $n=2$, we have ker $\pi^{det}_{\lambda}\subset $ ker $\pi_{\lambda}$ if $\lambda \notin \frac{\mathbb{Z}_{\leq 0}}{2}-1 \cup \mathbb{Z}_{\leq 0}-\ell$ and ker $\pi^{det^{-1}}_{\lambda}\subset $ ker $\pi_{\lambda}$, if $\lambda \notin \frac{\mathbb{Z}_{\geq 0}}{2}\cup \mathbb{Z}_{\geq 0}+\ell-1$.
\end{ex}
\begin{proof}
	First we recall the main results of \cite{MN1}. Let $X$ be a smooth, connected quasiprojective complex variety with an action of a connected reductive group $G$ and $\lambda: G\rightarrow \mathbb{C}^*$ be a character. Assume, in addition, $X$ is affine, the moment map $\mu:T^*X\rightarrow \mathfrak{g}^*$ is flat and the GIT quotient $\mu^{-1}(0)//_{\chi}G$ is smooth. The group $G$ is equipped with a finite set of one-parameter subgroups of a fixed maximal torus $T \subset G$, depending on $X$ and $\lambda$. These subgroups are known as the Kirwan-Ness one-parameter subgroups. Suppose that for each Kirwan-Ness subgroup $\beta$
	$$ \lambda(\beta) \in shift(\beta) + I(\beta) \subseteq shift(\beta) + \mathbb{Z}_{\geq 0},$$
	where $shift(\beta)$ is a numerical shift and $I(\beta) \subseteq \mathbb{Z}_{\geq 0}$.
	Then any $\lambda$-twisted, $G$-equivariant $D$-module with unstable singular support is in the kernel of quantum Hamiltonian reduction and the functor of global sections $\Gamma_{\lambda}$ is exact.
	\par
	Now we provide the proof of the second assertion (for $\theta=det^{-1}$), the statement for $\theta=det$ can be either shown analogously or derived from the isomorphism $\overline{\mathcal{A}}^{\theta}_{\lambda}(n, \ell) \cong\overline{\mathcal{A}}^{-\theta}_{-\lambda-1}(n, \ell)$ (see  Lemma \ref{Symplectomorphism}). 
	\par
	The computation is very similar to the one in Section $8$ of  \cite{MN1}, so we retain the notations.  The multiplicity of each weight $e_{i}-e_{j}$ is $\ell$ and the weights $e_{i}$ get substituted by $-e_{i}$ (alternatively, to avoid this substitution, one can use partial Fourier transform, 'swapping' $V^{*}$ with $V$, see \cite{MN1} for the details). The Kempf-Ness subgroups $\beta_{k}$ correspond to the weights $-\sum\limits_{i=1}^{k} e_{i}, k \in\{1, \hdots, n\}$.
	The shift (in \textit{loc. cit.}) becomes $(\ell-1)k(n-k)+\frac{k}{2}$ and $I(\beta)=\mathbb{Z}_{\geq 0}$.  Therefore, we need $$(-\lambda-\rho)\cdot \beta_{k}\notin \mathbb{Z}_{\geq 0}+(\ell-1)k(n-k)+\frac{k}{2}$$ $$\frac{k}{2}+k\lambda \notin \mathbb{Z}_{\geq 0}+(\ell-1)k(n-k)+\frac{k}{2}$$ $$\lambda \notin \frac{\mathbb{Z}_{\geq 0}}{k}+(\ell-1)(n-k), k \in \{1,\hdots, n\},$$
	where $\rho=\frac{1}{2}\sum\limits_{i=1}^{n} e_{i}$. For $\lambda$ as above, the functor of global sections $\Gamma_{\lambda}: \mathcal{A}^{\theta}_{\lambda}(n, \ell)\rightarrow \mathcal{A}_{\lambda}(n, \ell)$ is exact (see \cite{MN1}) and the inclusion ker $\pi^{\theta}_{\lambda}\subset $ ker $\pi_{\lambda}$ holds.
\end{proof}

\begin{thm}
	\label{SingParam}
	The algebra $\overline{\mathcal{A}}_{\lambda}(2, \ell)$ is not of finite homological dimension for $\lambda \in (-\ell;\ell-1)\cap \mathbb{Z} \mbox{ or } \lambda=-\frac{1}{2}$, i.e. such $\lambda$ are singular.
\end{thm}
\begin{proof}	
	The argument is completely anagolous to the one of a similar statement for Gieseker schemes in \cite{Los17_1} (see Corollary $5.2$). We give a brief outline. The statement is verified by contradiction. Assume $\overline{\mathcal{A}}_{\lambda}(2, \ell)$ is of finite homological dimension with $\lambda$ as in the statement of the theorem. Then the  main result (Theorem $1.1$) of \cite{MN2} implies that the derived localisation functor $D^{b}(\mathcal{\overline{A}}_{\lambda}\operatorname{-mod})  \rightarrow  D^{b}(\mathcal{\overline{A}}_{\lambda}^{\theta}\operatorname{-mod})$ is an equivalence, restricting to an equivalence $D^{b}(\mathcal{O}_{\nu}(\mathcal{\overline{A}}^{\theta}_{\lambda}(2,\ell)))\rightarrow D^{b}(\mathcal{O}_{\nu}(\mathcal{\overline{A}}_{\lambda}(2,\ell)))$. Since for our choice of $\lambda$ the functor $\Gamma_{\lambda}$ is exact (see Example \ref{n2MN}), the abelian equivalence holds for $\lambda$. From this one can conclude that the long wall-crossing  functor $\mathfrak{WC}_{-\theta\leftarrow \theta}$ induces an abelian equivalence $\mathcal{O}_{\nu}(\mathcal{\overline{A}}_{\lambda'}(2,\ell))\rightarrow \mathcal{O}_{\nu}(\mathcal{\overline{A}}_{\lambda''}(2,\ell))$ (here $\lambda'=\lambda+s$ with $s \in \mathbb{Z}_{>0}$ a sufficiently large integer, so that the category $\mathcal{O}_{\nu}(\mathcal{\overline{A}}_{\lambda'}(2,\ell))$ is a highest weight category and $\lambda''-\lambda' \in \mathbb{Z}$). Since $\mathfrak{WC}_{-\theta\leftarrow \theta}$ is also a Ringel duality and for our choice of $\lambda$ the category $\mathcal{O}_{\nu}(\mathcal{\overline{A}}_{\lambda'}(2,\ell))$ is not semisimple (see Theorem \ref{HomsStand}), we obtain a contradiction with Lemma $4.2$ in \cite{Los17_1}, asserting that a highest weight category $\mathcal{C}$, where the classes of standard and costandard objects coincide  is semisimple if and only if for any  Ringel duality $\mathcal{R}:D^b(\mathcal{C})\rightarrow D^b(\mathcal{C^{\vee}})$, we have $H_0(\mathcal{R(S)})\neq 0$ for any simple object $S\in \mathcal{C}$.
\end{proof}
\begin{prop}
	\label{ablocrmk}
	Arguing completely analogously to the proof of Theorem \ref{AbLocHolds}, one shows that abelian localisation holds for $\theta<0$ and $\lambda \in (-\ell;\ell-1), \lambda\not\in \mathbb{Z} \mbox{ and } \lambda\neq-\frac{1}{2}$.
\end{prop}
\begin{proof}  
	We notice that if  $\lambda\not\in \mathbb{Z}$ there are no finite dimensional irreducibles neither in the category $\mathcal{S}_{\lambda}$-mod nor the category of finitely generated modules over the corresponding quantization of the $2$-dimensional slice, the type $A_1$  Kleinian singularity $\mathbb{C}^2/\mathbb{Z}/2\mathbb{Z}$ (see Remark \ref{A_1SingSlice}). Since the aforementioned varieties expose the list  of slices (Table \ref{SliceTable}) we conclude that there are no finite-dimensional irreducibles over the quantization $\mathcal{SL}_{x,\tilde{\lambda}}$ for any slice $\mathcal{SL}_x$.
	
	On the other hand, if the equivalence does not hold, there exists a nontrivial bimodule $M$ in the kernel or cokernel of one of the maps in \eqref{AbLocCrit} (see the proof of Theorem \ref{AbLocHolds}) and a point $x\in \overline{\mathcal{M}}(2,\ell)$, s.t. $Res_{\dagger,x}(M)\neq 0$ is finite dimensional. Hence, we come up with a contradicion.
\end{proof}
Combining  Theorems \ref{AbLocHolds} and \ref{SingParam} with Proposition \ref{ablocrmk}, we  establish the abelian localisation theorem. 
\begin{thm}
\label{AbLocThm}
The abelian localisation holds  for $\lambda \not\in (-\ell;\ell-1)\cap \mathbb{Z} \mbox{ and } \lambda\neq-\frac{1}{2}$.
\end{thm}


\end{document}